\documentclass[letterpaper, 11pt,  leqno]{amsart}
\usepackage[margin=1.2in,marginparwidth=1.5cm, marginparsep=0.5cm]{geometry}

\usepackage{amsmath,amssymb,amscd,amsthm,amsxtra, esint, xcolor,mathtools}

\usepackage{todonotes}

\usepackage[implicit=true]{hyperref}

\allowdisplaybreaks[2]

\usepackage{comment}

\usepackage{color}

\newtheorem{theorem}{Theorem} [section]

\newtheorem{lemma}[theorem]{Lemma}
\newtheorem{proposition}[theorem]{Proposition}
\newtheorem{remark}[theorem]{Remark}
\newtheorem{condition}[theorem]{Condition}
\newtheorem{example}[theorem]{Example}

\newtheorem{definition}[theorem]{Definition}
\newtheorem{corollary}[theorem]{Corollary}

\newcommand{\noi}{\noindent}

\newcommand{\R}{\mathbb{R}}
\newcommand{\C}{\mathbb{C}}
\newcommand{\T}{\mathbb{T}}

\newcommand{\bul}{\bullet}

 \newcommand{\bfB}{\mathbf{B}}

\newcommand{\E}{\mathbb{E}}

\newcommand{\NN}{\mathcal{N}}

\newcommand{\cC}{\mathcal{C}}

\newcommand{\fH}{\mathfrak{H}}

\newcommand{\Cov}{\textup{Cov}}

\newcommand{\F}{\mathcal{F}}

\newcommand{\al}{\alpha}
\newcommand{\be}{\beta}
\newcommand{\dl}{\delta}

\newcommand{\eps}{\varepsilon}

\newcommand{\g}{\gamma}

\newcommand{\s}{\sigma}

\renewcommand{\o}{\omega}
\renewcommand{\O}{\Omega}

\newcommand{\les}{\lesssim}
\newcommand{\ges}{\gtrsim}

\newcommand{\jb}[1]
{\langle #1 \rangle}

\newcommand{\ind}{\mathbf 1}

\newcommand{\PP}{\mathbb{P}}
\newcommand{\M}{\mathcal{M}}
\def\e{\varepsilon}

\newcommand{\N}{\mathbb{N}}

\renewcommand{\H}{\mathcal{H}}

\newcommand{\ASCLT}{\textup{\textsf{ASCLT}}}

\newcommand{\Vol}{\textup{Vol}}

\newcommand{\Var}{\textup{Var}}

\numberwithin{equation}{section}
\numberwithin{theorem}{section}

\makeatletter
\@namedef{subjclassname@2020}{\textup{2020} 
Mathematics Subject Classification}
\makeatother

\begin{document}
\baselineskip = 14pt

\title[ASCLTs via chaos expansion and related results]
{Almost sure Central limit theorems \\
via chaos expansions  and related results}

 \author[L. Maini, M. Rossi,  and 
 G. Zheng]{Leonardo Maini, Maurizia Rossi, 
 Guangqu Zheng}

\address{
}

\email{}

\address{
Leonardo Maini,
Dipartimento di Matematica\\
Universit\` di Roma Tor Vergata\\
Via della Ricerca Scientifica, 1, 00133 Roma, Italy
}

\email{maini@mat.uniroma2.it}

\address{
Maurizia Rossi,
Dipartimento di Matematica e Applicazioni\\
Universit\`a di Milano-Bicocca\\
55, Via Roberto Cozzi, 20125 Milano (MI), Italy
}

\email{maurizia.rossi@unimib.it}

\address{
Guangqu Zheng, 
 Department of Mathematics and Statistics\\
Boston University\\
665, Commonwealth Avenue,
Boston MA 02215 
 }

\email{gzheng90@bu.edu}



\begin{abstract}

In this work, we investigate the asymptotic behavior of
integral functionals of stationary Gaussian random fields
as the integration domain tends to be the whole space.
More precisely, using the Wiener chaos expansion
and Malliavin-Stein method, 
we establish  an {\it almost sure central limit theorem} (ASCLT) 
only under mild conditions
on the covariance function of the underlying stationary
Gaussian fields.  In this setting,
we additionally derive a {\it quantitative central limit theorem} 
with rate of convergence in  quadratic Wasserstein distance, 
and show certain regularity property 
for the said integral functionals.  
In particular, we  solve an open question on the 
{\it Malliavin differentiability of 
the excursion volume of Berry's 
random wave model}. 
As a key consequence of our analysis, 
we obtain the exact asymptotic rate (as a function of the exponent $q$)
for the $q$-th  moment of Bessel functions, thus confirming   
a conjecture based on existing numerical simulations.
In the end, we provide two applications of our result:  
(i) ASCLT  in the context of
Breuer-Major  central limit theorems,  
(ii) ASCLT for Berry's random wave model.
Our approach does not require  
 any knowledge on the regularity properties of random variables 
(e.g., Malliavin differentiability) 
and hence not only complements the existing literature,
but also leads to novel results that are of independent interest.
\end{abstract}

\date{\today}
\maketitle

\tableofcontents

\baselineskip = 14pt

\vspace{-8mm}

\section{Introduction}

Let $\mathbf{B}= \{ B_x \}_{x\in \R^d}$ 
be a real-valued,  centered, and stationary Gaussian random field 
indexed by $\mathbb R^d$ with $d\geq 1$,
defined on some probability space $(\O, \F,  \PP)$.
The covariance function of $\mathbf{B}$ is given by

\noi
\begin{align}\label{cov_B}
    \cC(x) := \mathbb E[B_x B_0], \quad x\in \R^d.
\end{align}
Throughout this paper, we assume  that 
\begin{align}\label{C00}
\cC(0)=1
\quad
{\rm and}
\quad
\text{$\bfB$ is almost surely continuous on $\R^d$.}
\end{align}

In this paper, we study the following probabilistic object:

\noi
\begin{align}\label{def_Y}
Y_t=Y_t(\varphi) := \int_{t D} \varphi(B_x)  \,dx,
\quad t\geq 1,
\end{align}

\noi
where $D$ is  a compact subset  of $\R^d$ with 
nonempty interior
and   $\varphi\in L^2(\R, \frac{1}{\sqrt{2\pi}}e^{-x^2/2}dx)$.\footnote{Due to the stationarity of 
the underlying Gaussian field $\bfB$,
 the law of the Gaussian functional $Y_t$ in \eqref{def_Y} remains
 unchanged if we replace $tD$ by any of its translation $t(D-a)=\{ t(x-a): x\in D\}$. 
 Then, without losing any generality,
 one can always assume that $D$ contains an open ball centered at zero.
When $D$ is a ball, we simply assume that $D = \{ |x|\leq 1\}$.
  \label{ftone}}
For example, when  $\varphi(r) =\ind_{\{ r \geq u\}}$ with $u\in \mathbb R$,
the random variable  $Y_t$ is the volume of the upper level set 
$\{ \bfB \ge u\}\subseteq \R^d$ restricted to 
$tD=\{ ty: y \in D\}$.

Motivated by the recent study on the geometry 
of stationary Gaussian field, particularly on Berry's random wave model, 
we aim at establishing the almost sure central limit theorem ($\ASCLT$) 
for the  family $Y=\{Y_t\}_{t\geq 1}$ after proper normalization
under only mild assumptions on the covariance function $\mathcal C$
(i.e., decay at infinity and certain local structure near the origin).
 Additionally, we will find along the way 
a quantitative central limit theorem and 
show certain Malliavin differentiability for $Y$ (see Remark \ref{rem_malliavin}).
In particular, we solve an open question on the 
 Malliavin differentiability 
 of the excursion volume of Berry's random wave model (see Corollary \ref{corMalliavinBerry}).
Note that our method provides quantitative and almost sure central limit theorems 
even in cases where $Y$ is only minimally regular, 
while in existing literature higher order of Malliavin differentiability is usually required.
 Moreover, as a key consequence of our analysis, 
 we obtain the asymptotic exact rate (as a function of the exponent $q$) 
 of the $q$-th moment of Bessel functions 
 (and more generally of any covariance function
 with a certain local structure at the origin), 
 thus confirming a conjecture based on existing numerical simulations.

In what follows, we will give a brief overview on central limit 
theorems for the integral functional $Y_t$, notably we will recall 
the results around the celebrated Breuer-Major theorem. 
Then, we will introduce Berry's random wave model in 
Section \ref{S1_2}, and state our main results 
in Section \ref{S1_3}. 

\medskip
Let us first fix some notations.

\medskip

\noi
$\bul$ {\bf Notations.}  Given two functions $f(t), g(t)$, 
we write $f(t) \les g(t)$
if $\limsup_{t\to\infty}  \frac{f(t)}{g(t)} \in (0,  \infty)$.
Similarly, we write $f(t) \ges g(t)$ if $g(t) \les f(t)$;
we write $f(t) \asymp g(t)$ if $f(t) \les g(t)$
and $f(t) \ges g(t)$. We write $f(t) \sim g(t)$ if 
$\lim_{t\to\infty}  \frac{f(t)}{g(t)}  =1$.
In this paper, we let $\N = \{1, 2, ... \}$ be the set of positive
integers
and write $L^p(\R^d) = L^p(\R^d, dx)$ for the usual $L^p$ Lebesgue space.

\subsection{Breuer-Major type CLTs}
\label{S1_1}

Starting from the well-known fact that the Hermite polynomials

\noi
\begin{align} \notag 
 \Big\{ H_q(x) 
 = (-1)^q e^{\frac{x^2}{2}}  \frac{d^q}{dx^q}  e^{-\frac{x^2}{2}} :
 q\in\N\cup\{0\} \Big\}
\end{align}

\noi
are orthogonal polynomials 
with respect to the standard Gaussian measure on $\R$,
we have the following (Hermite) expansion  
in $L^2(\mathbb R, \frac{1}{\sqrt{2\pi}}e^{-x^2/2}dx)$: 
with $Z\sim\NN(0,1)$,

\noi
\begin{align}  \label{HerExp}
\varphi =  \E[ \varphi(Z)] +  \sum_{q=R}^\infty  a_q  H_q ,
\end{align}

\noi
where $a_q =a_q(\varphi) =\mathbb E[\varphi(Z)H_q(Z)]/q!$ 
(see also \eqref{HerEXP0}) 
and $a_R$ is the {\it first nonzero}
coefficient in this expansion, or equivalently,

\noi
\begin{align}    \notag   
R=\inf \{  q\geq 1 : a_q\neq 0 \}
\end{align}

\noi
with the convention $\inf\emptyset = +\infty$, 
is  called  the {\bf Hermite rank} of $\varphi$. 
With the above Hermite expansion \eqref{HerExp}, 
one can write 
in $L^2(\Omega)$
\begin{align} \label{CP0}
Y_t = \E[ Y_t] + \sum_{q\geq R} a_q \int_{tD} H_q(B_x) dx.
\end{align}

Next, using the orthogonality relation \eqref{HerP2} between
different Hermite polynomials together with the stationarity of $\bfB$, 
one can derive easily that with $Y_t$ as in \eqref{CP0} and $Z\sim  \NN(0,1)$,

\noi
\begin{align}   \notag   
  \E[Y_t] =  t^d  \, \Vol(D) \mathbb E[\varphi(Z)],
\end{align}

\noi
where $\Vol(D)$ denotes the {\it volume} of $D\subseteq\R^d$; on the other hand, 
by Fubini  and dominated convergence, 
we have

\noi
\begin{align}  \label{Var_Y}
\begin{aligned}
 \Var(Y_t) 
 &= \Big\|    \int_{t D} \big(  \varphi(B_x) -  \E [ \varphi(B_x)]  \big) \,dx \Big\|^2_{L^2(\O)} \\
 &=  \int_{ (t D)^2} \Cov  \big(  \varphi(B_x) ,   \varphi(B_y)     \big) dxdy \\
 &= \sum_{q=R}^\infty a_q^2 q!    \int_{ (t D)^2} \cC^q(x-y)\, dxdy,
 \end{aligned}
\end{align}

\noi
 where the last step follows from the fact that
$\E[ H_p(B_x ) H_q(B_y) ] = q! \cC^q(x-y) \ind_{\{ p=q\}}$; 
see also \eqref{HerP2} and \eqref{e_x2}.
In view of  the expression  \eqref{Var_Y}, 
the asymptotic behavior of the variance  
depends on the  integrability   of 
the covariance function $\cC$ and the coefficients $a_q$'s,
which further influence  the  fluctuation  of $Y_t$ (after proper normalization)
as  $t\to \infty$.

The problem of finding the exact fluctuation of 
the above integral functionals
has received  a great interest in past years 
since  the   work \cite{Mar76,  Taq75, Taq77, Taq79, DM79, Ros81, BM83}
by Maruyama, Dobrushin, Taqqu, Rosenblatt,  Breuer, 
and Major before 90s.
See also recent work \cite{LMNP24, NN20, CNN20, NZ20a, NZ20osc,
NNP21, NZ21,  MN23} that are akin to the Malliavin-Stein method
\cite{NP12}.

In the following, let us recall a few results around Breuer-Major's central
limit theorems (CLTs)  that are closely related to our results.

\begin{theorem}\label{recap}
Let $\varphi$ be as in \eqref{HerExp} with Hermite rank $R\geq 1$.
Let $\cC$ be the covariance function as in \eqref{C00}. 
Recall the definition \eqref{def_Y} of $Y_t$. 
Then, the following statements hold. 
 
\smallskip
\noi
{\rm (i) [Breuer-Major's theorem]} Assume  $\cC\in L^R( \R^d)$.  Then,
\begin{equation}\label{varBM}
 \frac{\Var(Y_t)}{t^d}
  \to \Vol(D)\sum_{q=R}^{+\infty} a_q^2\,q! \int_{\R^d}  \cC^q(z) \,dz
  =: \s^2 \in [0, \infty),
\end{equation}

\noi
and $\dfrac{Y_t - \E[Y_t]}{t^{d/2} }$ converges in law to $\NN(0, \s^2)$ as $t\to+\infty$.
See, e.g.,  \cite{BM83, CNN20}.

\smallskip
\noi
{\rm (ii)} Assume   $\cC \in L^M(\R^d) \setminus   L^R(\mathbb R^d)$ 
for some integer $M \geq R+1$
and  $ \cC^R(x)\geq 0$ for $|x|\geq x_0$
for some   $x_0 > 0$.
Then, the $R$-th chaotic component $\mathbf{J}_{t, R}:=a_R\int_{tD} H_R(B_x)dx $
in \eqref{CP0} is dominant:

\noi
\begin{align} \label{twoV}
\frac{\s_{t, R}^2}{\s_t^2}: =\Var\bigg( a_R\int_{tD} H_R(B_x)dx \bigg) \frac{1}{ \Var(Y_t)} \xrightarrow{t\to+\infty}1,
\end{align}

\noi
i.e., $\s_{t,R}^2\sim \s_t^2$.
Moreover, the following equivalence holds:
\begin{equation}  \notag 
\frac{\mathbf{J}_{t, R}}{\s_{t, R}} = \frac{a_R}{\s_{t, R}} \int_{tD} H_R(B_x)dx     \xrightarrow[\rm law]{t\to\infty}  \NN(0,1) 
   \quad\text{if and only if}
   \quad
      \frac{ Y_t - \E[Y_t]}{\s_t}  \xrightarrow[\rm law]{t\to\infty}  \NN(0,1). 
\end{equation}

\end{theorem}
 
 \begin{remark}\rm 
 	Theorem \ref{recap}-{\rm (ii)} is proved, e.g., in \cite[Proposition 2.2]{MN23}. In fact, even if
    $R$ is assumed to be even in \cite[Proposition 2.2]{MN23}, 
    its proof also works under the more general assumption 
    that ``$\cC^R(x)\geq 0$ for $|x|$ large enough". 
 \end{remark}

The original proof of Breuer-Major's theorem
is done by the method of moments, and a modern treatment
using the chaotic central limit theorem
 is given in, e.g., \cite[Chapter 7]{NP12}.
 The chaotic CLT was first developed in \cite{HN05},
 as a   consequence of the fourth moment theorems 
 \cite{NP05, PT05}. 
 Roughly speaking, one can break the proof of Theorem \ref{recap}-(i) into that on each chaos,
 meaning that due to \eqref{varBM}, the tail 
 $t^{-d/2}\sum_{q\geq k} a_q\int_{tD} H_q(B_x)dx$ can be uniformly (in $t$)
 controlled in $L^2(\Omega)$ as $k\to+\infty$,
 then the CLT in Theorem \ref{recap}-(i) follows from that for any  finite expansion,
 which will follow from a multivariate CLT
 \[
t^{-d/2} (\mathbf{J}_{t, R}, ... , \mathbf{J}_{t, k} ) \xrightarrow[\rm law]{t\to\infty}
 (Z_R, ... , Z_k)
 \]
for any $k > R$, where $ (Z_R, ... , Z_k)$ is a centered Gaussian vector 
with  independent entries. By the fourth moment theorems of 
Nualart, Peccati, and Tudor \cite{NP05, PT05} 
(see also Theorem \ref{FMT_NP}),
the proof of Theorem \ref{recap}-(i)  is then reduced to verifying the CLT for each  
component $\mathbf{J}_{t, k}$.

Note that the case (i) is in the short-range dependent setting, meaning
that the covariance function of 
$\varphi(\bfB)=\{ \varphi(B_x): x\in\R^d\}$ is globally integrable, 
while in the case (ii), the first nontrivial chaotic component 
(i.e., $\mathbf{J}_{t, R}$) is dominant, 
and the validity of the CLT for $Y_t$ is equivalent 
to that of $\mathbf{J}_{t, R}$. In this case, 
$Y_t$ does not always admit Gaussian fluctuations (regardless
of the parity of $R$). 
For example, 
when $\cC(x)= |x|^{-\be} L(|x|)$ with $\beta < d/R$ and 
$L: \R_+ \to  \R$ 
a slowly varying
function (see Section \ref{S2}), 
it is well known that 
$\sigma^2_t \sim \sigma^2_{t,R}$,  
and   $\mathbf{J}_{t, R}$ (and thus $Y_t$) does not admit
Gaussian fluctuation 
as soon as $R\geq 2$;  see \cite{DM79} and Lemma  \ref{tech5}.
We also refer interested readers to \cite{LMNP24,Mai23,LMthesis,MN23}
that elaborate more on the case (ii).
 
The above discussion is only about the central limit theorems at
the qualitative level. In the following, 
we will briefly review relevant quantitative central limit theorems
that often come along the application of Malliavin-Stein's method. 
  Let us first introduce frequently used distances (total variation distance
  and Wasserstein distance of oder $p\geq 1$):

  \noi
 \begin{align} \label{def_dist}
  \begin{aligned}  
 d_{\rm TV}( Y, Z) &= \frac{1}{2}\sup_{h\in\F_{\rm TV}} | \E[  h(Y)]  - \E[ h(Z) ] | \\
 W_p( Y, Z) &= \inf  \| Y_1 - Z_1\|_{L^p(\Omega)}, 
  \end{aligned}
  \end{align}
  
  \noi
where $\F_{\rm TV}$ denotes
 the set of bounded measurable functions $h:\R\to [-1,1]$,
and   the above infimum runs over all joint laws of $(Y_1, Z_1)$
with $Y_1 = Y$ in law and $Z_1 = Z$ in law.
For $p=1$, we have 
 $W_1( Y, Z) = \sup_{h\in\F_{\rm W}} | \E[  h(Y)]  - \E[ h(Z) ] |$,
  where
   $\F_{\rm W}$ is the set of Lipschitz continuous functions $h:\R\to \R$
such that $\| h'\|_\infty \leq 1$.
It is trivial that 
$W_1(Y,Z) \leq W_2(Y, Z) \leq \| Y-Z\|_{L^2(\Omega)}$;
see \cite[Chapter 6]{Villani} for more 
on Wasserstein metrics.
When there is only one chaos in \eqref{CP0}, say, $\mathbf{J}_{t, R}$ as
in \eqref{twoV}, one has the Nourdin-Peccati bound (see, e.g., \cite[Theorem 5.2.6]{NP12})

\noi
\begin{align}  \notag 
{\rm dist}\big(   \mathbf{J}_{t, R} /   \s_{t,R}   , \NN(0,1)  \big)
\les \sqrt{  \s^{-4}_{t,R}  \E[ \mathbf{J}^4_{t, R}] - 3  },
\end{align}

\noi
and further estimation of the fourth moment using product formula
will lead to computations of contractions, from which one can 
get a rate of convergence. 
The above {\rm dist} can be $d_{\rm TV}$
and $W_1$ distances in
\eqref{def_dist}.  
Note that a general Malliavin-Stein bound can be used
to deal with the case of finitely many chaoses 
(see, e.g., Proposition \ref{MS_bdd}), while it is a hard task 
to establish a rate of convergence for the case of \emph{infinitely} many
chaoses. The search for a quantitative central limit theorem
in the context of   Breuer-Major was first carried out 
in the work \cite{NZ21}, and further explored
in \cite{NPY, NNP21, ADP24}.  These works heavily rely
on Malliavin calculus and in particular assume at least
that the underlying function $\varphi$ is differentiable 
with  $\varphi'$ square-integrable, which rules out
the case where $\varphi(x)= \ind_{\{ x\geq 0 \}}$.
In another work \cite{KN19}, the authors imposed further 
assumption on the coefficients (i.e., $a_q$ in \eqref{HerExp}),
which in turn requires Malliavin differentiability of $\varphi$;
see Section 3.2 in \cite{KN19}.
One may try to first establish a quantitative CLT 
in Wasserstein distance for finitely many chaoses 
and then get the rate for the whole series by a triangle
inequality with a $L^2(\Omega)$-bound on the  remainder.
 See, e.g., \cite{Tod19} for an example. 
This approach is not useful in the Breuer-Major setting, 
since the (normalized) tail does not tend to zero in mean square.
For this reason,  one can instead truncate the chaotic series of $Y_t$ in 
\eqref{CP0} up to order $N=N_t$  with $N_t\uparrow +\infty$,
then control the contribution of the  first $N$   chaoses

\noi
\begin{align}  \label{equ:YtN}
Y_{t,N} := \sum_{q=R}^{N} a_q \int_{tD} H_q(B_x) dx
\end{align}

\noi
 via the usual Malliavin-Stein method 
 (see, e.g., Proposition \ref{MS_bdd}),
and then exploit the variance of  normalized tail.
Usually, to find a rate  for the latter, one takes advantage 
of the  decay rate of chaotic coefficients 
$\{ a_q^2 q!\}_{q\geq 1}$, 
as done, for instance,
 in \cite{MW14, Ros19}. 
 Note that also in \cite{NPP11}, 
 where the authors proved a quantitative Breuer-Major theorem 
 without regularity assumptions on $\varphi$, 
 or on the decay of the coefficients $\{ a_q^2 q!\}_{q\ge 1}$, 
 their bounds in Wasserstein distance $W_1$
 depend on the latter decay through the constant $A_{2,n}$ (defined in \cite[(2.14)]{NPP11}).  
 On the contrary, in our paper 
 we do not rely on the regularity assumption on
 $\varphi$, nor on     the decay properties of  coefficients, 
 but rather   we only impose mild assumptions
  on the covariance function $\cC$
 that ensures  a sufficiently fast decay for $\int_{\R^d}|\cC(x)|^q\,dx$ 
  as $q\rightarrow\infty$ (see Lemma \ref{tech1}).
Our assumptions are motivated by 
 the study of  Berry's random wave model introduced below. See Remarks \ref{rem4_56}--\ref{rema1} for more discussions.

\subsection{Berry's random wave model}
\label{S1_2}

As a key motivating example for our work, 
we briefly introduce Berry's random wave model; 
see  \cite{Wig10, MW14, NPR19,BCW20, PV20, Vid21, GMT24, Smu24} 
and the references therein.

Consider $d=2$ and $\cC(x)=J_0(| x|)$, 
where $J_p$ denotes the Bessel function 
of the first kind of order $p\geq 0$:

\noi
\begin{align}  \notag  
J_p(r) 
= (r/2)^p \sum_{j=0}^\infty  (-1)^j \frac{r^{2j} }{4^j j! \Gamma(p+j+1)},
\quad r\in\R. 
\end{align}
See, e.g.,  \cite{Kra14} for basics on Bessel functions. 
With the covariance structure $\cC(x)=J_0(| x|)$, $x\in\R^2$,
the Gaussian  random field  $\bfB$
 is the so-called Berry's random field,
  which is conjectured to be  the universal model for
   high-energy eigenfunctions at least on
    ``generic'' classically chaotic billiards \cite{Ber77}. 
   Actually, Euclidean random waves are well defined in any dimension 
   $d\geq 2$, having a \emph{radial} covariance function of the form 
   $
     \cC(x) = b_d(|x|)$
      for $x\in \R^d$,
 where, for $r>0$, 
 
 \noi
 \begin{align}\label{bd}
     b_d(r) := 
     2^{\frac{d}{2}-1}\Gamma  (d/2   ) 
     J_{\frac{d}{2}-1}(r) r^{-\frac{d}{2}+1};
 \end{align}
 
 \noi
see, e.g.,  \cite[Proposition 6.1]{MN23}. 
Note that $b_2=J_0$ and   
$b_d$ can be represented
as a Fourier transform on the unit sphere with respect to the surface measure 
(see, e.g., \cite[(17)]{MN23}):

\noi
\begin{align}\label{bd2}
\cC(z) = b_d(|z|) = \frac{1}{\o_d} \int_{\{ |\xi|=1  \}} e^{ i z \cdot \xi} d\xi,
\end{align}
where $\o_d$ is a normalizing constant such that $\cC(0)=1$
(i.e.,  $\omega_d = \frac{2\pi^{d/2}}{\Gamma(d/2)}$).
 The integrability property of the covariance function $\cC(x)=b_d(|x|)$ 
 follows from the asymptotic behavior for large argument of Bessel functions:
 \begin{equation}  \label{Jdinf}
 J_p(r) =
  \sqrt{\frac{2}{\pi r}} \cos \Big( r - \frac{2p+1}{4} \pi  \Big) +   O(r^{-3/2}),
  \quad
  \text{as $r\to+\infty$};
  \end{equation}
 see, e.g., \cite[Theorem 4]{Kra14}.
In particular, we have the following asymptotic results. 
For simplicity, we state them assuming that $D$ is a \emph{ball}, 
but they also hold for more general domains;  
see \cite[Section 4]{Mai23} and \cite{MN23} for more details. 
Moreover, we implicitly assume that $\varphi$ is not linear 
(i.e., $a_q\neq0$ for some $q\ge2$), 
since otherwise $Y_t$ is obviously Gaussian.

As a consequence of   \cite[Theorem 1.3]{GMT24} and \cite[(22)]{MN23},\footnote{Here,  
we provide more details on the case $q=1$ 
while the other cases can be found in  \cite[Theorem 1.3]{GMT24}:
since $\cC(x) = b_d(|x|)$, with $b_d$ given as in \eqref{bd2}, we can write (assuming $D=\{|x|\leq 1\}$)
\[
(\ast) :=\int_{(tD)^2}\cC(x-y)dxdy 
= \frac{1}{\o_d} \int_{|\xi|=1} d\xi \Big| \int_{|x|\leq t} e^{ix\cdot \xi} dx \Big|^2
=\frac{(2\pi t)^d}{\o_d} \int_{|\xi|=1} |\xi|^{-d} J^2_{d/2}(t |\xi|) d\xi,  
\]
where the last inequality follows from a well known fact on Fourier transform of indicator of balls (see, e.g., 
\cite[Lemma 2.1]{NZ20a}).
That is, we have $(\ast) = (2\pi t)^d J^2_{d/2}(t) \les t^{d-1}$ as $t\to\infty$ in view of \eqref{Jdinf}.
This explains the big-O bound for  $q=1$. 
} 
we have 
	\begin{align}\label{lab1}
	\int_{(tD)^2}\cC^q(x-y)\,dx\,dy\asymp\,
    \begin{cases}
		  O(t^{d-1}) \qquad &\textnormal{if } q=1 \\
		t^{d+1}\qquad &\textnormal{if } q=2 \\
		t^2 \log t\qquad &\textnormal{if } (d,q)=(2,4)\\
		t^d \qquad &\textnormal{otherwise}\\
	\end{cases} 
    \qquad\text{as $t\to +\infty$.}
	\end{align}
	
\noi
In particular, these rates have two consequences.
\begin{itemize}
  \item[(i)]    The first chaos is always negligible. Indeed, when $R=1$, 
  we can study $Y_t(\varphi-a_1H_1)$ (with Hermite rank $R'\ge2$) instead of $Y_t(\varphi)$, 
  since they have the same asymptotic variance and distribution. 
    \item[(ii)] We can easily deduce from   \eqref{Var_Y}
    and \eqref{lab1} that    

    \noi
	\begin{align}\label{varBerry}
\s^2_t { \asymp}  
\begin{cases}
t^{d+1}\qquad & \text{if  $a_2\neq 0$}  \\
t^2 \log t\qquad &\text{if $a_2=0, a_4\neq 0, d=2$ }\\
t^d \qquad &\text{if $a_2=0, a_4\neq 0, d\geq 3$}\\
t^d \qquad &\text{if $a_2=a_4= 0$, $a_q\neq 0$ for $q=3$ or $q\geq 5$}.
\end{cases}
\end{align}

    \end{itemize}
\noi
Recall that $R$ denotes the Hermite rank of $\varphi$ and 
let $R'$ denote the Hermite rank of $\varphi -  a_R H_R$ 
	(called the {\bf second Hermite rank} of $\varphi$). 
 For example,   $a_2\neq 0$ if and only if $R=2$,  or $(R=1, R'=2)$;
 $(a_2=0, a_4\neq 0, d=2)$ if and only if   $(d=2,\,R\ge 3, \,a_4\neq0)$,  or  $(d=2,\,R=1,\, R'\ge 3, \,a_4\neq0)$.

Moreover, excluding the cases 

\noi
\begin{align}    \label{excases}
\begin{aligned}
    \begin{cases}
 & \text{(1) $\varphi$ linear} \\
  &\text{(2) $R=3$, $a_4=0$, $d=2$} \\
  &\text{(3) $R=3$, $d=3$} \\
  &\text{(4) $R=1$, $R'=3$, $a_4=0$, $d=2$} \\
   &\text{(5) $R=1$, $R'=3$, $d=3$},
  \end{cases}
\end{aligned}
\end{align}

\noi
the {\it spectral central limit theorem} in \cite[Theorem 1.2]{MN23} 
implies that

\noi
\begin{equation*}
 \frac{Y_t- \E[Y_t]}{\sqrt{\Var(Y_t)}} \xrightarrow[\rm law]{t\to\infty}  \NN(0,1).
\end{equation*}

\noi
For instance, when $\varphi(y) =\ind_{\{ y \ge u\}}$
 with $u\neq 0$, we have $R=1$ and $R'=2$, so we have a CLT with $\sigma_t^2 \asymp t^{d+1}$;
 when $\varphi(y) =\ind_{\{ y \geq 0\}}$,
  we have $R=1$, $R'=3$ and $a_4=0$ (as in cases (4)-(5) in \eqref{excases}), so $\sigma_t^2 \asymp t^d$ but the asymptotic distribution is an open problem for $d=2,3$;
   see Example \ref{Example1}.

The asymptotic behavior of $Y_t$, when $R\ge 5$, or $(R\ge 3$ and $d>3)$,
 can be established via an application of  Theorem \ref{recap}-(i), 
 while for the cases $R=2$ and the case $(R, d)=(4, 2)$,
  a reduction principle to the $R$-th chaos 
  (as in Theorem \ref{recap}-(ii)) holds. The case $(R,R',d)=(3,4,2)$ is slightly different, since a reduction principle to the $4$-th chaos holds.
  In the remaining cases not excluded in \eqref{excases} with $R=1$, we can exploit the fact that the first chaos is always negligible, studying the asymptotic distribution of $Y_t(\varphi-a_1H_1)$ instead of that of $Y_t(\varphi)$, so we can replace every statement above for $R\ge2$ with the same statement, but with $R$ replaced by $R'\ge2$ (and $R=1$).

For more details on $d=2$ and generalization to 
the higher dimensional setting, 
we refer interested readers to   \cite{MN23, GMT24} and \cite[Section 4]{Mai23}. 
Indeed, these references contain results 
 in   their generality, 
 except for the excluded  cases in \eqref{excases},
  in which only the asymptotic variance is known  
  and  the asymptotic distribution is totally open. 
 In particular, for $d > 3$, asymptotic normality
  for $Y_t = \int_{tD} \ind_{\{ B_x\geq 0\}} dx$ 
  (i.e., the \emph{nodal excursion volume}) can be proved via 
  the chaotic central limit theorem (as for Theorem \ref{recap}-(i)),
  once   its first chaotic component is shown to 
  be asymptotically  negligible.

We would like to point out that there is no   general result 
on quantitative central limit theorems 
for integral functionals of Berry's random wave model.
 Depending on the Hermite rank of $\varphi$ 
 and its regularity properties,
  one may choose one of the strategies mentioned 
  below  \eqref{def_dist}
and pursue it to get a rate of convergence. 
Our result, Theorem \ref{thm:main}-{\bf(1)},
 instead will allow to directly provide  a rate of convergence  
  for {\bf every} $\varphi$ such that   $Y_t= Y_t(\varphi)$  
  is known to admit  Gaussian fluctuations 
  (i.e., excluding the cases in \eqref{excases}).
 As we will see, 
  the covariance function of Berry's random wave model 
  \eqref{bd}  satisfies  our Conditions \ref{cond5}-\ref{cond8} below.

\subsection{Main results}
\label{S1_3}

Let us first recall the definition of almost sure central limit theorem
($\ASCLT$ for short).

\begin{definition} \label{ASCLT}
 A family  $\{  F_t \}_{t\ge 1}$ of real random variables on 
 $(\Omega, \mathcal F, \PP)$
 are said to satisfy 
the   $\ASCLT$,
if for $\PP$-almost every  $\omega\in\Omega$,

 \noi
 \begin{align} \label{def2}
\nu_T^\o
:= \frac{1}{\log T} \int_1^T \dl_{F_t(\omega)}   \frac{1}{t}  dt
\end{align}

\noi
weakly converges to  the standard Gaussian measure,
as $T\to\infty$.
See, e.g., \cite[Definition 1.1]{BXZ23}.

\end{definition}

\begin{remark}\label{equivdef}\rm  


Due to the separability of $\R$, one can find a 
sequence $\Phi:=\{\phi_n\}_{n\geq 1}$ 
of real bounded Lipschitz functions on $\R$
such that $\Phi$ is a separating class 
for the weak convergence of probability measures on $\R$.
It is then clear that 
$\{F_t\}_{t\geq 1}$ satisfies the $\ASCLT$
if and only if  for any bounded Lipschitz continuous 
function $g:\R\to\R$,

\noi
\begin{align}\label{equiv1}
\frac{1}{\log T} \int_1^T  g( F_t)   
\frac{1}{t} dt \xrightarrow[\textup{almost  surely}]{T\to+\infty} 
\int_\R g(x) \frac{1}{\sqrt{2\pi}} e^{-\frac{x^2}2} dx.
\end{align}

Because  $g$ is bounded, 
the above almost sure convergence still holds if we replace
$\frac{1}{\log T} \int_1^T  g( F_t)   \tfrac{1}{t} dt $
by $\frac{1}{\log T} \int_{t_0}^T  g( F_t)   
\tfrac{1}{t} dt $ for any given $t_0 >0$.
In this paper, when we consider a family of random  variables
$\{ F_t: t\geq t_0\}$ and $F_t$ may not be defined for $t < t_0$,
we will say 
$\{ F_t: t\geq t_0\}$
or $\{ F_t: t\geq 1\}$ satisfy the $\ASCLT$ if \eqref{equiv1}
holds. This shall not cause any ambiguity.
%

\end{remark} 

The $\ASCLT$ in its simplest form can be stated 
for i.i.d. random variables $\{X_i\}_{i\geq 1}$ with mean zero and   
variance one: the classical CLT asserts that
$M_n = \frac{1}{\sqrt{n} } (X_1+ ... + X_n)$ converges in law
to a standard normal distribution as $n\to+\infty$ 
and the  $\ASCLT$ asserts that 

\noi
\begin{align}\label{ASCLT_dis}
\frac{1}{\log n} \sum_{k=1}^n \frac{1}{k} g(M_k) 
\xrightarrow[\rm \text{almost surely}]{n\to+\infty} 
\int_\R g(x) \frac{1}{\sqrt{2\pi}} e^{-\frac{x^2}2} dx,
\end{align}

\noi
which is a discrete-time analogue of \eqref{def2} and \eqref{equiv1}.
The first $\ASCLT$ result was stated by P. L\'evy in
his book \cite[page 270]{Levy54}
without a proof, and
it had not gained much attention until being rediscovered 
by various authors in the 1980's 
(\cite{Fisher87, Bro88, Schatte88, LP90}).
See \cite{BC01, Jon07} for a brief introduction.

To the best of our knowledge, existing criteria to prove the  
$\ASCLT$ for  Gaussian functionals  would often require 
demanding  conditions on the  Malliavin derivatives
(see, e.g.,  \cite[Theorem 3.2]{BNT}). 
For instance, if $\varphi$ is a polynomial, 
or more generally  symmetric and twice differentiable
such that  $\mathbb E[|\varphi''(Z)|^4]<+\infty$ with $Z\sim  \NN(0,1)$, 
then \cite[Theorem 3.4]{BNT} immediately entails 
the $\ASCLT$ for the discrete-time counterpart of \eqref{def_Y}
as in  \eqref{ASCLT_dis}, when 
  the underlying Gaussian sequence 
  has summable covariance function. 
These approaches would not   perform well 
 with a wide class of pairs $(\varphi, \bfB)$,
  for which a central limit theorem already holds,
  for example, when $\varphi(r) = \ind_{\{ r\ge 0\}}$ 
  and $\bfB$ has the covariance structure  \eqref{bd} 
  (Berry's random wave model) with $d>3$.
 In this case,  $\varphi(r) = \ind_{\{ r\ge 0\}}$  is not regular, 
 and moreover dealing with differentiability properties of 
 the excursion volume is not an easy task in general; 
 see \cite{AP20, PS24} for the case of the volume of level sets, 
 and \cite{CGR24} for the investigation of Malliavin differentiability 
 of  smooth statistics of Gaussian random waves 
 on the round sphere, 
 whose aim is to prove a quantitative CLT
  in the total variation distance 
  following  the approach in \cite{BCP}.

\medskip

The main goal of our paper is to establish $\ASCLT$
 for $ \{ Y_t \} $ in \eqref{def_Y}
\emph{without imposing  any regularity assumption
on the function  $\varphi$}.
Notably, we will be able to establish the very first 
$\ASCLT$ for Berry's random wave model (see Corollary 
\ref{cor:berry}).

 In order to state our main theorems, we need to introduce a few assumptions. First of all, we need some control on the behavior of the covariance
 function  $\cC$ at infinity (global dependence) and at zero (local dependence).

\begin{condition}
    \label{cond5}
   There exist some  $\dl, C_1\in(0,\infty)$  such that 
    $|\cC(x)| \leq   C_1 |x|^{-\dl}$ for every $x\in\R^d$.

\end{condition}

\begin{condition}
    \label{cond6}
    
There exist some constants  $C_2,\epsilon, \al \in(0,\infty)$ 
such that   for $|x|<\epsilon$:
\begin{equation} \notag  
 \cC(x) \leq 1 - C_2 |x|^{\al}.
 \end{equation}
   
  \end{condition}
  \noi
  See Remark \ref{rem4_56} for more elaboration on the above two conditions.
Define now, for any integer  $M\ge R$,  
 
 \noi
  \begin{equation}\label{wrM}
     r\in\R_+
            \mapsto w_{r,M}:=\int_{|x|\le r}\sum_{q=R}^{M} q!a_q^2 \cC^q(x)\,dx.
            \end{equation}

\noi
Roughly speaking, \eqref{wrM} refers to the integral  behavior 
of the covariance function 
${\rm Cov}( \varphi(B_x), \varphi(B_0))$
 of 
$\varphi(B_x)=\sum_{q}a_qH_q(B_x)$ and 
$\varphi(B_0)$ on growing balls
 when we cut the series at the threshold $M$.

 \noi
\begin{condition}
    \label{cond7}
    One of the  following three sets of conditions holds:
    \begin{enumerate}
    
\item[(c1)]   $\cC\in L^R(\R^d)$ and $\varphi-\E[\varphi(Z)]$ is not odd with $Z\sim\NN(0,1)$;

\item[(c2)]    $\cC^R\ge0$ and $r\mapsto w_{r,R}$ in \eqref{wrM} 
        is regularly varying
        at infinity 
         {\rm(}see Section \ref{SEC_RV}{\rm)};
         
\item[(c3)]   $D$ is a centered ball and 
$\exists$ $M>\frac{d}{\dl}-1$ {\rm(}$\dl$ as in Condition \ref{cond5}{\rm)}
 such that, as $r\to\infty$, 
 the function 
            $   r\in\R_+\mapsto w_{r,M}$ in \eqref{wrM}
             is regularly varying at infinity
 with a  limit $w_{\infty,M}\in(0,\infty]$.

                \end{enumerate}
\end{condition}

Furthermore, let us introduce 
for $t\ge 1$, 
\begin{equation}
    \label{def_HT}
    h_t(k_1,k_2):=
    \int_{(tD)^4}
    \cC^{k_1}(x-y)\,
    \cC^{k_1}(z-w)
    \cC^{k_2}(x-z)
    \cC^{k_2}(y-w)\,dx dy dzdw
\end{equation}
and for any integer $m\geq R$,

\noi
\begin{align}
    \label{XIM0}
    \xi_m(t):=\sup\Big\{ \frac{\sqrt{h_t(k_1,k_2)}}{\sigma_t^2}:
    k_1,k_2\ge 1 \,\,\, {\rm and}\,\,\, k_1+k_2\geq m \Big\}.
\end{align}

\begin{condition}    \label{cond8}
   There exist two constants   $\theta_0 ,  C  \in(0, \infty)$
  such that
  \begin{equation}\label{xiLeo}
  \xi_R(t) \leq   \frac{C}{   \log^{\theta_0}(t)},
  \end{equation}
  where $\xi_R$ is defined as in \eqref{XIM0}.
   See Remark \ref{rem4_78} for more discussions.
\end{condition}

We are now in a position to state our main result.

\begin{theorem}\label{thm:main}
Let $\bfB=\{ B_x\}_{ x\in \R^d}$ be a real-valued  continuous 
centered stationary Gaussian random field
with   covariance function $\cC$ 
as in \eqref{cov_B} and \eqref{C00}.
 Assume that  
Condition \ref{cond5} holds for some $\dl > 0$,
Condition \ref{cond6} holds for some $\al >0$,
Condition \ref{cond8} holds for some $\theta_0 > 0$,
and Condition \ref{cond7} also holds. 
Consider the random variable $Y_t$ 
defined as in \eqref{def_Y}
with $\varphi$ having Hermite rank $R\geq 1$
and   $\s_t^2 = \Var( Y_t)$.
Then, the following statements hold.

\smallskip
\noi
{\bf (1) [\textsf{QCLT}]} $\s^2_t > 0$ for any $t > t_0$
with $t_0 > 0$ large enough,
 and 
we have the following quantitative central limit theorem 
described in quadratic Wasserstein distance \eqref{def_dist}:

\noi
\begin{align} \label{QCLT_bdd1}
W_2 \Big(  \frac{ Y_t - \E[Y_t]}{\s_t},  \NN(0,1)  \Big) 
\les   \log^{-\theta_1}(t),
\end{align}

\noi
where $\theta_1 
=   \min\{  \theta_0 - \theta, \tfrac{\theta d}{2\al} \}$
and we can choose freely any 
$\theta$ satisfying

\noi
\begin{align}\label{th0}
0<  \theta <  \min\{  1, \theta_0  \}.
\end{align}
\noi
In particular, if $\theta_0<\frac{d+2\al}{2\al}$, then taking $\theta=\frac{2\theta_0 \al}{d+2\al}<\min\{\theta_0,1\}$, we obtain
\[
W_2 \Big(  \frac{ Y_t - \E[Y_t]}{\s_t},  \NN(0,1)  \Big) 
\les  \log^{-\theta_0\frac{d}{d+2\al}}(t)  .
\]

\smallskip
\noi
{\bf (2) [$\ASCLT$]} The family 
$\{(Y_t - \E[ Y_t] )/\s_t : t\geq t_0 \}$ satisfies an $\ASCLT$.

%
%

 \end{theorem}

We will prove Theorem \ref{thm:main}
in Section \ref{SEC3_1}. We would like to point out that 
 we actually
have an explicit limiting order of $\s_t^2$,
which is the content of Lemma \ref{tech2}.

\medskip

  Our conditions could imply certain 
regularity of \eqref{def_Y}. In order to explore this fact, we need a bound for the
 moments of the covariance function of the underlying Gaussian field.
\begin{lemma}\label{tech1}
Let $C_1, C_2, \al, \dl, \eps_0$ be positive constants. 
Suppose the covariance function $\cC$ \eqref{cov_B} 
satisfies the following bounds:
\noi
\begin{align} \label{cond56}
| \cC(x) | \leq C_1 |x|^{-\dl}, \,\, \forall x\in\R^d
\quad
\text{and}
\quad
 1 - \cC(y) \ge C_2  |y|^\al 
 \,\,\text{for $|y| \leq \eps_0$.}
\end{align}
\noi
Then,  there exists $c>0$ such that 
$$
\int_{\R^d}|\cC(z)|^N dz 
\le\,c\, N^{-\frac{d}{\alpha}}
$$
for any integer $N \geq \frac{d}\dl +1$.  Moreover if we additionally assume that 
\begin{equation}\label{minore}
1-\cC(y) \le C_3 |y|^{\alpha}
\end{equation}
 for some $C_3>0$ and for $|y|\le \eps_0$, we have 
\begin{equation}\label{stima_momenti_gen}
\int_{\R^d}|\cC(z)|^N dz 
\asymp N^{-\frac{d}{\alpha}}.
\end{equation} 
\end{lemma}

\noi
The proof of Lemma \ref{tech1} is technical, hence we postpone it to Section \ref{S5}. 
It is worth stressing that \eqref{stima_momenti_gen} holds  
with $|\cC^N|$ replaced by $\cC^N$ under the same assumptions \eqref{cond56}-\eqref{minore}, 
which are satisfied by Berry's random wave model; 
see Remark \ref{remModulo}. 
Moreover, 
Lemma \ref{tech1} ensures that 
the limiting variance $\sigma^2$ in Breuer-Major's theorem (see \eqref{varBM})
is strictly positive, as soon as $\varphi$ has an infinite chaos expansion,
provided the   assumptions \eqref{cond56}-\eqref{minore} hold. 
Note that obtaining the positivity of the limiting variance is often a difficult task
and we refer interested readers to the recent paper \cite{Gas25} by Gass,
who presents a  Fourier-type criterion 
based on spectral measure and its interplay with the chaotic components.

\begin{remark}[Malliavin regularity]\rm \label{rem_malliavin}
Under Conditions \ref{cond5} and \ref{cond6}, 
we have certain Malliavin differentiability for $Y_t$ (see \eqref{DYinf}). 
Actually,  Lemma \ref{tech1} indicates that 
$\int_{\R^d}|\cC^q(z)| dz \les q^{-d/\al}$,
so that 

 \noi
\begin{align}\label{stima_derivata}
\sum_{q=R}^\infty q^{\frac{d}{\al}} 
\E\bigg[ \bigg( \int_{tD} a_q H_q(B_x) dx \bigg)^2 \bigg]
\leq t^d \Vol(D)  \sum_{q=R}^\infty  a^2_q  q! q^{\frac{d}{\al}} 
 \int_{\R^d} |\cC^q(z)| dz < \infty,
\end{align}

\noi
In particular, if $d/\alpha \ge 1$, then $Y_t\in \mathbb D^{1,2}$ 
(i.e., $Y_t$ is Malliavin differentiable). 
In general, we observe an {\it inverse} relationship between the regularity of the field 
and the smoothness of the functional:
The parameter $\alpha$ restricts the maximal regularity of the sample paths
of the Gaussian field $\bfB$ (see, e.g., Example \ref{example_WM}), 
and as it gets smaller, we have faster decay 
in terms like \eqref{stima_momenti_gen}, which leads
to higher-order Malliavin differentiability of the integral function $Y_t$
in view of \eqref{stima_derivata}.
In this paper, we are not going to further explore 
the Malliavin differentiability
of $Y_t$,
since we do not need it in order to prove 
an $\ASCLT$ or 
a quantitative central limit theorem.
 We restrict ourselves to the case of Berry's random wave model 
 (Corollary \ref{corMalliavinBerry});
see \cite{Mai26} for further  discussions. Let us briefly mention that 
the author of \cite{Mai26} proved that   the excursion volume 
$Y \in \mathbb D^{k, 2}$
 if and only if $k < \frac{d}{\al} +\frac{1}{2}$. 
 See \eqref{DK2} for the definition of  $ \mathbb D^{k, 2}$.

 \end{remark}

The above remark ensures that our method leading to Theorem \ref{thm:main} 
allows us to prove $\ASCLT$ and quantitative central limit theorems 
even for functionals $Y_t(\varphi)$ 
with $\varphi$ merely in $L^2(\R, \frac{1}{\sqrt{2\pi}} e^{-x^2/2}dx)$
(i.e., no Malliavin differentiability like $\varphi(x) = \ind_{\{ x > 0\}}$).

 An important consequence of Lemma \ref{tech1} and 
 the asymptotic  \eqref{stima_momenti_gen}
 (more precisely the one in  Remark \ref{remModulo})
 is the following result that resolves an open question 
 on the Malliavin differentiability of the excursion volume of 
 Berry's random wave model.

 \begin{corollary}\label{corMalliavinBerry} 
 Let $\mathcal C$ be the covariance function of $d$-dimensional 
 Berry's random wave model 
 {\rm(}$d\ge 2${\rm)}, i.e., $\mathcal C(z)=b_d(|z|)$ as in \eqref{bd}, then 

 \noi
  \begin{align}\label{stima_momentiBessel}
 \int_{\mathbb R^d} \mathcal C(z)^q\, dz \asymp q^{-\frac{d}{2}}. 
 \end{align}
 
 \noi
As a consequence, the random variable $Y_t$ 
belongs to $\mathbb{D}^{1, 2}$, i.e., 
it is Malliavin differentiable. 
\end{corollary}

\noi
 In particular, the excursion volume $Y_t(\varphi)$ 
 with $\varphi(r)=\ind_{\{ r \geq u\}}$ 
 of Berry's model  is regular in the Malliavin sense, 
 a question left open so far; 
 see, e.g., \cite{CGR24} for the case of smooth statistics, 
 i.e., for \emph{smooth} $\varphi$. 
 Let us now focus on \eqref{stima_momentiBessel}:
  in \cite{FCMT19} some numerical simulations  
  show that the moments of $J_0(z)=b_2(|z|)$
 asymptotically behave as $2/q$, 
 our estimate hence confirms this conjectured behavior; 
 see equation (39) therein. We would also like to mention a result 
 in \cite{GMT24}:  the $q$-th moment of the Bessel function is 
 strictly positive for every $q$.

\bigskip

Before we state corollaries of our main theorem, 
let us present a few examples of covariance functions
that satisfy Conditions \ref{cond5}-\ref{cond8}.

\begin{example} \label{example_WM} \rm
It is known \cite{Gne02} that if $\bfB$ is nondegenerate, isotropic, and  
$\cC(x)= 1-a|x|^{\alpha} + o(|x|^{\alpha})$ 
as $|x|\to 0$ for some $\alpha>0$, 
then  $\al\in (0,2]$.
 For instance, 
\begin{equation}\label{cov_exp}
     \text{$\cC(x) = e^{-|x|^\alpha}$ with  $\al \in (0,2]$}
\end{equation}

\noi
satisfies Condition \ref{cond6}, 
so do  the Whittle-Mat\'ern family of covariance functions

\noi
\begin{align}\label{cov_WM}
    \text{$\cC(x) = \frac{2^{1-\mu}}{\Gamma(\mu)} |x|^\mu K_\mu(|x|)$
    with $\mu \in (0,+\infty)$ and $\al = 2 \min\{ \mu, 1\}$},
\end{align}

 \noi
 where $K_\mu$ is the modified Bessel function 
 (of the second kind) with index $\mu$. Since 
\[
  \text{$K_\mu(r) = e^{-r}\sqrt{\frac{2}{\pi r}} \big( 1  + O\left( r^{-1}\right )\big)$ as $r\to +\infty$},
\]

\noi 
in particular the covariances \eqref{cov_exp} 
and \eqref{cov_WM} satisfy also Condition \ref{cond5}, 
are in $L^p(\mathbb R^d)$ for every $p>0$ 
and satisfy also Conditions \ref{cond7}-\ref{cond8} 
(by the positivity of $K_\mu$ and Lemma \ref{tech4}).
These models allow one to the control of the field's sample path regularity
 by tuning the parameter $\al$. 
Specifically, a smaller value of $\al$ corresponds to 
a higher degree of roughness in the trajectories, 
whereas $\alpha=2$ represents the smoothest case.
 Geometrically, this fine-tuning implies that we can treat excursion sets 
 whose boundaries have specific regularity.

\end{example}

\begin{corollary}[$\ASCLT$ in the Breuer-Major setting]
\label{cor:ASBM}
Let $Y_t$ be defined as in \eqref{def_Y}.
Assume that the following conditions hold:
\begin{itemize}
\item  $\varphi-\E[\varphi(Z)]$, with $Z\sim\NN(0,1)$, is not odd, 
\item Condition \ref{cond5} holds with $\delta>d/R$,  
\item Condition \ref{cond6} holds.

\end{itemize}

\noi
Then,  $\Big\{ \dfrac{Y_t - \E[ Y_t]}{\sqrt{\Var(Y_t)}} :
t \geq 1  \Big\}$ 
satisfies an $\ASCLT$.

\end{corollary}

The proof of Corollary \ref{cor:ASBM} 
will be given in Section \ref{SEC3_2}. 
Note that  Condition \ref{cond5} with $\dl>d/R$
 implies $\cC\in L^R(\R^d)$ and $\varphi-\E[\varphi(Z)]$ non-odd implies
  $\s^2>0$ as in \eqref{varBM}. Thus, by Breuer-Major theorem,
   we immediately obtain a CLT for $Y_t$.
   The above Corollary \ref{cor:ASBM}
   indicates that  under the additional 
   Condition \ref{cond6}, we also obtain an $\ASCLT$.

\begin{corollary}[$\ASCLT$ for Berry's random wave model]
\label{cor:berry} 

Let  $\bfB$ be $d$-dimensional 
Berry's random wave model with   $d\ge2$.
Let $\varphi$ be
given as in \eqref{HerExp}, with Hermite rank $R\geq 1$
and $Y_t$ be defined as in \eqref{def_Y}.
Exclude the cases in  \eqref{excases} 
and suppose that $D$ is a ball centered at $0$. 
Then, 
$\Big\{ \dfrac{Y_t - \E[ Y_t]}{\sqrt{\Var(Y_t)}} :
t \geq 1  \Big\}$
satisfies the $\ASCLT$.
\end{corollary}
 
 The above Corollary \ref{cor:berry} presents the very first 
$\ASCLT$ result in the context of Berry's random wave model. 
 
  \subsection{Further comments}

  Now let us explain our conditions   and highlight some consequences.

\begin{remark}[On Conditions \ref{cond5} and \ref{cond6}]\rm \label{rem4_56}
(i)
Clearly, Condition \ref{cond5} refers to integrability   of 
$\cC$ on $\R^d$, i.e., under Condition \ref{cond5},
$\cC\in L^{m}(\R^d)$ for any $m > d/\dl$. 
An assumption of this kind is quite common  
for obtaining  limit theorems for functionals of random fields, 
as it allows one to easily control the dependency structure 
of   $\varphi(\bfB) = \{ \varphi(B_x): x\in\R^d\}$.

\smallskip
\noi
(ii)
The technical Condition \ref{cond6} is   
  satisfied by ``most'' Gaussian fields 
(for instance, if $\bfB$ is not degenerate, isotropic,
 and   mean square differentiable, then Condition \ref{cond6} holds true 
with $\al=2$) 
-- except  in pathological cases.\footnote{Here is a 
pathological example:
if  $\varphi(r)=r$ and $\cC\equiv 1$,  
then we get a degenerate Gaussian field 
$Y_t = t^d \textnormal{Vol}(D) B_0$,
and thus  a central limit theorem trivially holds   for $ \{ Y_t\}$,
 but the $\ASCLT$   clearly does not hold true.}
Note that  Condition \ref{cond6}, together with Condition \ref{cond5},
will ensure a power decay of $\int_{\R^d} |\cC(z)|^q dz$ 
as  $q\to+\infty$;
see Lemma \ref{tech1}.

\end{remark}

\begin{remark}[On Conditions \ref{cond7} and \ref{cond8}] \rm \label{rem4_78}

(i) Condition \ref{cond7} is a technical condition 
for estimating the variances of $Y_t$ and its truncated version
$Y_{t, N_t}$  defined in \eqref{equ:YtN};
see Lemmas \ref{tech2}-\ref{tech3}.
Note that getting the order of the limiting variances is usually 
the first step to prove a limit theorem.

\smallskip
\noi
(ii) 
The ``compact'' assumption $\cC^R \ge 0$ in (c2) of Condition \ref{cond7}
 on one hand would imply $\s^2>0$ in \eqref{varBM} when $\cC\in L^R$, 
 and on the other hand would be needed to handle the case in Theorem \ref{recap}-(ii) when $\cC\notin L^R$. 
As indicated in Theorem \ref{recap}-(ii),
  the assumption $\cC^R\ge 0$  is not optimal: 
    for Berry's random wave model (see Section \ref{S1_2}),
   it is possible to establish the asymptotic normality for $Y_t$
   when $R=1$ and
    the second chaotic component  is dominant.

\smallskip
\noi
(iii)  Condition \ref{cond8}  is related to a uniform control 
of the contractions (see Section \ref{SEC_MS}), 
arising from the computation of  
the fourth moment of 
each chaotic component of $Y_t$ in \eqref{CP0};
see \eqref{GTQ_r2}.
 As already mentioned in Remark \ref{rem4_56}-(i), Condition \ref{cond5}
implies $\cC\in L^{m}(\R^d)$ for $m > d/\dl$.
As we will see in Lemma \ref{tech4}, $\xi_m(t)$ has power decay in $t$
under the additional  Condition \ref{cond7},
and 
the combination of Condition \ref{cond5} and Condition \ref{cond7}
leads to a control on the sum of  higher-order chaoses (see, e.g., 
 \eqref{RMB}-\eqref{MPT_bdd}),
while the control for lower-order chaoses is guaranteed 
by Condition \ref{cond8}.

\end{remark}

Let us now comment a bit on the quantitative CLT in 
 Theorem \ref{thm:main}-{\bf (1)} and the restriction 
 \eqref{th0} on $\theta$.
 
\begin{remark} \rm
\label{rema1}

(i)  It is well known that the rate of convergence in 
Theorem \ref{thm:main}-{\bf (1)} is  not  optimal, both in cases like (i) 
(see, e.g., \cite[Theorem 1.2]{KN19}) 
and in cases like (ii), where  a finite number of chaotic components of 
$Y_t$ are  dominant.
 Indeed,  recalling \eqref{equ:YtN}, 
 if   $\Var(Y_{t,M})\sim \Var(Y_t) = \s_t^2$  for some fixed $M$ (i.e., 
 the first $M$ chaoses are dominant),
 one could proceed with the usual triangle inequality to get 
     \[
     W_2\Big(\frac{Y_t - \E[Y_t]}{ \sqrt{ \Var(Y_{t})}}, \NN(0,1) \Big)
    \les 
  W_2 \Big(\frac{Y_{t, M}}{\sqrt{\Var(Y_{t, M})}  }, \NN(0,1) \Big)
   + \sqrt{  1 -  \frac{\Var(Y_{t, M})}{\Var(Y_{t})}      },
    \]
 
 \noi
where the second term   vanishes as $t\to\infty$.
In this case,  we do not have to choose $M = N_t \uparrow +\infty$, 
as $t\to+\infty$. For example, 
when  $\bfB$ is  Berry's random wave model, 
excluding the cases in \eqref{excases}, 
Conditions \ref{cond6}-\ref{cond8} hold with 
$\al=2$, $\theta_0=1/2$
(see the  proof of Corollary \ref{cor:berry}  and of Corollary \ref{corMalliavinBerry}),
and thus  Theorem \ref{thm:main}-{\bf (1)} yields
(for $d=2$):

\noi
\begin{align*}
      W_2 \Big( \frac{Y_t - \E[Y_t]}{\s_t} ,  \NN(0,1) \Big) 
    \les    \log^{- \frac{1}{6}}(t).
\end{align*}

\noi
   However, if $a_2\neq0$, then choosing $M=2$,  \eqref{boundcontractionberry} and \eqref{varBerry} yield
    \[
   W_2 \Big(\frac{Y_t - \E[Y_t]}{\s_t} ,Z\Big)\lesssim \frac{1}{t^{1/4}}+\sqrt{\frac{\log(t)}{t}} 
    \les \frac{1}{t^{1/4}} .
    \]

\smallskip
\noi
(ii) As already mentioned, we will  let $N = N_t$ depend on $t$
as  in \eqref{equ:YtN}. 
 More precisely, we let $N_t$ be an integer-valued function that
is non-decreasing in $t$ with

\noi
\begin{equation}
\label{def_NT}
    \qquad N_t \sim \log^{\theta}(t),
\end{equation}

\noi
where $\theta$ satisfies the restriction  \eqref{th0}. 
It is worth noting that the smaller the value of $\alpha$ is (i.e., the rougher the field), the larger $\theta$ becomes.
Let us point out that this restriction  \eqref{th0}
will be combined 
with the bound \eqref{MPT_bdd},
 to control the expressions in 
\eqref{deda1} and \eqref{th0b}.

%
%

\end{remark}

 \noi
$\bul$
{\bf Organization of this paper.}  The rest of this paper is organized as follows:
Section \ref{S2} presents a few preliminaries on 
the theory of regular varying functions, Malliavin-Stein method,
and auxiliary results on $\ASCLT$ that are needed for our proofs.
We will prove our main results in Section \ref{SEC3}
and provide several key technical results in Section \ref{S5}.

\bigskip
   \noi
$\bul$
{\bf  Acknowledgments.} The authors would like to thank the an  anonymous referee
for the helpful comments that improved the quality of the paper.
The authors would also like to thank J\"urgen Angst and Guillaume Poly 
for organizing the 2023 conference ``Random nodal domains'',
during which this project was initiated.  
L.M. is associated with INdAM (Istituto Nazionale di Alta Matematica ``Francesco Severi'')
and the GNAMPA group. L.M. acknowledges financial support from the MUR 2022 PRIN project
GRAFIA (project code: 202284Z9E4) and the MUR Excellence Department
Project MatMod@TOV, awarded to the Department of Mathematics, University of Rome Tor Vergata
(CUP E83C18000100006).
M.R. would like to acknowledge support from the PRIN/MUR project 2022 GRAFIA 
and the GNAMPA-INdAM project 2024 Geometria di onde aleatorie su variet\`a.

\section{Preliminaries} \label{S2}
We collect a few preliminaries in this section. 
Section \ref{SEC_RV} contains basics on functions of regular variation,
Section \ref{SEC_MS} presents a brief introduction to our toolbox 
(i.e., the Malliavin-Stein method),
and Section \ref{SEC_aux} collects a few auxiliary 
results on $\ASCLT$

\subsection{Functions of regular variation} \label{SEC_RV}
We say a   measurable function $L: \R_+\to \R$ is slowly varying at infinity
if $L$ is positive on $[X, \infty)$ for some $X > 0$ and 
\[
\lim_{t\to+\infty} \frac{L(\lambda t)}{L(t) } = 1
\]
 for any $\lambda > 0$.
 We say $f: \R_+\to\R$ is regularly varying at infinity with index $\rho\in\R$,
 if $f(t) = t^\rho L(t)$ for some slowly varying function $L$.
 And it is a well known fact that 
 for  $L:\R_+\to\R$ slowly varying at infinity, one can represent it as 
 follows: there is some $a > 0$ such that

 \noi
\begin{align}\label{SV_rep}
L(t) = c(t) \exp\bigg( \int_a^t \frac{\eps(u)}{u} du \bigg)
\end{align}
with $c(t)\to c \in(0,\infty)$ and $\eps(u)\to 0$ as $t, u\to+\infty$;
see \cite[Equation (1.5.1)]{BGT87}.
For example, we can represent $\log(t)$ as follows:
\[
\log(t)  = \exp\bigg( \int_1^t \frac{1}{u \log u} du \bigg).
\]
We refer readers to the classic book \cite{BGT87} for general theory
on functions of regular variations. 
In what follows, we collect a few results that are needed for our proofs.

\begin{lemma}\label{lemP}
 
 \textup{(i) (Asymptotic behavior at infinity)}  
 Suppose $L:\R_+\to\R$ is slowly varying at infinity. 
 Then, for any $\eps > 0$, we have 
 \begin{align}\label{asy1}
  L(t) t^{-\eps} \xrightarrow{t\to\infty} 0
\quad
  {\rm and}
  \quad
    L(t) t^{\eps} \xrightarrow{t\to\infty} +\infty.
  \end{align}
  In particular, $L$ is locally bounded on $[X, \infty)$ for some $X>0$.

\smallskip

\noi
   \textup{(ii) (Index of regular variation)}  
Assume that  $f$ is regularly varying at infinity with index $\rho$
such that  $f(\infty) = \lim_{t\to\infty} f(t)$ exists in $(0,\infty]$.
Then,   $\rho \geq 0$.

\smallskip

\noi
\textup{(iii) (Integral of regularly varying functions)}   
Suppose  $L:\R_+\to \R$  is slowly varying at infinity  
and let   $\rho \in\R$.
Then, we have, with $X$ as in {\rm (i)}, that
 for any $\s \geq  - (\rho +1)$,
\[
\frac{x^{\s+ \rho + 1} 
   L(x) }{\int_X^x t^{\s +\rho} L(t) dt} \xrightarrow{x\to+\infty} \s + \rho +1.
\] 
Moreover,   
$x\in\R_+\mapsto \int_X^x t^{\s+\rho} L(t) dt$ is regularly varying with 
index $\s+\rho+1$ {\rm(}also in the case $\s+\rho+1=0$, where it is in particular slowly varying{\rm)}.

\smallskip

\noi
\textup{(iv) (Potter's bound)}
Let $f$ be  regularly varying   at $+\infty$
with index $\rho$.  
Then, for any $\dl > 0$ and for any $A>1$, 
there exists some 
constant $X = X(A, \dl)$ such that 
\[
 \tfrac{1}{A} \min\big\{  (y/x)^{\rho+\dl},    (y/x)^{\rho-\dl} \big\}       \leq    \frac{f(y)}{f(x)} \leq A \max\big\{  (y/x)^{\rho+\dl},    (y/x)^{\rho-\dl} \big\}   
\]
for any $x, y\in [ X, \infty)$.

\end{lemma}

\begin{proof}
(i) It follows from the representation \eqref{SV_rep} that

\noi
\begin{align} \notag 
L(t) t^{2\eps} =   c(t) \exp\Big(  \be \log a  
+   \int_a^t \frac{\eps(u) +2\eps }{u} du \Big).
\end{align} 

\noi
For $a_1 > a$ large enough,
we can have  $\eps(u) + 2\eps > \eps > 0$ for  any $u\geq a_1$.
Thus, 
$L(t) t^{2\eps} \ges \exp\big(    \int_{a_1}^t \tfrac{\eps}{u} du \big)
 \ges t^{\eps}$.
It follows that $L(t) \ges t^{-\eps}$ for $t\geq a_1$.
In the same way, we can show  $L(t) \les t^{\eps}$ for $t\geq a_2$
with $a_2 > a$ large enough.
That is, 
\eqref{asy1} is proved with

\noi
\begin{align}\label{lemP1}
  t^{-\eps}  \les   L(t) \les  t^{\eps} \,\,\,
 \text{for $t\geq X$ with $X$ large enough.}
\end{align}

\noi
Therefore, the local boundedness of $L$ follows.

\medskip

\noi
(ii) We first write $f(t) = t^\rho L(t)$ with $L$ slowly varying at infinity. 
Suppose that $\rho < 0$. Then, 
It follows from \eqref{lemP1} with $\eps = -\rho/2$ that 
\[
  t^{-3\eps} \les  f(t) \les  t^{-\eps}
\]
as $t$ tends to infinity, which is a contradiction to $f(\infty) > 0$.
Therefore, $\rho \geq 0$.

\medskip

\noi
(iii) This is a simple reformulation 
of  Karamata's Theorem. See \cite[Theorem 1.5.11]{BGT87}.

\medskip

\noi
(iv) The upper bound is  taken from  \cite[Theorem 1.5.6(iii)]{BGT87}. 
Then, we apply the upper bound to obtain
\[
\frac{f(x)}{f(y)} \leq A \max\big\{  (x/y)^{\rho+\dl},    (x/y)^{\rho-\dl} \big\}   
\]
from which we get
\[
\frac{f(y)}{f(x)} \geq \tfrac1{A} \min\big\{  (y/x)^{\rho+\dl},    (y/x)^{\rho-\dl} \big\}.   
\]

Hence, the proof is completed. 
\qedhere

\end{proof}

\subsection{Basics on Malliavin-Stein method} 
\label{SEC_MS}
Let $\bfB=(B_x)_{x\in\R^d}$ be a 
real-valued, centered stationary
Gaussian random field with covariance function $\cC:\R^d\to\R$.
Throughout this paper, 
we assume that 
\begin{center}
$\bfB$ is almost surely continuous and $\cC(0)=1$.
\end{center}
In particular, $\cC$ is continuous on $\R^d$.
In what follows, we first build an isonormal framework from 
the given Gaussian random field $\bfB$, then we develop
the basic tools from Malliavin calculus and finally state 
the crucial bounds from the Malliavin-Stein method. 

\medskip

\noi
$\bul$ {\bf Isonormal framework.} 
By continuity of $\bfB$,
the $L^2$-Gaussian Hilbert space $\H_\bfB$ generated by $\bfB$
is identical to that generated by $\{B_x: x\in \mathbb{Q}^d\}$,
with $\mathbb{Q}$ the set of rational numbers in $\R$.
In particular, the said Gaussian Hilbert space $\H_\bfB$ is a real separable 
Hilbert space. Let $\fH$ be a real separable Hilbert space. Then, 
by the general theory of Hilbert spaces, 
one can find an isometry

\noi
\begin{align}\label{e_x0}
I_1 :   \fH \to  \H_\bfB
\end{align}

\noi
and   
$e_x\in\fH$ such that

\noi
\begin{align}\label{e_x1}
I_1(e_x) = B_x
\end{align}
 for any $x\in\mathbb{Q}^d$ (then extended 
 for any $x\in\R^d$ by continuity). Note that $\fH$ is the closure
 of the linear span of $\{ e_x: x\in\mathbb{Q}^d\}$.
 The resulting process $\{ I_1(h): h\in\fH\}$ is called an isonormal 
 Gaussian process indexed by $\fH$. It is a centered Gaussian family 
 with covariance structure given by 
 
 \noi
 \begin{align}\label{e_x2}
 \E[ I_1(e_x) I_1(e_y) ] = \langle e_x, e_y\rangle_\fH 
 : =\E[ B_x B_y] = \cC(x-y).
 \end{align}
This isometry relation can be easily  extended to 
\[
 \E[ I_1( h) I_1(g) ] 
 = \langle h, g\rangle_\fH 
 = \lim_{n\to+\infty}  \langle h_n, g_n\rangle_\fH
\]
for some $g_n, h_n\in \text{span}\{e_x: x\in \mathbb{Q}^d\}$ that
converge to $g, h$ in $\fH$ respectively.
Suppose $h_n 
= \sum_{x\in A_n} \al_x e_x$ and $g_n = \sum_{y\in B_n} \be_y e_y$
with $A_n, B_n$ finite subsets of  $\mathbb{Q}^d$ and $\al_x, \be_y\in\R$, 
then  $ \langle h_n, g_n\rangle_\fH 
= \sum_{(x, y)\in A_n\times B_n} \al_x\be_y \cC(x-y)$.

 \medskip
 
 \noi
$\bul$ {\bf Hermite polynomials.} 
 The well known family of Hermite polynomials 
 are orthogonal polynomials
 for the standard   Gaussian measure on $\R$. 
 They can be defined recursively:
 
 \noi
 \begin{align}\label{H1x}
H_0 =1, \quad H_1(x) =x, \quad H_2(x) = x^2 -1, \quad
{\rm and}\quad H_{p+1}(x) = x H_p(x) - p H_{p-1}(x)
\end{align}

\noi
for any integer $p\geq 2$. Alternatively, one can define them via the 
Rodrigues' formula
\begin{align}\label{Rod}
\qquad\qquad\qquad
\text{$H_p(x) = (-1)^p \frac{1}{\phi(x) } \frac{d^p}{dx^p} \phi(x)$,
where $\phi(x) = \frac{1}{\sqrt{2\pi}} e^{-x^2/2}$}.
\qquad\quad
\end{align}

Then, using the recursive definition, one can show by induction
that 

\noi
\begin{align}\label{HerP1}
\frac{d}{dx} H_p(x) = p H_{p-1}(x),\,\, \forall p\in\N;
\end{align}

\noi
using Rodrigues' formula,
one can show that
for any jointly Gaussian distributed random variables $G_1, G_2\sim\NN(0,1)$,
we have 

\noi
\begin{align} \label{HerP2}
\E[H_q(G_1)H_p(G_2)]= \ind_{\{q=p\}}\,q!\,(\E[G_1 G_2])^q\,.
\end{align}
Indeed, with $q\geq p$,  $\rho = \E[G_1 G_2]$, and $G\sim\NN(0,1)$,
we can write by using Rodrigues' formula \eqref{Rod}
and integration by parts with \eqref{HerP1} that

\noi
\begin{align*}
\E[H_q(G_1)H_p(G_2)] 
&= \E \int_\R H_q(x) \phi(x)  H_p(\rho x  + \sqrt{1-\rho^2} G)dx \\
&= (-1)^q \E \int_\R    H_p(\rho x  + \sqrt{1-\rho^2} G) \phi^{(q)}(x) dx \\
&=  (-1)^{q+p} \rho^p p! \int_\R     \phi^{(q-p)}(x) dx,
\end{align*}
 which coincides with \eqref{HerP2}. Moreover, it is not difficult to see
 that every monomial $x^q$ can be written as a finite linear combination 
 of Hermite polynomials, which together with the above orthogonality 
 relation \eqref{HerP2}, indicates that $\{ H_p/\sqrt{p!}: p\in\N \cup\{0\} \}$
 is an orthonormal basis of $L^2(\R, \phi(x)dx)$.
 That is, for any $\varphi:\R\to\R$ square-integrable 
 with respect to the standard Gaussian measure $\phi(x)dx$, we have 
the following Hermite expansion in 
$L^2(\R,  \phi(x) dx)$:

\noi
\begin{align}\label{HerEXP0}
\qquad\qquad
\varphi = \sum_{k=0}^\infty a_k H_k
\quad{\rm with}
\quad a_k = \frac{1}{k!} \int_\R H_k(x) \varphi(x) \frac{1}{\sqrt{2\pi}} e^{-x^2/2}dx.
\end{align}
 In particular, we have $\sum_{k\geq 0} a_k^2 k! <\infty$.
 We say $\varphi$ has {\it Hermite rank} $R$ if $a_R \neq 0$
 and $a_k = 0$ for $1\leq k< R$.
 In this case, the Hermite rank of $\varphi - a_R H_R$
 is called the {\it second Hermite rank} of $\varphi$, denoted by $R'$.

\begin{example} \rm \label{Example1}
Consider $\varphi(x) =\ind_{\{ x \ge u\}}$ 
with a given $u\in\R$.
 One can easily deduce from Rodrigues' formula \eqref{Rod}  
 that  
 with $G\sim \NN(0,1)$,
 \[
  \E[\ind_{ \{ G \geq u\} }] 
  = \PP( G \geq  u)
 \quad{\rm and}
 \quad 
 a_q= \frac{1}{q!\sqrt{2\pi} } e^{-\frac{u^2}2} H_{q-1}(u) ,
 \]
and hence 
   $(R, R') =(1, 2)$ for the function 
   $x\mapsto \ind_{\{ x \ge u\}} - \E[\ind_{ \{ G \geq u\} }]  $
for   every $u\neq 0$
and $ (R, R') =(1, 3)$ for $u=0$ in view of \eqref{H1x}.

\end{example}

\medskip 

\noi
$\bul$ {\bf  Wiener chaos expansion.} Let $\{ I_1(h): h\in\fH\}$ be the isonormal
process as in  \eqref{e_x0},  \eqref{e_x1}, and \eqref{e_x2}. Then, the well-known
Wiener-It\^o chaotic decomposition asserts 
that the $L^2(\Omega)$ space generated by $\{ I_1(h): h\in\fH\}$
, or equivalently by $\bfB$, can be decomposed into 
mutually orthogonal closed subspaces.
That is, 

\noi
\begin{align}\label{WCE1}
L^2(\Omega, \s\{\bfB\}, \PP ) = \bigoplus_{q=0}^\infty \C_q,
\end{align}
where $\C_0\simeq \R$ denotes the set of constant random variables
and
$\C_q$ is called the $q$-th Wiener chaos that is the $L^2(\Omega)$-closure
of 
$
{\rm span}\big\{ H_q(I_1(h)) : \|h\|_{\fH} = 1\big\} .
$
See, e.g.,  \cite[Theorem 1.1.1]{Nua06} and \cite[Theorem 2.2.4]{NP12}.

Let $\fH^{\otimes p}$ denote the $p$-th tensor product of $\fH$
and let $\fH^{\otimes p}_{\rm sym}$ be its symmetric subspace. 
Letting $\{h_i: i\in\N\}$ denote a fixed orthonormal basis
of $\fH$,\footnote{In this paper, we use this orthonormal basis to define
various terms, such as the $r$-contractions in \eqref{def_SUO}
and the Mallavin derivative $DI_q(f_q) = q I_{q-1}(f_q)$
in \eqref{MD1b}. Note that these definitions do not depend 
on the particular choice of orthonormal basis of $\fH$, and we will
not repeat this point in this paper.}
$f\in\fH^{\otimes p}$ can be represented as

\noi
 \begin{align}\label{rep_F}
f = \sum_{i_1, ... , i_p\in\N} f_{i_1, ... , i_p} 
h_{i_1} \otimes \cdots \otimes h_{i_p},
\end{align}

\noi
and $f\in\fH^{\otimes p}_{\rm sym}$ if and only if 
$f_{i_{\tau(1)} , i_{\tau(2)}, ..., i_{\tau(p)}     } =  f_{i_1, ... , i_p} $
for any permutation $\tau$ on $\{1, ... , p\}$
and for any $i_1, ..., i_p\in\N$.
We denote by ${\rm sym}(f)$ the canonical symmetrization of 
$f\in\fH^{\otimes p}$:
\[
{\rm sym}(f) =  \sum_{i_1, ..., i_p}
\frac{1}{p!} \sum_{\tau\in\mathfrak{S}_p}
f_{i_{\tau(1)} , i_{\tau(2)}, ..., i_{\tau(p)}     }  
h_{i_1} \otimes \cdots \otimes h_{i_p},
\]
where $\mathfrak{S}_p$ denotes the set of 
permutations  on $\{1, ... , p\}$. 
By Jensen's inequality, 
we have

\noi
\begin{align}\label{sym_bdd}
\| {\rm sym}(f)  \|_{\fH^{\otimes p}} \leq \| f  \|_{\fH^{\otimes p}} 
\end{align}
for any $f\in\fH^{\otimes p}$.
We refer readers to  \cite[Appendix B]{NP12}
for the Hilbert space notations. 
For every $p\in\N$, there is a modified isometry, 
denoted by $I_p$,
from $\fH^{\otimes p}_{\rm sym}$ to the $p$-th Wiener chaos $\C_p$
characterized by

\noi
\begin{align}\label{e_x3}
I_p\big( {\rm sym}( \otimes_{i=1}^\infty h_i^{\otimes a_i}   )   \big)
= \prod_{i=1}^\infty H_{a_i}( I_1(h_i)  )
\end{align}

\noi
with $a_i\in\N\cup\{0\} $ such that $\sum_{i\geq 0} a_i = p$,
and $\{h_i: i\geq 1\}$ the orthonormal basis of $\fH$.
Then, one can write 
$
\C_p = \{ I_p(f) : f\in\fH^{\otimes p}_{\rm sym} \},
$
 while we also write $I_p(f)  = I_p(  {\rm sym}(f)  )$ 
 for general $f\in\fH^{\otimes p}$ and call it the $p$-th multiple 
 integral of $f$. Now we can rewrite \eqref{WCE1}
 as follows: for any $F\in L^2(\Omega, \s\{\bfB\}, \PP )$,
 one can find (unique) kernels $f_p\in\fH^{\otimes p}_{\rm sym}$, $p\geq 1$,
 such that

\noi
\begin{align}  \notag 
F = \E[ F] + \sum_{p\geq 1} I_p(f_p).
\end{align}

\noi
Then, we can view the Wiener chaos expansion
as an infinite-dimensional generalization of the 
Hermite expansion \eqref{HerEXP0},
where the latter only involves one Gaussian random variable. 
Moreover, the relation \eqref{HerP2} can be easily generalized,
by \eqref{e_x3} and a density argument,
 to

\noi
\begin{align}\label{chaos_iso}
\E[ I_p(f) I_q(g)] = p! \ind_{\{ p=q\}}
 \langle \textup{sym}(f),  \textup{sym}(g) \rangle_{\fH^{\otimes p}} 
\end{align}
for any  $f\in\fH^{\otimes p}$ and $g\in\fH^{\otimes q}$.

\begin{example} \label{Example2b}  \rm
Using  \eqref{e_x1} and \eqref{e_x3}, we can  write

\noi
\begin{align}
H_p(B_x) = I_p( e_x^{\otimes p}).  \label{HPB}
\end{align}

\noi
Then, letting $Y_t$ be defined as in \eqref{def_Y}
and $\varphi$ have the Hermite expansion \eqref{HerEXP0},
we can write 
 
 \noi
\begin{align} \label{GTQ}
\begin{aligned}
Y_t - \E[ Y_t]  
&=  \int_{t D} \sum_{q\geq 1} I_q( a_q e_x^{\otimes q})  dx \\
& =  \sum_{q\geq 1} I_q\bigg(  \int_{t D} a_q e_x^{\otimes q}  dx \bigg)
=:  \sum_{q\geq 1} I_q(g_{t, q}),
\end{aligned}
\end{align}

\noi
where the above interexchange of the multiple integral and 
the Lebesgue integral over $tD$ is guaranteed 
by a stochastic  version of Fubini's theorem;
see, e.g., \cite[Lemma 2.6 (ii)]{BZ24}.

 \end{example}
 
Before we present the Malliavin-Stein bound, let us first 
introduce an important tool  from the  Hilbert space theory
and state Nualart-Peccati's fourth moment theorem \cite{NP05} that 
is central in the Malliavin-Stein method.

\medskip

\noi
$\bul$ {\bf Contractions.}  Suppose 
$f\in\fH^{\otimes p}$ and $g\in\fH^{\otimes q}$,
with $p, q\in\N$,
are represented as in \eqref{rep_F}:
 \[
f = \sum_{i_1, ... , i_p\in\N} f_{i_1, ... , i_p} 
h_{i_1} \otimes \cdots \otimes h_{i_p}
\quad
{\rm and}
\quad
g = \sum_{i_1, ... , i_q\in\N} g_{i_1, ... , i_q} 
h_{i_1} \otimes \cdots \otimes h_{i_q},
\]
then the $r$-contraction of $f$ and $g$, with $r\in\{0, ... , p\wedge q\}$,
is defined by 
\begin{align}\label{def_SUO}
\begin{aligned}
f\otimes_r g
&=  \sum_{i_1, ... , i_{p-r}, k_1, ... , k_{q-r}}
\bigg( \sum_{j_1, ... , j_r\in\N} 
f_{j_1, ... , j_r, i_1, ... , i_{p-r}} g_{j_1, ... , j_r,  k_1, ... , k_{q-r}}  \bigg) \\
&\qquad\qquad\qquad\qquad
 h_{i_1}\otimes \cdots \otimes h_{i_{p-r}} \otimes 
  h_{k_1}\otimes \cdots \otimes h_{k_{q-r}} .
\end{aligned}
\end{align}
By Cauchy-Schwarz, one can derive easily that 

\noi
\begin{align}   \label{CS_bdd}
 \| f\otimes_r g \|_{\fH^{\otimes (p+q-2r)}}
  \leq \|f\|_{\fH^{\otimes p}}  \| g\|_{\fH^{\otimes q}}
\end{align}

\noi
for any $f\in\fH^{\otimes p}$ and $g\in\fH^{\otimes q}$, while it 
can also be easily verified that 

\noi
\begin{align} \label{C2b}
 \| f\otimes_r g \|_{\fH^{\otimes (p+q-2r)}}^2
 = \langle f\otimes_{p-r} f, g\otimes_{q-r} g \rangle_{\fH^{\otimes 2r}}.
\end{align}

If $g_{t,q}$ is defined as in \eqref{GTQ}, then the $r$-contraction of 
$g_{t,q}$ with itself is given by 

\noi
\begin{align}\label{GTQ_r}
\begin{aligned}
 g_{t,q} \otimes_r g_{t,q}
 & =a_q^2 \int_{tD\times tD}  e_x^{\otimes q} \otimes_r e_y^{\otimes q}  dx dy\\
 &= a_q^2 \int_{tD\times tD}  e_x^{\otimes q-r} \otimes  e_y^{\otimes q-r} \cC^r(x-y)  dx dy,
\end{aligned}
\end{align}

\noi
where we used \eqref{e_x2}, i.e., 
$\langle e_x, e_y\rangle_\fH = \cC(x-y)$.
By the definition of $\fH$,
we have 

\noi
\begin{align}  \label{GTQ_r2}
\begin{aligned}
&\|  g_{t,q} \otimes_r g_{t,q} \|^2_{\fH^{\otimes 2q-2r}} \\
& \qquad  =a_q^4 \int_{(tD)^4} \cC^{q-r}(x-z) \cC^{q-r}(y-w)
 \cC^r(x-y) \cC^r(z-w)  \, dx dy  dz dw.
\end{aligned}
\end{align}

  In particular, one deduces from \eqref{CS_bdd} 
  that 

  \noi
\begin{equation}\label{equ:important inequality}
\|  g_{t,q} \otimes_r g_{t,q} \|_{\fH^{\otimes 2q-2r}} 
  \leq \| g_{t,q}\|^2_{\fH^{\otimes q}}.
\end{equation}

\noi
The contractions appear naturally  when we 
represent the product of two multiple integrals 
as a linear combination of new multiple integrals.

\begin{lemma}\textup{(Product formula, \cite[Theorem 2.7.10]{NP12})} \label{product_f}
Suppose $F= I_p(f)$ and $G = I_q(g)$
for some $f\in \fH^{\otimes p}_{\rm sym}$
and $g\in \fH^{\otimes q}_{\rm sym}$, with $p, q\in\N$.
Then, 
\[
FG = \sum_{r=0}^{p\wedge q} r! \binom{p}{r} \binom{q}{r} 
I_{p+q-2r}\big(  \textup{sym}( f\otimes_r g) \big).
\]
\end{lemma}
In particular, taking $F = H_p(B_x)$ and $G= H_q(B_x)$ 
with \eqref{HPB} leads  to 
the following product formula for Hermite polynomials:

\noi
\begin{align*} 
H_p H_q =  \sum_{r=0}^{p\wedge q} r! \binom{p}{r} \binom{q}{r}  H_{p+q-2r}.
\end{align*}

\noi
This implies that 
\begin{align}\label{SeH}
\begin{aligned}
 \sum_{r=0}^{p\wedge q} \bigg[ r! \binom{p}{r} \binom{q}{r} \bigg]^2  (p+q-2r)!
 &= \E\big[ H_p(B_x)^2  H_q(B_x)^2  \big] \\
 &\leq  \sqrt{ \E\big[ H_p(B_x)^4\big]   \E\big[  H_q(B_x)^4   \big]} \\
 &\leq 3^{p+q} p! q!,
\end{aligned}
\end{align}

\noi
where in the last step, we used the Wiener chaos estimate 
$ \E[ H_p(B_x)^4 ] \leq 3^{2p}  \E[ H_p(B_x)^2]^2 = 3^{2p}(p!)^2$;
see, e.g., \cite[Corollary 2.8.14]{NP12}.
The above inequality \eqref{SeH} will be used to simplify
the expressions    \eqref{MPQ} in the proof of Proposition 
\ref{MS_bdd}.

\begin{theorem}\textup{(Fourth moment theorem, \cite{NP05})} \label{FMT_NP}
Suppose $F_n = I_p(f_n)$ is a sequence of random variables
 in $\C_p$, 
with $p\geq 2$ and $f_n\in\fH^{\otimes p}_{\rm sym}$, 
such that $\E[ F_n^2] = p! \| f_n\|^2_{\fH^{\otimes p}} \to 1$ as 
$n\to\infty$. Then, the following statements are equivalent:
\begin{itemize}
\item[(i)] $F_n$ converges in law to $Z\sim \NN(0,1)$ as $n\to\infty$;

\item[(ii)] $\E[ F^4_n]$ converges   to $ \E[Z^4]=3$ as $n\to\infty$;

\item[(iii)] $\| f_n\otimes_r f_n\|_{\fH^{\otimes {2p-2r}}} \to 0 $  
 as $n\to\infty$ for every $r = 1, ... , p-1$.

\end{itemize}

\end{theorem}

Note that the computation of fourth moment
$\E[ F^4_n]$ can be done by first
expanding $F_n^2$ using Lemma \ref{product_f}
and then expressing $\Var(F_n^2)$ using 
the orthogonality relation \eqref{chaos_iso},
and with some effort, one can show the equivalence
between (ii) and (iii). For the equivalence between
(i) and (iii), we refer interested readers
to the original paper \cite{NP05}
that utilized the random time change technique
and to \cite[Chapter 5]{NP12} for a
modern treatment using the 
Malliavin-Stein method.

\medskip

\noi
$\bul$ {\bf Malliavin-Stein bounds.} 
Since the publication of Nualart and Peccati's striking 
fourth moment theorem \cite{NP05} in 2005,
there have been many important progresses 
that explore the limit theorems on Wiener chaoses. 
In 2008, Nualart and Ortiz-Latorre 
provided in \cite{NOL08} another equivalent condition 
for the fourth moment 
theorem:

\noi
\begin{align}  \notag 
\textup{(iv)} \quad
\Var(  \| DF_n \|^2_{\fH} )  \to 0
\quad\text{as $n\to\infty$,}
\end{align}

\noi
where $DF_n$ denotes the Malliavin derivative of $F_n$;
see \eqref{MD1a}    and \eqref{DYinf} for the definition of $D$.
The paper  \cite{NOL08}  contains a methodological breakthrough:  
the authors used Malliavin calculus tools,
 in particular the Gaussian  integration by parts formula
 to derive the central limit theorems. 
 On the other hand, Stein's lemma asserts that 
 for an integrable random variable $Z$, we have 
 $Z\sim\NN(0,1)$ if and only if $\E[ f(Z)Z] = \E[ f'(Z)] $
 for any  differentiable $f$ with $ \E[ | f'(Z)|] <\infty$.
 This is nothing else but a simple case of 
 the Gaussian integration by parts formula,
 and is  one of the fundamental blocks in Stein's method -- a powerful
  toolbox for establishing quantitative central limit theorems. 
 In a 2009 paper \cite{NP09}, Nourdin and Peccati combined 
 Malliavin calculus and Stein's method for the very first time
 in order to quantify the above fourth moment theorem (Theorem \ref{FMT_NP}), and 
 then they   unraveled a new research field, nowadays known
 as the Malliavin-Stein method or Nourdin-Peccati analysis;
 see \cite{NP12} for a comprehensive treatment and
 see also \cite[Chapter 1]{GZthesis} for a brief historical
 account. 

\medskip

In the following, we first briefly introduce several 
operators in Malliavin calculus and then state the 
Malliavin-Stein bounds. Note that
for our purpose, we only define  most of these operators on   
 finitely many chaoses 
while we refer interested readers to the books
\cite{Nua06, NP12} for the general theory.

\begin{definition} 
Suppose that $Y$ has a finite chaotic decomposition:

\noi
\begin{equation}
    \label{chaos_Y}
    Y=\E[Y] + \sum_{q=1}^N I_q(f_q)
\end{equation}

\noi
with $f_q\in\fH^{\otimes q}_{\rm sym}$ for  $1 \leq q \leq N < \infty$.
Then, 
its Malliavin derivative $DY$ is a $\fH$-valued random variable 
defined by 

\noi
\begin{align}\label{MD1a}
DY =\sum_{q=1}^N q\, I_{q-1}(f_q),
\end{align}

\noi
where

 \noi
\begin{align}\label{MD1b}
 I_{q-1}(f_q) : =  \sum_{i\in\N } I_{q-1}( f_q\otimes_1 h_i) h_i
\end{align}
with $\{ h_i: i\in\N \}$ the orthonormal basis of $\fH$
and $\otimes_1$   the $1$-contraction defined as in \eqref{def_SUO}.
We also define  the Ornstein-Uhlenbeck operator $L$
and its    pseudo-inverse $L^{-1}$ by setting
\[
LY = \sum_{q=1}^N -q I_q(f_q)
\quad
{\rm and}
\quad
L^{-1}Y =  \sum_{q=1}^N -\frac1q I_q(f_q).
\]
In particular, one has $LL^{-1}Y  = L^{-1} L Y = Y - \E[Y]$.

\end{definition}

It is clear that with $Y$ as in \eqref{chaos_Y},
\begin{align*} 
\E\big[ \| DY \|_\fH^2 \big]
= \sum_{q=1}^N  q \E\big[  ( I_q(f_q) )^2 \big].
\end{align*}

\noi
For a square-integrable random variable $F$ with an
infinite chaos expansion of the form \eqref{chaos_Y}
with $N = +\infty$, 
if  $
\sum_{q\ge 1}  q \E\big[  ( I_q(f_q) )^2 \big] 
<+\infty$,
we say that $Y\in \mathbb D^{1,2}$ (i.e., $Y$ is Malliavin differentiable) and we define its Malliavin derivative by
\begin{equation}\label{DYinf}
DY = \sum_{q\ge 1} q\, I_{q-1}(f_q).
\end{equation}
When  $Y$ has a finite chaotic decomposition as in \eqref{chaos_Y},    $Y\in\mathbb D^{1,2}$. 
More generally, 
we say that 
\begin{align} \label{DK2}
\text{$Y\in \mathbb D^{k,2}$ for $k>0$ if 
$
\sum_{q\ge 1}  q^k \E\big[  ( I_q(f_q) )^2 \big] <+\infty$.}
\end{align}
Now we are ready to state the following Malliavin-Stein bounds.

\begin{proposition}\label{MS_bdd}
{\rm (i)}    Let $Y$ be a centered random variable with   variance $\s^2 >0$
    and a finite chaotic decomposition as in \eqref{chaos_Y}.
    Then, with  $Z\sim N(0,1)$,
    
    \noi
    \begin{align}\label{MSbdd1}
      \textup{dist}(Y, Z) \leq 2 | 1- \s^2| + 
      2 \sqrt{\Var\big(\langle DY,-DL^{-1}Y \rangle_{\fH} \big)},
    \end{align}

 \noi
 where the distributional distance $\textup{dist}$ can be 
 the total-variation distance $d_{\rm TV}$ and $W_p$ metric with $p\in[1,2]$
 in \eqref{def_dist}.

    \smallskip
    
    \noi
    {\rm (ii)} With $Y$ as in  \eqref{chaos_Y} such that $\s^2=1$, 
    one has
    
    \noi
    \begin{align} \label{MSbdd2}
     \begin{aligned}
    	\sqrt{\Var\big(  \langle DY,-DL^{-1}Y' \rangle_{\fH}   \big)}
        &        \leq 2N\sum_{p=1}^N  
    	3^{2p} p! \M_{p} ,
    \end{aligned}
    \end{align}
  
 \noi
   where
   \begin{align}\label{def_MP}
 \M_{p} = \max_{1\leq r \leq p-1}  \| f_p \otimes_r f_p \|_{\fH^{\otimes 2p-2r}}.
\end{align}

     \smallskip
    
    \noi
    {\rm (iii)} Suppose  additionally $Y' = \sum_{q=1}^{N'} I_q( g_q)$ 
    has unit variance
     with $g_q\in\fH^{\otimes q}_{\rm sym}$
     and $\s^2 =1$.
Then,       we get 
      
        \noi
    \begin{align} \label{MSbdd3}
    \begin{aligned}
    & \sqrt{\Var\big(  \langle DY,-DL^{-1}Y' \rangle_{\fH}   \big)}
      \le \sqrt{N'}\sum_{q=1}^{N'} 
   3^{2q} q!\M'_{q} 	 + N\sum_{p=1}^N  
   3^{2p} p! \M_{p} ,
   \end{aligned}
    \end{align}

  \noi
   where
   \begin{align}\label{def_MQ}
 \M'_{q} = \max_{1\leq r \leq q-1}  \| g_q \otimes_r g_q \|_{\fH^{\otimes 2q-2r}}.
\end{align}

\end{proposition}

The proof of Proposition \ref{MS_bdd} is postponed to Section \ref{S5}.

\subsection{Auxiliary results  on ASCLT} \label{SEC_aux}

In this section, we state a few auxiliary results  on $\ASCLT$.
The first result  (Lemma \ref{lem_red})
is useful in
reducing the proof of $\ASCLT$
to the case of finitely many chaoses; see
Lemma \ref{Red_step}.

\begin{lemma}\label{lem_red}
Let $F_t = M_t + R_t$ for $t\geq t_0$
for some given $t_0 > 0$.
Suppose $\{M_t: t\geq t_0\}$ satisfies the {\rm ASCLT} 
and 
there exists some   $\theta\in ( 0, 1)$ such that 
$\E( | R_t|) \les (\log t)^{-\theta}$ for all $t\geq t_0$.
Then,  $\{F_t: t\geq t_0\}$ also satisfies the $\ASCLT$.

\end{lemma}

\begin{proof}

Let $\phi: \R\to\R$ be any bounded Lipschitz function, 
then in view of Remark \ref{equivdef}, 
we can assume $t_0 = 1$ and 
 we  need to show 

\noi
\begin{align}
\frac{1}{\log T} \int_1^T \frac{1}{t} \phi(M_t + R_t)dt 
\xrightarrow[a.s.]{T\to+\infty} \int_\R \phi(x) \frac{1}{\sqrt{2\pi}}
e^{-\frac{x^2}2}dx. 
\label{Y1}
\end{align}

\noi
Since  \eqref{Y1} holds with $M_t+R_t$ replaced by 
$M_t$, it suffices to show that 

\noi
\begin{align*}
L_T:=
\frac{1}{\log T} 
\int_1^T \frac{1}{t} \big[\phi(M_t + R_t) - \phi(M_t)\big] dt 
\xrightarrow[a.s.]{T\to+\infty} 0.  
\end{align*}

\noi
Using $\E( | R_t|) \les (\log t)^{-\theta}$ with $\theta>0$
and the fact that $\phi$ is bounded Lipschitz, 
we can obtain

\noi
\begin{align*}
\begin{aligned}
  \E\big[ | L_{2^{K^m }} |  \big]  
&\les \frac{1}{K^{m}} \int_{1}^{2^{K^m}} \frac{1}{t \log^\theta t} dt \\
&= \frac{1}{K^{m}} \int_{0}^{K^m \log 2} \frac{1}{ y^\theta } dy
\quad\text{with $y = \log t$}\\
&\les K^{-m\theta},
\end{aligned}
\end{align*}

\noi
where $m\in\N$ is large enough such that $m\theta > 1$.
It follows that

\noi
\begin{align*}
   \E \sum_{K=0}^\infty | L_{2^{K^m}} |
\les\sum_{K=0}^\infty   K^{-m\theta} < +\infty,
\end{align*}

\noi
and thus,   $L_T$ converges almost surely to zero
along $T = 2^{K^m}$ with $K\in\N\to+\infty$.
Then, for any $T \geq 1$, one can find a unique nonnegative integer
$K = K_T$ such that 

\noi
\begin{align}  \notag 
2^{K^m} \leq T < 2^{(K+1)^m}.
\end{align}

\noi
Due to the boundedness of $\phi$, we can proceed as follows:

\noi
\begin{align} \notag 
L_T = \frac{1}{ \log T} \int_{2^{K^m}}^{T} 
\frac{1}{t} \big[\phi(M_t + R_t) - \phi(M_t)\big] dt 
+ \frac{\log 2^{K^m}}{\log T} L_{2^{K^m}},
\end{align}

\noi
where the second summand tends to zero almost surely
as $K\to\infty$ and the first term is uniformly bounded by 

\noi
\begin{align*}
   \frac{2\|\phi\|_\infty}{\log T} \big( \log T - \log 2^{K^m} \big)
 \leq  \frac{2\|\phi\|_\infty}{K^m} 
 \big[ (K+1)^m - K^m \big] ,
\end{align*}
which tends to zero as $K = K_T\to+\infty$.
Hence the proof of Lemma \ref{lem_red} is completed.
\qedhere

\end{proof}

\begin{proposition}[Ibragimov-Lifshits' criterion]
\label{prop:IL} 

A family of real-valued random variables $\{  F_t\}_{t\geq 1}$ 
satisfies the $\ASCLT$
if $t\in[1,\infty)\mapsto   F_t$ is measurable almost surely, 
and
the following inequality holds

\noi
\begin{align} 
\sup_{|s| \leq T} 
\int_2^\infty
\frac{\E \big[  |K_r(s)|^2\big]  }{r  \log r} dr <\infty   
\notag 
\end{align}

\noi
for any finite $T >0$, where 

\noi
\begin{align}\notag 
K_r(s) 
:= \frac{1}{\log r} \int_1^r 
\big( e^{is   F_t} - e^{-\frac{s^2}2} \big)  \frac{dt}{t},
\quad r \in( 1, \infty).
\end{align}

\end{proposition}

In Ibragimov-Lifshits' original paper \cite{IL99},
their criterion is stated for the  discrete sum,
we refer interested readers to \cite[Proposition 3.3]{BXZ23}
for the above integral version and see Appendix A therein
for a proof.  In what follows, 
we combine  Ibragimov-Lifshits' method of characteristic 
functions with the above Malliavin-Stein bounds
to derive conditions to ensure the $\ASCLT$. 
Note that this combination was first established in 
the paper \cite{BNT} by Bercu, Nourdin, and Taqqu.
See also \cite{CZthesis, Zheng17,AN19, BXZ23, XZ24}
with the same flavor. 
Note that the Ibragimov-Lifshits'  criterion 
essentially requires 
a quantitative control
over the difference of characteristic functions 
of the random probability $\nu_T^\omega$ in (1.17)
and of the standard normal.
Concretely, we need  to control
 the above quantity $\mathbb{E}\big[ |K_r(s)|^2 \big]$
 and as we shall see shortly in the proof of Proposition 
 \ref{prop210}, the estimation of $\mathbb{E}\big[ |K_r(s)|^2 \big]$
 can be easily reduced to
 bounding the  Wasserstein distance $W_1(F_t, \NN(0,1))$.
 To this end, the Malliavin-Stein method enters the picture naturally.

\begin{proposition}\label{prop210}
Suppose that  $F_t$ is a real centered random variable 
with   variance one and  finite chaotic expansion
for each $t \geq 1$.
Assume that there is some constant $X \geq  e$ such that 
the following conditions hold:
\begin{itemize}

\item[\bf (a)]  for $t_2 \geq t_1 \geq X$ and some $\beta_1>0$, we have 
  \[
  \Cov(F_{t_1}, F_{t_2} ) 
\les  \left(\frac{t_1}{t_2}\right)^{\be_1}  ;
\]

\item[\bf (b)]  for any $t_1, t_2\in [X, \infty)$ and some $\beta_2,\beta_3\in(0,1)$, we have  
\[
\Var \big( \langle DF_{t_1}, -DL^{-1} F_{t_2} \rangle_\fH \big)
\les   \log^{-\be_2}(t_1) 
+  \log^{-\be_3}(t_2)\,. 
\]
\end{itemize}
Then, $\{F_t: t\geq 1\}$ satisfies the   $\ASCLT$.

\end{proposition} 

Later, we will apply the above proposition with
 $F_t$  
the normalized truncation of $Y_t$ up to finitely many chaoses 
as in \eqref{def_FT}. That is why we only stick to the case
where $F_t$ has finite chaotic expansion, while 
to keep the notation light, we have only introduced 
the Malliavin operators for random variables with 
finite chaotic expansions.
Of course, the above result can be easily generalized to 
the case where $F_t$ is twice Malliavin differentiable
with $\E[ \| DF_t \|^4_\fH ] <\infty$;
see, e.g., \cite{BNT}.
Note also that the restriction $X\geq e$ is immaterial 
and kept for convenience, and it is used, e.g., in 
\eqref{X_e} and \eqref{be45_y}.

\begin{proof}[Proof of Proposition \ref{prop210}]
In view of Ibragimov-Lifshits' criterion (Proposition \ref{prop:IL}),
it suffices to show that 
\[
\E \big[  |K_r(s)|^2\big] \les \frac{1}{\log^{\be_0}( r)}
\]
uniformly in $s\in\R_+$, for some   constant $\be_0 > 0$.
For any given $X > 0$, it is easy to see that 
\[
\Big| \frac{1}{\log r} \int_1^X 
\big( e^{is   F_t} - e^{-\frac{s^2}2} \big)  \frac{dt}{t} \Big|
\leq \frac{2\log X}{\log r},
\]
so that it is enough to show 
\begin{align}\label{eng1}
\E\bigg[  \bigg|  \frac{1}{\log r} \int_X^r 
\big( e^{is   F_t} - e^{-\frac{s^2}2} \big)  \frac{dt}{t}  \bigg|^2 \bigg]
\les  \frac{1}{\log^{\be_0}( r)}.
\end{align}

It is standard to expand the square in \eqref{eng1}
and relate the resulting expressions to the distances \eqref{def_dist},
from which one can apply the Malliavin-Stein bounds; 
see, e.g., \cite[Theorem 3.2]{BNT} and \cite[Section 3.2]{BXZ23}.
Indeed, we can first express left side of \eqref{eng1}
as 
\begin{align*}
\frac{1}{\log^2(r)} \int_{X}^r  \int_{X}^r \frac{1}{t_1t_2} 
\E\big[ e^{is (F_{t_1} -F_{t_2} ) }
& - e^{-s^2}  \big] dt_1 dt_2
 - e^{-\frac{s^2}2} \frac{\log r -\log X}{\log^2 (r)} 
   \int_{X}^r \frac{1}{t_1} 
\E\big[ e^{is F_{t_1} }
 - e^{-\frac{s^2}{2}}  \big] dt_1   \\
 &- e^{-\frac{s^2}2} \frac{\log r -\log X}{\log^2 (r)} 
   \int_{X}^r \frac{1}{t_2} 
\E\big[ e^{-is F_{t_2} }
 - e^{-\frac{s^2}{2}}  \big] dt_2 ,
\end{align*}

\noi
which can be further bounded by

\noi
\begin{align*}
\frac{4}{\log^2(r)} \int_{X}^r  \int_{X}^r \frac{1}{t_1t_2} 
d_{\rm TV}( \tfrac{F_{t_1} -F_{t_2} }{\sqrt{2} } , Z ) dt_1 dt_2
+\frac{2}{\log (r)} 
   \int_{X}^r \frac{1}{t} 
d_{\rm TV}( F_{t}    , Z ) dt
:=\T_1 + \T_2 ,
\end{align*}

\noi
 using $ | \E[ e^{is Y_1} -e^{is Y_2}] | \leq 4 d_{\rm TV}( Y_1, Y_2)$.

Now using the Malliavin-Stein bound \eqref{MSbdd1}
with condition (b), we get   
\begin{align}\label{X_e}
\T_2 \les
 \frac{1}{\log (r)} 
   \int_{X}^r \frac{1}{t} 
 \log^{-\min\{\be_2,\be_3\}}( t) dt \les   \frac{1}{\log^{\min\{\be_2,\be_3\}} (r)} .
\end{align}

\noi
To estimate the first term $\T_1$, we first 
observe that 
$\Var\big( \tfrac{F_{t_1} -F_{t_2} }{\sqrt{2} } \big) = 1 - \Cov(F_{t_1} , F_{t_2} ) $
and then we apply the Malliavin-Stein bound \eqref{MSbdd1}
to get

\noi
\begin{align}\label{term_T1}
\begin{aligned}
\T_1 &
\les 
\frac{1}{\log^2(r)} \int_{X}^r  \int_{X}^r \frac{1}{t_1t_2} 
|   \Cov(F_{t_1} , F_{t_2} )  |  dt_1 dt_2 \\
&\qquad\qquad + \frac{1}{\log^2(r)} \int_{X}^r  \int_{X}^r \frac{1}{t_1t_2} 
 \sum_{i, j=1}^2 \sqrt{ \Var \big(  \jb{DF_{t_i},  -DL^{-1} F_{t_j} }_{\fH}  }
 \,\,  dt_1 dt_2, 
\end{aligned}  
\end{align}
where the second term can be easily dealt with
using condition (b).
Indeed, with $\be_2\in (0,1)$,
one can easily see   that 

\noi
\begin{align*}
\frac{1}{\log^2(r)} \int_{X}^r  \int_{X}^r \frac{1}{t_1t_2} 
  \log^{-\be_2}(t_1)      dt_1 dt_2
& \les \frac{1}{\log^{\beta_2}(r)}\,,
\end{align*}

\noi
and in this way, we can get
\begin{align}\label{be45_y}
\begin{aligned}
&\frac{1}{\log^2(r)} \int_{X}^r  \int_{X}^r \frac{1}{t_1t_2} 
 \sum_{i, j=1}^2 \sqrt{ \Var \big(  \jb{DF_{t_i},  -DL^{-1} F_{t_j} }_{\fH}  }
 \,\,  dt_1 dt_2 \\
 &\qquad\qquad\qquad
 \les   \frac{1}{\log^{\be_2}(r)} + \frac{1}{\log^{\be_3}(r)}\les \frac{1}{\log^{\min\{\beta_2,\beta_3\}}(r)} .
\end{aligned}
\end{align}
\noi
And the first term in \eqref{term_T1} can be estimated using condition (a) as follows:

\noi
\begin{align*} 
\begin{aligned}
    &\frac{1}{\log^2(r)}\int_X^r\int_X^r
      \big| \Cov(F_{t_1}, F_{t_2} )  \big|  \frac{dt_1}{t_1}\frac{dt_2}{t_2} 
    \les 
    \frac{1}{\log^2(r)}
    \int_X^r\frac{dt_2}{t_2}\int_X^{t_2}\frac{dt_1}{t_1}
     \Big(\frac{t_1}{t_2} \Big)^{\be_1}\,\les\frac{1}{\log(r)}\,.
\end{aligned}
\end{align*}

Therefore, combining the above estimates for $\T_1$ and $\T_2$,
we can see that 
\eqref{eng1} holds with $\be_0 =\min\{ \be_2,\be_3\} \in(0,1)$. 
Hence 
$\{F_t: t\geq 1\}$ satisfies the   $\ASCLT$.
\qedhere

\end{proof}

\section{Main proofs}\label{SEC3}

\subsection{Proofs of QCLT and ASCLT }\label{SEC3_1}

As announced in our introduction, we will first 
reduce the problem to  the case of finitely many chaoses
by using Lemma \ref{lem_red} 
and then prove the $\ASCLT$ for the truncated version.
As we will see shortly, the obtention of the {\bf\textsf{QCLT}} 
(i.e., the bound \eqref{QCLT_bdd1} in part {\bf (1)}) 
is a by-product of this argument.

Let us recall from \eqref{def_Y} 
that 
\[
Y_t  - \E[Y_t]= \sum_{q= R}^\infty \int_{tD} a_q H_q(B_x) dx,
\] 
where  $R \geq 1$ is the Hermite rank of $\varphi$, and we 
truncate the above series up to $N$:
\begin{align}\label{chaos_YTN}
Y_{t, N} =  \sum_{q= R}^N \int_{tD} a_q H_q(B_x) dx.
\end{align}

\noi
We put $\s_t^2 =\Var( Y_t)$ and $\s_{t, N}^2 = \Var( Y_{t, N})$, i.e.,

\noi
\begin{align}\label{s_tn}
\s^2_{t, N} =  \sum_{q= R}^N  a^2_q q! \int_{(tD)^2} \cC^q(x-y)  dxdy
\end{align}
 and $\s^2_t = \s^2_{t,\infty} \in [0,\infty]$.  
Later, we will let $N = N_t$ depend on $t$ 
as in \eqref{def_NT}.

 \begin{lemma}[Reduction step] \label{Red_step}
 Suppose that the Conditions \ref{cond5},  \ref{cond6}, and  \ref{cond7} hold.
 Let $Y_t,  Y_{t, N}$ be given as above. 
 Then, there is some integer $M\geq 1$ independent of $t$
 and there is some $t_0 > 0$ 
  such that 
 
 \noi
 \begin{align} \label{red1a}
\E\Big[  \big| \frac{1}{\s_t} Y_t -  \frac{1}{\s_{t, N}}  Y_{t, N} \big|^2 \Big]
   \les  N^{-\frac{d}{\al}}
    \end{align}
 for any $N\geq M$ and $t\geq t_0$.
 
 \end{lemma}
 
 \begin{proof}[Proof of Lemma \ref{Red_step}]
 By simple algebra and \eqref{s_tn}
 with $|\cC| \leq 1$, 
 we can  write 
 
 \noi
 \begin{align*}
 \E\Big[  \big| \tfrac{1}{\s_t} Y_t -  \tfrac{1}{\s_{t, N}}  Y_{t, N} \big|^2 \Big]
&\leq 2     \big\| \tfrac{1}{\s_t} (Y_t  -  Y_{t, N} ) \big\|_2^2
+  2 \big\| ( \tfrac{1}{\s_t}  -   \tfrac{1}{\s_{t, N}})  Y_{t, N} \big\|_2^2  \\
& \leq 2 \frac{\s_t^2 - \s^2_{t, N}}{\s_t^2} 
+ 2 \frac{ (\s_t - \s_{t, N})^2}{\s_t^2} = 4  \frac{\s_t - \s_{t, N}}{\s_t} \\
&\leq 4  \frac{\s^2_t - \s^2_{t, N}}{\s^2_t} 
= \frac{4}{\s_t^2} \sum_{q=N+1}^\infty  a^2_q q! \int_{(tD)^2} \cC^q(x-y)  dxdy\\
&\leq  \bigg( 4 \sum_{q=N+1}^\infty  a^2_q q! \bigg)
\frac{\Vol(tD)}{\s_t^2} \int_{\R} | \cC(z)|^N dz \\
&\les \frac{\Vol(tD)}{\s_t^2} \int_{\R} | \cC(z)|^N dz, 
 \end{align*}
 
 \noi
where we used  the fact $\sum_{q=R}^\infty  a^2_q q! <\infty$
in the last step. Note that
Conditions \ref{cond7} and   \ref{cond5} ensure
the existence of $t_0 > 0$ such that 
$\s_t^2 \ges t^d$ for any $t\geq t_0$ (see Lemma \ref{tech2}), 
while Conditions \ref{cond5} and \ref{cond6} imply that 
$ \int_{\R} | \cC(z)|^N dz \les N^{-d/\al}$ for any $N \geq \frac{d}\dl + 1$
(see Lemma \ref{tech1}).
Hence, the bound \eqref{red1a} follows immediately. 
\qedhere

 \end{proof}
 
 In view of the bound  \eqref{red1a}
 and Lemma \ref{lem_red}, in order to show the $\ASCLT$
 for $Y_t/\s_t $, it suffices to show
 that for 
 
\noi
 \begin{align} \label{def_FT}
 F_t := \frac{1}{\s_{t, N_t}} Y_{t, N_t} 
 \quad\text{with  $N_t \sim \log^\theta(t)$}
 \end{align}
 for some   $\theta > 0$. In fact, 
 we will prove it for $\theta$
subject to the restriction \eqref{th0}.
In particular, $N_t$ can be chosen larger if 
the Gaussian field $\bfB$ exhibits stronger local independence;
consequently, the approximation of $Y_t$ 
by the truncation $Y_{t,N_t}$ improves 
as the local parameter $\alpha$ decreases (see Condition 1.6).

 Note that under the Condition \ref{cond5} and Condition \ref{cond7},
Lemma \ref{tech2} implies the existence of
 two constants $t_0, M >0$ such that
 $\s_{t, N} > 0$ for $t\geq t_0$ and $N\geq M$.
 In this case, the random variable $F_t$ is well defined for large $t$,
 and for our purpose of proving the $\ASCLT$, 
 we can just assume $F_t = 1$ for $t < t_0$ (see Remark \ref{equivdef}) 
 and we will 
 not  mention this point any more.

\begin{proposition}\label{prop_FT}

Let $\{F_t: t\geq 1\}$ be defined as in \eqref{def_FT}.
Suppose that Conditions 
\ref{cond5},
\ref{cond6},
\ref{cond7},
and \ref{cond8}
hold. 
Then, $\{F_t: t\geq 1\}$ satisfies the $\ASCLT$.

\end{proposition}
 
  \begin{proof} In view of   Proposition \ref{prop210},
  we need to verify the conditions {\bf (a)-(b)} therein. 
   In what follows, we 
  will use the fact $\sum_{k\geq R}a_k^2 k! < \infty$
  for several times. 
  
\medskip
\noi
$\bul$  {\bf Verification of condition  (a).} Assume that $t_2\geq t_1$.
First of all, we deduce from the Hermite expansion 
\eqref{chaos_YTN} and the orthogonality relation
\eqref{HerP2} that 

\noi
\begin{align} \label{FT1}
\Cov \big( Y_{t_1,N_{t_1}}, Y_{t_2,N_{t_2}} \big)
= \sum_{k=R}^{N_{t_1}} k!a_k^2 \int_{t_1D} \int_{t_2 D} 
\cC^k(x-y) dxdy.
\end{align}

%

Next, we will consider three situations in Condition \ref{cond7} separately. 

\medskip

$\underline{\textbf{In case   (c1)}}$, we assumed that $\cC\in L^R(\R^d)$.
Together with $|\cC| \leq 1$, we deduce that 
$\| \cC\|_{L^k(\R^d)}^k \leq  \| \cC\|_{L^R(\R^d)}^R \les  1$
for any $k\geq R$, 
so that 
we have

\noi
\noi
\begin{align} \label{FT1a}
\big| \Cov \big( Y_{t_1,N_{t_1}}, Y_{t_2,N_{t_2}} \big) \big|
&\leq \sum_{k=R}^{N_{t_1}} k!a_k^2 \Vol(t_1 D)  \| \cC\|_{L^R(\R^d)}^R 
\les  t_1^d.
 \end{align}
By the same argument together with dominated convergence 
theorem, we can see that 
$\s^2_{t, N_t} \asymp t^d$; 
see also similar arguments later in  \eqref{routine1}-\eqref{routine2}.
Then, it follows from \eqref{FT1a}
that 
\begin{align*}   
\big| \Cov  ( F_{t_1} ,  F_{t_2} ) \big|
\les  (t_1/t_2)^{\frac{d}{2}}.
\end{align*}

\medskip

$\underline{\textbf{In case   (c2)}}$,
 we assumed that $\cC^R \geq 0$
and 
\begin{align}\label{a_rv1}
r\in\R_+\mapsto w_{r, R} = a_R^2 R! \int_{|z|\leq r} \cC^R(z)dz
\,\,\,\, \text{is regularly varying with index $\rho\in[0, d)$. }
\end{align}

\noi
First, we deduce from \eqref{FT1} and $\cC^R \geq 0$
with $|\cC| \leq1$ and $\sum_{k\geq R} k! a_k^2 <\infty$
that for $t_2\ge t_1$

\noi
\begin{align}\label{a_rv2}
\begin{aligned}
\big| \Cov \big( Y_{t_1,N_{t_1}}, Y_{t_2,N_{t_2}} \big) \big|
&\leq  \sum_{k=R}^{N_{t_1}} k!a_k^2 \int_{t_1D} \int_{t_2 D} 
\cC^R(x-y) dxdy \\
& \les t_1^d   a_R^2 R!  \int_{\lbrace |z|\le 2\mathfrak{m}t_2\rbrace} 
\cC^R(z) dz = t_1^d   w_{2\mathfrak{m}t_2, R} \\
&\les  t_1^d   w_{t_2, R}, 
\end{aligned}
\end{align}

\noi
where in the last step, we used the assumption \eqref{a_rv1}
and the fact that
 $|z|=|x-y|\le |x|+|y|\le 2 \mathfrak{m} t_2$
  for $x\in t_1D_1$ and $y\in t_2D_2$, where $\mathfrak{m} := \sup\{|x|  :  x\in D\}<\infty$.
Therefore, it follows from \eqref{a_rv2} and 
Lemma \ref{tech2} (in particular \eqref{expand2})
that  
 \[
\big| \Cov  ( F_{t_1} ,  F_{t_2} ) \big|
\les \frac{ t_1^d   w_{t_2, R} }{\sqrt{ t_1^d w_{t_1, R}  t_2^d   w_{t_2, R}} }
=  (t_1/t_2)^{\frac{d}{2}}  \Big( \frac{w_{t_2, R}}{w_{t_1, R}} \Big)^{\frac{1}{2}}.
\]
 By Potter's bound  in 
 Lemma \ref{lemP}, one has 
 \[
  \frac{w_{t_2, R}}{w_{t_1, R}} \les (t_2/t_1)^{\rho'}
 \]
 with $\rho' \in (\rho, d)$ for $t_2 \geq t_1 \geq X$
 for some sufficiently large $X > 0$.
  Thus, we have 
  
  \noi
\begin{align*}  
\big| \Cov  ( F_{t_1} ,  F_{t_2} ) \big|
\les   (t_1/t_2)^{\frac{d-\rho'}2}. 
\end{align*}

 \medskip
 
$\underline{\textbf{In case   (c3)}}$,
 we assumed that $D$ is a centered closed ball and there
 is some integer $M > \frac{d}{\dl} -1$ (with $\dl$ as in Condition \ref{cond5})
 such that $w_{r, M} \to w_{\infty, M}\in (0, \infty]$ as $r\to+\infty$.
 Then, we can apply Lemma \ref{tech3}
 with 
 \begin{align}  \label{a_rv3a}
 K(z) = \sum_{k=R}^M  k!a_k^2   \cC^k(z)
 \quad
 {\rm and}
 \quad
 w_t = w_{t,M} = \int_{\{ |z| \leq t \} }  K(z)dz
 \end{align}
 to get 
 \begin{align} \label{a_rv3b}
 \int_{t_1D} \int_{t_2 D} 
K(x-y) dxdy  \les t_1^d \,w_{t_2, M} ,
 \end{align}
 
 \noi
 where $t_2\geq t_1 \geq X$ for some sufficiently large $X > 0$.
 Meanwhile, Condition \ref{cond5} implies that 
 $\cC\in L^{M+1}(\R^d)$ so that 

\noi
 \begin{align} \label{a_rv3c}
 \sum_{k=M+1}^{N_{t_1}} k!a_k^2 \int_{t_1D} \int_{t_2 D} 
\cC^k(x-y) dxdy 
\les t_1^d.
 \end{align}
 Thus, combining \eqref{a_rv3a},  \eqref{a_rv3b}, \eqref{a_rv3c},
 and \eqref{FT1}, we get 
 
 \noi
\begin{align*}  
   \big| \Cov ( Y_{t_1,N_{t_1}}, Y_{t_2,N_{t_2}} )\big|
\les t_1^d\,w_{t_2,M}\, ,
\quad  t_2\ge t_1\ge X.
\end{align*}
Similarly as in the case (c2), 
we can use $\s^2_{t, N_t} \asymp t^d w_{t, M}$ 
(see Lemma \ref{tech2}) and Potter's bound, 
so condition {\bf (a)} of Proposition \ref{prop210}
is also verified  in this case. 

\medskip

Therefore, we have verified the condition {\bf (a)} in all the three cases (c1)-(c2)-(c3).

\medskip
\noi
$\bul$  {\bf Verification of condition  (b).}
Note that 
\[
F_t = \sum_{k=R}^{N_t} 
 I_k \bigg( \frac{a_k}{\s_{t, N_t} } \int_{tD} e_x^{\otimes k} dx \bigg)
 = :  \sum_{k=R}^{N_t} 
 I_k ( \widehat{g}_{t, k} ).
\]
Then, as in \eqref{GTQ} and \eqref{GTQ_r},
we get for $r\in\{1, ... , k-1\}$:

\noi
\begin{align} \notag 
\begin{aligned}
&\big\|\widehat{g}_{t, k}  \otimes_r \widehat{g}_{t, k} 
 \big\|_{\fH^{\otimes (2k-2r)}}
= \frac{a_k^2}{\s_{t, N_t}^2} \Big\| \int_{(tD)^2} 
\cC^r(x-y) e_x^{\otimes k-r} \otimes  e_y^{\otimes k-r} dxdy 
      \Big\|_{\fH^{\otimes (2k-2r)}} \\
&= \frac{a_k^2}{\s_{t, N_t}^2}
\bigg(   \int_{(tD)^4} \cC^r(x-y)   \cC^r(z-w)  \cC^{k-r}(x-z)  \cC^{k-r}(y-w) dxdydzdw \bigg)^{1/2} \\
&=  \frac{a_k^2}{\s_{t, N_t}^2} \sqrt{h_t(r, k-r)}
\quad\text{with $h_t$  as in \eqref{def_HT}.}
\end{aligned}
\end{align}

\noi
Now
we  apply the Malliavin-Stein bound \eqref{MSbdd3}
with $Y = F_{t_i}$ and $Y'  = F_{t_j}$, $i\in\{1,2\}$:

\noi
\begin{align}  \label{deda1}
\begin{aligned}  
    \sqrt{\Var\big(  \langle DF_{t_i},-DL^{-1}F_{t_j} \rangle_{\fH}   \big)}&\le \sqrt{N_{t_i}}\sum_{q=1}^{N_{t_i}} 
    3^{2q} q!\M_{q,t_i} 	 + N_{t_j}\sum_{p=1}^{N_{t_j}}  
    3^{2p} p! \M_{p,t_j}\,,
  \end{aligned}
  \end{align}
where 
\[
\M_{p, t} =  \frac{a_p^2}{\s_{t, N_t}^2} \sup_{1\leq r \leq p-1}\sqrt{h_t(r, p-r)}.
\]

Note that with $\s_{t, N_t}\asymp \s_t$ (see Lemma \ref{tech2})
and $\sum_{k\geq R} a_k^2 k! <\infty$,
we have

\noi
\begin{align}\label{RMB}
\begin{aligned}
 \sum_{p=1}^{N_{t}}  3^{2p} p!    \M_{p, t}
 &\leq  \sum_{p=1}^{N_{t}}  3^{2p} p! a_p^2    
  \frac{1}{\s_{t, N_t}^2} \sup_{1\leq r \leq p-1}\sqrt{h_t(r, p-r)} \\
  &\les 
   \sum_{p=1}^{m}  3^{2p} p! a_p^2    \xi_R(t)
   +  3^{2N_t}   \bigg( \sum_{p=m+1}^{N_t}  p! a_p^2  \bigg)    \xi_m(t)
 \end{aligned}
 \end{align}
 
 \noi
with $m > d/\dl$, where $\xi_m(t)$ is defined as in \eqref{XIM0}.
By Condition \ref{cond8}, $  \xi_R(t) \les \log^{-\theta_0}(t)$
and by Lemma \ref{tech4}, $\xi_m(t) \les t^{-a}$ for some $a > 0$.
Thus,  with $N_t \sim \log^\theta(t)$ and 
$\theta$ as in \eqref{th0} (i.e., $0 < \theta < \min\{1,\theta_0\}$)

\noi
\begin{align}\label{MPT_bdd}
 \sum_{p=1}^{N_{t}}  3^{2p} p!    \M_{p, t}
 \les   \log^{-\theta_0}(t) +  9^{\log^{\theta}(t)} t^{-a}
 \les    \log^{-\theta_0}(t). 
\end{align}
Therefore, the condition {\bf (b)} can be verified 
by invoking 
  \eqref{deda1} and \eqref{MPT_bdd} with 
\eqref{th0}:
\begin{align}
	\sqrt{\Var\big(  \langle DF_{t_i},-DL^{-1}F_{t_j} \rangle_{\fH}   \big)} \les \log^{\frac\theta 2-\theta_0}(t_1)+\log^{\theta-\theta_0}(t_2) ,
	 \label{th0b}
\end{align}

\noi
where, due to \eqref{th0}, $\frac\theta 2-\theta_0\leq  \theta - \theta_0 < 0$.

\medskip

Hence, the condition {\bf (b)} is verified and thus by Proposition
\ref{prop210}, $\{F_t = \frac{1}{\s_{t, N_t}} Y_{t, N_t}: t\geq 1\}$
satisfies the $\ASCLT$. 
\qedhere
  \end{proof}

\begin{proof}[Proof of Theorem \ref{thm:main}]
The proof  of part  {\bf (2) [$\ASCLT$]} can be easily done
by combining  Lemma \ref{Red_step} (with $N = N_t \sim \log^\theta(t)$),
 Proposition \ref{prop_FT}, and Lemma \ref{lem_red},
 where $\theta$ is given as in \eqref{th0}.
 
 Next, we prove part  {\bf (1) [\textsf{QCLT}]}. 
 Using the definition of Wasserstein distance, we have 
 \[
W_2\big( Y_t / \s_t, \NN(0,1)  \big)
 \leq W_2\big( Y_t/\s_t,  F_t \big)
 + W_2\big(  F_t,  \NN(0,1) \big).
 \]
 The second term can be bounded by using
 Proposition \ref{MS_bdd} and \eqref{th0b}:
 \[
W_2  \big(  F_t,  \NN(0,1) \big)
   \les \log^{\theta- \theta_0  }(t);
 \]
 the first term can be bounded by the $L^2(\O)$-norm of $F_t -  \tfrac{Y_t}{\s_t}$,
 so that it follows from Lemma \ref{Red_step} that 
 \begin{align}\label{SIMIa}
W_2\big( Y_t/\s_t,  F_t \big)\leq
  \sqrt{ \E\big[   | \tfrac{1}{\s_t} Y_t -  \tfrac{1}{\s_{t, N}}  Y_{t, N_t}  |^2 \big]  }
  \les \log^{-\frac{d\theta}{2\al}}(t).
 \end{align}
Thus,  the bound \eqref{QCLT_bdd1} holds true. 
Hence, the proof of Theorem \ref{thm:main} is completed.
 \qedhere
    
\end{proof}

\begin{remark} \rm
By Proposition \ref{MS_bdd}, we have 
\[
d_{\rm TV}(\tfrac{1}{\s_{t, N}}  Y_{t, N_t} , \NN(0,1)  )
\les \log^{\theta - \theta_0  }(t).
\]
However, a  bound similar to \eqref{SIMIa}
does not hold for the total variation distance.
This is why we describe our QCLT in Theorem \ref{thm:main}
by using the Wasserstein distance $W_2$.
We could not work with $W_p$ for $p>2$,
since under our general mild assumptions, 
we can only bound the second moment (not higher moments)
of the difference $\tfrac{1}{\s_t} Y_t -  \tfrac{1}{\s_{t, N}}  Y_{t, N_t}  $.

\end{remark}

\subsection{Proofs of Corollaries} \label{SEC3_2}

  Let us begin with the proof of the Malliavin differentiability of integrals functionals 
  of Berry's random wave model.

 \begin{proof}[Proof of Corollary \ref{corMalliavinBerry}]
The covariance function 
$\cC(x)=b_d(|x|)$ (as defined in \eqref{bd})  
of Berry's random wave model satisfies Condition \ref{cond5} 
with $\delta=(d-1)/2$, thanks to \eqref{Jdinf}.  
It is also known that

\noi
\begin{align}\label{Jd0}
     J_{\frac{d}{2}-1}(r) =   \frac{r^{\frac{d}{2}-1}}{2^{\frac{d}{2}-1}\Gamma\left (\frac{d}{2} \right )} - \frac{r^{\frac{d}{2}+1}}{2^{\frac{d}{2}+1}\Gamma\left (\frac{d}{2}+1 \right )} + O\left (r^{\frac{d}{2}+3}\right ),\qquad  r\to 0^+,
 \end{align}

 \noi
see, e.g., \cite[(1.71.1)]{szego39},  
which yields 
\[
b_d(|x|) = 1 - \frac{|x|^2}{2d} + O(|x|^4)
\quad\text{as $|x|\to 0^+.$}
\]
That is, Condition \ref{cond6} with $\alpha=2$
is satisfied by $\cC(x)=b_d(|x|)$. 
Moreover, the condition \eqref{minore} in Lemma \ref{tech1} also holds
true. Therefore, Lemma    \ref{tech1}, together with
Remark \ref{remModulo}, implies  \eqref{stima_momentiBessel}
and
\[
\int_{\R^d} | \cC^q(z)|dz \asymp q^{-d/2},
\]
which further implies  the desired
Malliavin differentiability in view of \eqref{stima_derivata}.
Hence, the proof is concluded. \qedhere

\end{proof}

 Next, we present  the proof of the $\ASCLT$ 
 in the Breuer-Major setting (i.e.,  Corollary \ref{cor:ASBM}). 
As we anticipated in Section \ref{S1_3},
Corollary \ref{cor:ASBM} can be proved as 
a direct application of Theorem \ref{thm:main}. 
Assuming Condition \ref{cond5} with $\delta>d/R$ 
and $\varphi-\E[\varphi(Z)]$ non-odd would be enough 
to prove a CLT, since one can apply 
Breuer-Major theorem (Theorem \ref{recap}-(i));
for the $\ASCLT$ to hold,   
we need to additionally assume Condition \ref{cond6}.

\begin{proof}[Proof of Corollary \ref{cor:ASBM}]
   First of all, we note that Condition \ref{cond5} (with $\delta > d/R$) and Condition \ref{cond6} hold by assumption.
  Secondly,     Condition \ref{cond7} is 
  also  verified, since $\cC\in L^R(\R^d)$ 
  and $\varphi-\E[\varphi(Z)]$, with $Z\sim\NN(0,1)$, is not odd.
     Moreover, Condition \ref{cond8} is satisfied as well.
     In fact,   Lemma \ref{tech4} ensures that for $m > d/\delta$
     \begin{equation}\label{lesss}
     \xi_m(t) \les t^{-a},
     \end{equation}
     for some $a>0$. Since $d/\delta < R$, we can choose $m=R$ in \eqref{lesss}, 
     thus \eqref{xiLeo} holds true.
     Hence, the $\ASCLT$ holds as a consequence of 
       Theorem \ref{thm:main}-{\bf(2)}.
\end{proof}

We conclude this section by proving Corollary \ref{cor:berry},
i.e., the $\ASCLT$ for integral functionals of Berry's random wave model.

\begin{proof}[Proof of Corollary \ref{cor:berry}]
In view of \eqref{lab1} and  Lemma \ref{lem_red},
we see that the first chaotic component is asymptotic negligible
and the $\ASCLT$ for $(Y_t - \E[Y_t])/\sqrt{\Var(Y_t)}$ is equivalent
to that for $(Y_t - \E[Y_t] - a_1 \int_{tD} B_xdx)/\sqrt{\Var(Y_t)}$.
From now on, we assume $R\geq 2$.

    Recall from \eqref{bd}, \eqref{Jdinf},  and \eqref{Jd0} that
       the covariance kernel of Berry's random wave model $\cC(x)= b_d(|x|)$ 
    satisfies Condition \ref{cond5} with $\delta=(d-1)/2$ and Condition \ref{cond6} 
    with $\alpha =2$; see also the proof of   Corollary \ref{corMalliavinBerry}.

Now let us verify the Condition \ref{cond7}.
Excluding the cases in \eqref{excases},
we {\it claim} that
the function  

\noi
\begin{align} \label{claim_wrm}
\begin{aligned} 
t\in\R_+\mapsto & \,\, w_{t, M}  
= \sum_{q= R}^M a_q^2 q! \int_{\{ |x| \leq t \}} \cC^q(x)dx  \\
&\text{is regularly varying at infinity
with $w_{\infty, M} > 0$}
\end{aligned}
\end{align}

\noi
for $M\geq \max\{5, R\}$.\footnote{The number $5$ is picked for the consideration 
of the fourth case in \eqref{varBerry}.}
Since   
\begin{align} \label{5M}
\text{$M\ge 5$ and  $\frac{d}{\dl} -1=\frac{2d}{d-1} -1 \leq 3$} 
\end{align}
for $d\ge2$ and  $D$ is assumed to be a centered ball,
the condition (c3) in Condition \ref{cond7} holds true.
For the sake of completeness, we sketch the proof of  the  claim \eqref{claim_wrm}
as follows.  The proof is almost identical to that in \cite[Lemma 4.4]{Mai23},
although only the case $d=2$ is treated therein.
\begin{itemize}

\item Note that for any $d\geq 2$ and $q\geq 3$ (excluding $(d, q) = (2, 4)$), 
\begin{align*}
 \int_{\{ |x| \leq t \}} \cC^q(x)dx =    c_d \int_0^t b^q_d(r) r^{d-1}dr 
\end{align*}

\noi
{\it has a positive and finite limit} as $t\to+\infty$, where the immaterial 
$c_d$ in this proof is a constant that may vary from 
line to line; see (1.2), (2.3), and Lemma 2.6 in \cite{GMT24}.
It remains to consider the case $q=2$ and the case $(d, q) = (2, 4)$.

\item 
 For $q=2$, we use    the bound \eqref{Jdinf} with \eqref{bd} and
 $2\cos^2(x) = 1 + \cos(2x)$
  to write 
  
  \noi
\begin{align} \label{V2T}
\begin{aligned}
 v_{2, t}:= \int_{\{ |x| \leq t \}} \cC^2(x)dx 
 &=   c_d \int_0^t b^2_d(r) r^{d-1}dr  = c_d \int_0^t J_{\frac{d}2-1}^2(r) r dr \\
 &=  c_d \int_0^t  \cos^2(r - \tfrac{d-1}{4}\pi) dr + O\bigg(1 +  \int_1^t  r^{-2}   dr  \bigg) \\ 
 &= c_d \, t + \int_0^t \cos(2r - \tfrac{d-1}{2}\pi)dr + O(1),
\end{aligned}
\end{align}

\noi
which implies $ v_{2, t}\sim c_d\, t$ for some  constant $c_d\in(0, \infty)$.

\item For $(d,q)=(2, 4)$, we write 

\noi
\begin{align} \label{V4T}
 v_{4, t}:= \int_{\{ |x| \leq t \}} \cC^4(x)dx 
 &=   c_d \int_0^t J^4_0(r) r dr  \sim c'_d \log t
\end{align}
for some constant $c'_d \in(0,\infty)$, which follows by the exactly same arguments
as in the proof of \cite[Lemma 4.4]{Mai23}.

\end{itemize}

\noi
Therefore, combining the above cases with $M\geq \max\{5, R\}$, 
we can establish the claim \eqref{claim_wrm}.

\medskip

It remains to verify  the Condition \ref{cond8}.  
Recall the definitions of $\xi_R(t)$, $h_t$ from 
\eqref{def_HT} and \eqref{XIM0}.
Letting $\s_t^2$ denote the variance of $Y_t$
as in \eqref{XIM0} and let 
$\s_t^2[q]$ denote the variance of the $q$-th chaotic
component of $Y_t$, we have

\noi
\begin{align}
\begin{aligned}
   \xi_R(t) &\leq \ind_{\{a_2\neq0\}}\,   \frac{\sqrt{h_t(1,1)}}{\s^2_t[2]}  \times \frac{\s^2_t[2]}{\s^2_t}
   +  \ind_{\{a_3\neq0\}}\,\frac{\sqrt{h_t(1,2)}}{\s^2_t[3]}\times  \frac{\s^2_t[3]}{\s^2_t}
   \\
   &\qquad + \ind_{\{a_4\neq0\}}\,\frac{\sqrt{h_t(1,3)}
   +\sqrt{h_t(2,2)}}{\s^2_t[4]}  \times \frac{\s^2_t[4]}{\s^2_t}
   + \xi_5(t) =: \sum_{k=2}^5 {\pmb{\g}_k}.
\end{aligned}\label{addends}
\end{align}

\noi
Since $5 > d/\dl =4$ for $d= 2$ by \eqref{5M}, Lemma \ref{tech4}
implies that ${\pmb{\g}_5}  =\xi_5(t) \les t^{-a}$ for some $a > 0$.
Moreover, in view of \eqref{lab1}, \eqref{varBerry},
and
\eqref{equ:important inequality}, we have the following observations for $d=2$:
\begin{itemize}

\item[(a)] 
when $a_2 \neq 0$, we have ${\pmb{\g}_3} +{\pmb{\g}_4}  \les \frac{\log t}{t}$;

\item[(b)]
when $a_2=0$, and $a_3 = a_4 =0$, we have $\xi_R(t) \leq \xi_5(t)\les t^{-a}$;

\item[(c)] when $a_2=0$,  the situation  $a_3 \neq 0 = a_4$ is excluded  as in \eqref{excases}; 

 \item[(d)] when $a_2=0$ and $a_4\neq 0$, we have ${\pmb{\g}_3}   \les \frac{1}{\log t}$.
\end{itemize} 
That is, when $d=2$, we only need to prove 

\noi
\begin{align}\label{case_d=2}
\frac{\sqrt{h_t(1,1)}}{\s^2_t[2]}  + \frac{\sqrt{h_t(1,3)} + \sqrt{h_t(2,2)} }{\s^2_t[4]}  \les \log^{-\theta}(t)
\end{align}

\noi
 for some $\theta > 0$.  When $d\geq 4$, we have $d/\dl = \frac{2d}{d-1} < 3$, then 
 Lemma \ref{tech4}
implies that $\xi_3(t) \les t^{-a}$ for some $a > 0$. That is, 
when $d\geq 4$, we only need to prove 

\noi
\begin{align}\label{case_d34}
\frac{\sqrt{h_t(1,1)}}{\s^2_t[2]}    \les \log^{-\theta}(t)
\end{align}

\noi
 for some $\theta > 0$. When $d=3$, we have  $\frac{d}{\dl} = \frac{2d}{d-1}=3$,
 then Lemma \ref{tech4}
implies that $\xi_4(t) \les t^{-a}$ for some $a > 0$.
Similarly, we have the following observations:

\noi
 \begin{itemize}

\item[(a')] 
when $a_2 \neq 0$, we have ${\pmb{\g}_3}   \les \frac{1}{t}$;

\item[(b')]
when $a_2=0$ and $a_3=0$, we have $\xi_R(t) \leq  \xi_4(t) \les t^{-a}$;

\item[(c')] when $a_2=0$, the situation $a_3\neq 0$ (i.e., $R=3$) is excluded  as in \eqref{excases}.
 
\end{itemize} 
That is, when $d=3$, we only need to prove \eqref{case_d34}. 
Therefore, from the above discussion, 
we only need to show

\noi
\begin{align}\label{Over1}
\begin{aligned}
\frac{\sqrt{h_t(1,1)}}{\s^2_t[2]}    &\les \log^{-\theta}(t) \quad\text{for $d\geq 2$}\\
\frac{  \sqrt{   h_t(1,3)} +  \sqrt{   h_t(2,2)} }{\s^2_t[4]}    &\les \log^{-\theta}(t) \quad\text{for $d = 2$.}\\
 \end{aligned}
  \end{align}

\noi
To prove the first bound in   \eqref{Over1},
we  can reason as in {\it Step 3} in the proof of \cite[Proposition 4.1]{MN23} to show

\noi
   \begin{align}
       \frac{h_t(1,1)}{\s^4_t[2]}
       \les t^{-1/2}\,.
       \label{boundcontractionberry}
   \end{align}
Indeed, 
 {\it Step 3} in the proof of \cite[Proposition 4.1]{MN23} indicates
 that  $h_t(1,1)/\s^4_t[2] \les (v_{2,t})^{-1/2}$ with $v_{2, t}$ as in \eqref{V2T},
 and the bound \eqref{boundcontractionberry} follows from $v_{2,t}\asymp t$.\footnote{Note that 
 $v_{2,t}$ coincides with the notation $w_{R,t}$ in \cite[(46)]{MN23}
 with $R=2$ therein.}

Next, we prove the second bound in  \eqref{Over1}.
For this purpose, we can apply \cite[Proposition 3.3]{MN23}
with $(d,R) =(2,4)$ to get
\begin{align}\label{Over2}
\pmb{\gamma}_4^2 \les \bigg(  \int_{\{ |x| \leq t\}}  \cC^4(x)dx  \bigg)^{-1} \les \frac{1}{\log t},
\end{align}
where the last bound is due to \eqref{V4T}.\footnote{Indeed, 
the normalized  contractions $h_t(1,3)/\s^4_t[4]$ and $h_t(2,2)/\s^4_t[4]$
can be bound by expressions in \cite[(35)]{MN23} with $q=4$ therein, 
and the condition (36) in  \cite[Proposition 3.3]{MN23} can be easily verified 
for $R=4, d=2$ and $C(x) = J_0(|x|)$ by using the bound \eqref{Jdinf}.
Then, the above bound \eqref{Over2}, in view of the last two displays in the proof of  \cite[Proposition 3.3]{MN23},
follows immediately. }
Thus,  we just verified the Condition \ref{cond8}.

\medskip

Hence, the proof  is completed by invoking Theorem \ref{thm:main}.
\qedhere

\end{proof}

\section{Technical results}
\label{S5}
 
 Let us first prove Lemma \ref{tech1}, a key result for our analysis.

\begin{proof}[Proof of Lemma \ref{tech1}]
To estimate the integral $\int_{\R^d}|\cC(z)|^N dz$, 
we will first break it into two parts. 
The first part concerns the integration over $\{ |x| > K\}$ using
the first bound in \eqref{cond56}.
To bound the integral over $\{ |x| \leq K\}$, we will relate it 
with the integral over $\{ |x| \leq \eps\}$ for small enough $\eps$
so that we can use the second bound in \eqref{cond56}.
The comparison between the integral over   these two balls 
is made possible via a doubling inequality for nonnegative nonnegative-definite
functions from  \cite{GT19}:
  for any  convex compact subsets 
 $U, V\subset \R^d$ that are symmetric about zero,  
 there is some  constant 
 $C = C(d, U, V) \in(0,\infty)$
 such that 
 \[
 \int_U \gamma(x)dx \leq  C(d, U, V)  \int_V \gamma(x)dx
 \]
 for any   nonnegative and nonnegative-definite function $\gamma: \R^d\to\R_+$.

Taking $\eps < \min\{ \eps_0 ,  C_2^{-\frac1\al} , 1\}$,
we can have     
\begin{align}    \label{pos1}
0< \cC(y) \leq 1- C_2  |y|^\al
\end{align}
for $|y| \leq \eps$. 
Now we can deduce from \eqref{pos1}
that 

\noi
 \begin{align} 
  \int_{\{|x|\leq \eps\}}\cC^N(x)\,dx 
&\leq \int_{\{|x|\leq \eps\}}   \big( 1- C_2  |x|^\al \big)^N   dx   \label{56a2}\\
&= \frac{2\pi^{\frac{d}{2}}}{\Gamma(d/2 )} \int_{0}^\eps   r^{d-1}  \big( 1- C_2  r^\al \big)^N   dr \notag \\
&=  \frac{2\pi^{\frac{d}{2}}}{\Gamma(d/2 )} \frac{1}{\alpha\,C_2^{d/\al}}\int_{0}^{C_2 \eps^\alpha}  y^{\frac{d}\al -1} ( 1- y )^N   dy \label{56a3} \\
&\le \frac{2\pi^{\frac{d}{2}}}{\Gamma(d/2 )}\frac{1}{\alpha\,C_2^{d/\al}}  \int_{0}^1  y^{\frac{d}\al -1} ( 1- y )^N   d y 
\les (N+1)^{-\frac{d}{\al}}, \label{56a4}
\end{align}
 
 \noi
where we made the change of variable $y = C_2  r^\al \in(0,1)$
in \eqref{56a3}
and used the Stirling's approximation for the integral in \eqref{56a4}
(note that this integral gives us $\textup{Beta}(\tfrac{d}{\al}, N+1)$).
 
Now we consider the integration over $\{ |x| > K \}$
with  $K > \max\{1,  (2C_1)^{1/\dl} \}$.
 We can get with $N \geq M := \lfloor \frac{d}{\dl}\rfloor +1$ that

\noi
\begin{align}\label{56a5}
 \int_{\{|x| > K\}}|\cC^N(x)| dx
& \leq    \int_{\{|x| > K\}} |\cC^M(x)| \big( C_1 |x|^{-\dl} \big)^{N-M} dx  \notag \\
& \leq \|  \cC \|^M_{L^M(\R^d)}  (C_1 K^{-\dl} )^{N-M} \notag \\
 &\leq \|  \cC \|^M_{L^M(\R^d)} 2^{M-N} \les 2^{-N}.
\end{align}

Finally,  letting
 \[
 N' = N\ind_{\{\frac{N}{2}\in\N\}} + (N-1)\ind_{\{\frac{N}{2}\notin\N\}}
 \]
 so that $\cC^{N'}$ is a nonnegative  nonnegative-definite function,\footnote{Let $\mu$ be the associated symmetric spectral (probability) measure, then for any even integer $m\geq 2$, 
 $\cC^{m}(x)  
 = \int_{\R^{dm}} e^{-i x (\xi_1+ ... + \xi_m)} 
 \mu(d\xi_1) ... \mu(d\xi_m )$,
 from which we deduce that 
 $\sum_{i=1}^n \cC^m(x_j)\cC^m(x_i) \lambda_j  \overline{\lambda_i} \geq 0$
 for any finite number of complex numbers $\lambda_1, ... , \lambda_n$
 and for any $x_1, ... , x_n\in\R^d$, $n\in\N$.
 That is, $\cC^m$ is positive-definite for any even $m\geq 2$.  } 
 we apply the said doubling inequality for $V = \{ | x| \leq \eps \}$
 and $U  = \{ | x| \leq K \}$ with $K$ as in \eqref{56a5}
 and $\eps$ as in \eqref{56a2}
 to get 
 \begin{align}\label{add1}
 \int_{|x| \leq K} \cC^{N'}(x) dx 
 \leq C_0
  \int_{|x| \leq \eps} \cC^{N'}(x) dx  ,
 \end{align}
 where the constant  $C_0 > 1$ does not depend on $N$. 
 It follows from \eqref{add1} and \eqref{56a4} that 
 
 \noi
 \begin{align} \notag 
 \begin{aligned}
  \int_{\R^d}   | \cC^{N}(x) | dx 
 &=   \int_{\{ | x| > K \}}   | \cC^{N}(x) | dx  + \int_{\{ | x| \leq K \}}   | \cC^{N}(x) | dx\\
 &\les 2^{-N} +     \int_{\{ | x| \leq K \}}   \cC^{N'}(x)  dx 
 \leq 2^{-N} +   C_0   \int_{\{ | x| \leq \eps \}}   \cC^{N'}(x)  dx \\
 &\les 2^{-N} + (N'+1)^{-d/\al} \les N^{-d/\al}.
\end{aligned}
\end{align}
 
 \noi 
 In order to conclude the proof of Lemma \ref{tech1}, 
 bearing in mind \eqref{minore}, it suffices to split 
 $\int_{\R^d} |C^N(x)|dx$ as above, define 
 \begin{center} 
 $N''=N$ if $N$ is even
 and 
  $N''= N+1$ if $N$ is odd,
 \end{center}
 and note that the contribution of the integral over 
 $\{ |x| >K\}$ is negligible 
 (indeed, it can be bounded as in \eqref{56a5}), 
 while the contribution of the integral over $\{ |x| \le K\}$ 
 is $\ges N^{-\frac{d}{\alpha}}$, as follows. 
 Let $\eps < \min\{ \eps_0 ,  C_3^{-\frac1\al} , 1\}$ with $C_3$ as in \eqref{minore}.
Note first that 
 $$
 \int_{\{ |x| \le K\} } |C^N(x)| dx \ge \int_{\lbrace |x| \le K\rbrace} C^{N''}(x) dx 
 \ge \int_{\lbrace |x| \le \eps \rbrace} C^{N''}(x) dx ,
 $$

\noi
where the last inequality is due to $K \geq \eps$.
  Then we can write for $N > 1/(C_3 \eps^\alpha)$:

 \noi
 \begin{align*}
 \int_{\lbrace |x| \le \eps \rbrace} C^{N''}(x) dx 
 &\ges
  \int_{\{ |x|\leq \eps\}}   \big( 1- C_3  |x|^\al \big)^N   dx \\
  &=  \frac{2\pi^{\frac{d}{2}}}{\Gamma (d/2 )} \frac{1}{\alpha\,C_3^{d/\al}}\int_{0}^{C_3 \eps^\al}  
 y^{\frac{d}\al -1} ( 1- y )^N   dy  \\
&\ges  \int_{0}^{\frac{1}{N}}  y^{\frac{d}\al -1} ( 1- y )^N   dy  \\
&\geq  \Big ( 1- \frac{1}{N} \Big)^N\int_{0}^{\frac{1}{N}}  y^{\frac{d}\al -1}    dy    \\
&=   ( 1-  N^{-1})^N  \frac{\al}{d} N^{-\frac{d}\al}  \asymp N^{-\frac{d}{\al}},
 \end{align*}

\noi
using $\lim_{N\to+\infty} ( 1-  N^{-1})^N=e^{-1}$ for the last step.
Hence the proof is concluded.
 \qedhere
 
\end{proof}

 \begin{remark}\rm \label{remModulo}

A careful investigation of the proof of Lemma \ref{tech1} reveals that 
\eqref{stima_momenti_gen} holds also without taking the absolute value of $\cC^N$, that is, 
\begin{equation} \label{equsenzava}
\int_{\R^d} {\cC}^N(z)dz \asymp N^{-d/\al}
\end{equation}

\noi
under the assumptions \eqref{cond56}-\eqref{minore} of Lemma \ref{tech1}. 
The upper bound in \eqref{equsenzava} is an obvious consequence of Lemma \ref{tech1}.
Now  let us  prove
the lower bound in \eqref{equsenzava}.
First of all, 
note that condition \eqref{cond56} implies\footnote{Indeed, 
if $\cC(x)= \cC(0)=1$ for some $x\neq0$, then  we have 
$\E[ (B_x - B_0)^2] = 0$ implying that $B_x = B_0$ almost surely.
By stationarity of $\bfB$, we have $B_{nx}= B_{(n-1)x} = ... =B_0$ almost surely
thus $\cC(nx) =1$ 
for any   integer $n\geq 1$, which     contradicts \eqref{cond56}.}
\begin{equation}
\label{additionalcondition}
\cC(x)=\cC(0)=1\,\iff\,\,x=0.
\end{equation} 
Moreover, if \eqref{minore} holds, then $\cC$ is positive on $\{ |x|\le \eps\}$ 
 for $\eps > 0$ small enough,
and thus

\noi
\begin{align} \label{3terms}
    \int_{\R^d} \cC^N(z)dz&=\int_{\{ |x|\le \eps\}}|\cC^N(x)|\,dx
    +\int_{\{ \eps < |x| \le K\}}\cC^N(x)\,dx
    +\int_{\{|x|>K\}}\cC^N(x)\,dx.
\end{align}

\noi
By conditions \eqref{cond56}-\eqref{minore}, reasoning as in the proof of Lemma \ref{tech1}, 
we have that the first summand
in \eqref{3terms} is of order   $N^{-d/\alpha}$ and dominates the third summand, i.e., 

\noi
\begin{align*}
    \Big|   \int_{\{ |x|>K\}}\cC^N(x)\,dx\Big| 
    \les  2^{-N} = o\big(N^{-d/\al}\big)   \,\quad \text{as }N\to\infty.
\end{align*}
To conclude, we only need to show that  
 the second summand in \eqref{equsenzava} is also dominated by the first one.
By the continuity assumption in \eqref{C00}, 
the covariance function
\[
\cC(x) = \E[B_x B_0]
\]
is a continuous function on $\R^d$. Due to the equivalence 
\eqref{additionalcondition}, one has
$
\max_{\substack{\e\le\|x\|\le K}}|\cC(x)|<1
$
for any $0< \eps < K <\infty$.
As a consequence, 
\[
\Big|\int_{\e< |x|\le K}\cC^N(x)\,dx\Big| 
\les  \left(\max_{\substack{\e\le |x|\le K}}|\cC(x)|\right)^N
=  o\big(N^{-d/\al}\big) .
\]

\noi
Hence, \eqref{equsenzava} is proved.  
\hfill $\square$
\end{remark}

Now we give the proof of Proposition \ref{MS_bdd}.  

    \begin{proof}[Proof of Proposition \ref{MS_bdd}]
 The bound \eqref{MSbdd1}  (in $d_{\rm TV}$ and $W_1$)   in (i)  can be found in, e.g., 
 \cite[Theorem 5.1.3]{NP12}.
  It is known in the 
 Malliavin-Stein community that
 the density assumption therein is not really needed.
 For example, 
 one can apply   \cite[Proposition 2.1.1]{GZthesis}  and
 \cite[Proposition 5.1.1]{NP12} to derive  \cite[(5.1.4)]{NP12},
 which is exactly our \eqref{MSbdd1} with $\textup{dist} = d_{\rm TV}$.
 For the bound  \eqref{MSbdd1}  in $W_2$ distance,
 one can start with the inequality $W_2(Y, Z) \leq S(Y)$,
 with $S(Y)$ denoting the Stein's discrepancy; see \cite[Proposition 3.1]{LNP15}. 
 In our setting, $S(Y) =\|  \E[ \langle DY, -DL^{-1}Y \rangle_\fH | Y ] - 1 \|_{L^2(\Omega)}$
 and the expression 
 $\E[ \langle DY, -DL^{-1}Y \rangle_\fH | Y ]$ is known as the 
 Stein's kernel for the law of $Y$;
 see also the discussions in \cite[pages 257-259]{LNP15}.

\medskip

Now let us prove the bound \eqref{MSbdd3}, while  
\eqref{MSbdd2} is a particular case with $Y = Y'$.
Since
\[
Y = \sum_{p=1}^N I_p(f_p) 
\quad
{\rm and}
\quad
Y' = \sum_{q=1}^{N'} I_q(g_q)
\quad\text{with $f_p\in\fH^{\otimes p}_{\rm sym}$
and $g_q\in\fH^{\otimes q}_{\rm sym}$, }
\]
  and $Y, Y'$ are assumed to have variance one, 
we get 
$
\| f_p \|^2_{\fH^{\otimes p}} \leq 1/p!
$
and
$
\| g_q \|^2_{\fH^{\otimes q}} \leq 1/q!.
$

\medskip

Let us first 
express the inner product 
$\langle DY,-DL^{-1}Y' \rangle_{\fH}$ using the product formula (see Lemma \ref{product_f}) and \eqref{MD1b}:

\noi
\begin{align}\label{TPQ}
\begin{aligned}
&\langle DY,-DL^{-1}Y' \rangle_{\fH}
 = \sum_{p=1}^N  \sum_{q=1}^{N'} p   
\langle I_{p-1}(f_p), I_{q-1}(g_q)\rangle_{\fH} \\
&= \sum_{p=1}^N  \sum_{q=1}^{N'} p    \sum_{i\geq 1}
 I_{p-1}(f_p \otimes_1 h_i) I_{q-1}(g_q \otimes_1 h_i)  \\
& = \sum_{p=1}^N  \sum_{q=1}^{N'} p  
\sum_{r=0}^{(p\wedge q)-1} r! \binom{p-1}{r} \binom{q-1}{r} \\
 &\qquad\qquad \times  \sum_{i\geq 1}
  I_{p+q-2r-2}\Big(   
  {\rm sym}\big( [ f_p \otimes_1 h_i] \otimes_r  [ g_q \otimes_1 h_i ]  \big)
  \Big) \\
  &= \sum_{p=1}^N  \sum_{q=1}^{N'} p  
\sum_{r=0}^{(p\wedge q)-1} r! \binom{p-1}{r} \binom{q-1}{r}
  I_{p+q-2r-2}\big(     {\rm sym} ( f_p \otimes_{r+1}  g_q    ) \big)\\
&  =:  \sum_{p=1}^N  \sum_{q=1}^{N'} p  \mathbf{T}_{p,q} ,
  \end{aligned}
\end{align}

\noi
where $\{ h_i:i \geq1\}$ is the orthonormal basis of $\fH$,
and the last line follows simply from the definition of contractions;
see also equation (6.3.2) in \cite[Chapter 6]{NP12}.
Note that for the $\mathbf{T}_{p,q}$, defined in \eqref{TPQ},
its expectation is nonzero only if $p+q-2r-2 =0$,
or equivalently only if $p=q=r+1$
(this nonzero expectation does not play a role
when we take variance in \eqref{MPQ} below).
And when   $p=q=r+1$ does not hold, 
we can deduce from 
\eqref{sym_bdd}
and  \eqref{C2b} that

\noi
\begin{align}\label{TPQ2}
\begin{aligned}
 & \|   {\rm sym} ( f_p \otimes_{r+1}  g_q)   
               \|_{\fH^{\otimes (p+q-2r-2)}}^2\\
&  \leq \langle f_p\otimes_{p-r-1} f_p, g_q\otimes_{q-r-1} g_q \rangle_{\fH^{\otimes 2r+2}} \\
&\leq  \| f_p\otimes_{p-r-1} f_p\|_{\fH^{\otimes 2r+2}}
 \| g_q\otimes_{q-r-1} g_q\|_{\fH^{\otimes 2r+2}}\\
&\leq  \tfrac{1}{p!}\M'_{q} \ind_{\{ p\le q \}}  
+ \,\tfrac{1}{q!} \M_{p}\ind_{\{ p>q \}}
\end{aligned}
\end{align}

\noi
with $\M_p, M'_q$ as in \eqref{def_MP} and \eqref{def_MQ},
where
the  first term comes from the case 
$q \geq  p$
and the second term comes from  the case $p> q$, combined   with the fact that the squared norm of $f$ (resp. of $g$) is less than $1/p!$ (resp. of $1/q!$),  since we are assuming unit variances.
Therefore, it follows from 
Minkowski inequality, orthogonality relation \eqref{chaos_iso},
 \eqref{TPQ}, \eqref{TPQ2}, and \eqref{SeH}
that 
 \begin{align} \label{MPQ}
	\begin{aligned}
		&\sqrt{\Var\big( \langle DY,-DL^{-1}Y' \rangle_{\fH} \big)}
		\leq  \sum_{p=1}^N  \sum_{q=1}^{N'} p  \sqrt{\Var(  \mathbf{T}_{p,q}  ) } \\
		&=\sum_{p=1}^N  \sum_{q=1}^{N'} p
		\sqrt{\sum_{r=0}^{(p\wedge q)-1} \Big[ r! \binom{p-1}{r} \binom{q-1}{r}\Big]^2
		(p+q-2r-2)! }  \\
		&\qquad\qquad \times   
		\sqrt{ \tfrac{1}{p!}\M'_{q} \ind_{\{ p\le q \}}  + \,\tfrac{1}{q!} \M_{p}\ind_{\{ p>q \}} }
		\\
		&\leq  \sum_{p=1}^N  \sum_{q=1}^{N'} p
		\sqrt{3^{p+q-2} (p-1)! (q-1)!} 
		\sqrt{ \tfrac{1}{p!}\M'_{q} \ind_{\{ p\le q \}}  + \,\tfrac{1}{q!} \M_{p}\ind_{\{ p>q \}} }.
	\end{aligned}
\end{align}

\noi
Then, using  $\sum_{1 \leq p \le q} \sqrt{3}^{p-1} =\tfrac{\sqrt{3}^q-1}{\sqrt{3}-1}\leq \sqrt{3}^{q+1}$ and Jensen's inequality,
we can further   get 

\noindent
 \begin{align*}  
	\begin{aligned}
 \sqrt{\Var\big( \langle DY,-DL^{-1}Y' \rangle_{\fH} \big)} 
		&\leq  \sum_{p=1}^N  \sum_{q=1}^{N'} 
		\sqrt{ 3^{q-1} q!\M'_{q} } \sqrt{\,3^{p-1}  \ind_{\{ p\le q \}}}	\\
		&\qquad + \sqrt{N} \sum_{p=1}^N  \sum_{q=1}^{N'} 
		\sqrt{ 3^{p-1} p! \M_{p}}\sqrt{{3^{q-1}} \ind_{\{ q<p \}}} \\
		&\leq  \sum_{q=1}^{N'} 
		3^q\sqrt{ q!\M'_{q} }	 + \sqrt{N}\sum_{p=1}^N  
		3^p \sqrt{p! \M_{p}}\\
        &\le \sqrt{N'}\sum_{q=1}^{N'} 
		3^{2q} q!\M'_{q} 	 + N\sum_{p=1}^N  
		3^{2p} p! \M_{p}.
	\end{aligned}
\end{align*}

\noi
Therefore, the proof is completed. 
\qedhere

\end{proof}

\begin{lemma}
    \label{tech2} 
 Recall from \eqref{wrM} the definition of $w_{r, M}$ and put
 \[
   w_{\infty,M} :=  \sum_{k=R}^{M} k!a_k^2   \int_{\R^d} \cC^k(x)\,dx.
 \]   
 Let $\s_t^2 = \Var(Y_t)$ with $Y_t$ as in \eqref{def_Y} and \eqref{CP0},
 $\s_{t, N}^2 = \Var(Y_{t, N})$ with $Y_{t, N}$ as in \eqref{equ:YtN}. 
 Recall also the three cases {\rm (c1)-(c3)} in Condition \ref{cond7}.
 Then,  the following statements hold:
 \begin{itemize}
 \item[\bf (i)] In case {\rm (c1)}, 
 we can find some sufficiently large $t_0, M > 0$
 and two finite constants $c, c' >0$ such that 
$c t^d \leq \s^2_{t, N} \leq \s_t^2 \leq c' t^d$
for any $t\geq t_0$ and $N\geq M$. In particular, 
$w_{\infty, M}\in(0, \infty).$

 \item[\bf (ii)]  In case {\rm (c2)}, we can 
  find some sufficiently large $t_0> 0$
 and two finite constants $c, c' >0$ such that 
$c t^d w_{t, R} \leq \s^2_{t, N} \leq \s_t^2 \leq c' t^d w_{t, R}$
for any $t\geq t_0$ and $N\geq R$.
The sub-case where $w_{\infty, R} <\infty$ is also covered by the case 
 {\rm (c1)}.

  \item[\bf (iii)]  In case {\rm (c3)} and under the condition \eqref{cond5}
  \textup{(}i.e., $|\cC(x)| \les |x|^{-\dl}$
  for some $\dl > 0$\textup{)}, 
 we can 
  find some sufficiently large $t_0> 0$
 and two finite constants $c, c' >0$ such that 
$c t^d w_{t, M} \leq \s^2_{t, N} \leq \s_t^2 \leq c' t^d w_{t, M}$
for any $t\geq t_0$ and $N\geq M    > \frac{d}{\dl} -1$.
 
 \end{itemize}
%
%
  
\end{lemma}

\begin{proof} 

Let us first prove  {\bf (i)}.
It is routine to proceed as follows:

\noi
\begin{align}
\label{routine1}
\begin{aligned}
\frac{\s_t^2}{\Vol(tD)} =&  
\frac{1}{\Vol(tD)}  \sum_{k\geq R} a_k^2 k! \int_{(tD)^2} \cC^k(x-y) dxdy \\
= &  \sum_{k\geq R} a_k^2 k! 
\int_{tD - tD} \cC^k(z)  \frac{ \Vol( tD \cap (tD -z)   ) }{\Vol(tD)}   dz \\
= & \sum_{k\geq R} a_k^2 k! 
\int_{tD - tD} \cC^k(z)  \frac{ \Vol( D \cap (D -\tfrac{z}{t})   ) }{\Vol(D)}   dz \\
\xrightarrow{t\to+\infty}
& \sum_{k\geq R} a_k^2 k! 
\int_{\R^d} \cC^k(z)    dz \in[0,\infty),
\end{aligned}
\end{align}

\noi
 which follows from the dominated convergence theorem 
 with $|\cC(z)| \leq 1$, $\cC\in L^R(\R^d)$ 
 ($R$ being the Hermite rank of $\varphi$), and the fact that 
 $  \Vol( D \cap (D -\tfrac{z}{t})   ) /  \Vol(D)  \to 1$ as $t\to+\infty$.
 In the same way, we obtain for $N\geq M$ that 
 
 \noi
 \begin{align}  \label{route1a}
 \frac{\s_t^2}{\Vol(tD)} \geq \frac{\s_{t, N}^2}{\Vol(tD)} \geq 
  \frac{\s_{t, M}^2}{\Vol(tD)} \xrightarrow{t\to+\infty} 
  \sum_{k= R}^M a_k^2 k! 
\int_{\R^d} \cC^k(z)    dz.
 \end{align}
 
 \noi
Note that 
\begin{align}\label{neg1}
\int_{tD-tD} \cC^k(z)  \frac{ \Vol( tD \cap (tD -z)   ) }{\Vol(tD)}   dz \geq 0
\quad
{\rm and}
\quad
\int_{\R^d} \cC^k(z)    dz \geq 0
\end{align} 
  for any $k\geq R$, since they are  the  variance and the limiting variance
  of $\frac{1}{\sqrt{\Vol(tD)}} \int_{tD} H_k( B_x) dx$.
  Meanwhile,  
 the assumption that $\varphi-\E[\varphi(Z)]$ is not odd implies
that there is some {\it even} integer $q\geq R$ 
such that $a_q\neq 0$, and thus for any $M\geq q$,
we deduce from \eqref{neg1} and \eqref{routine1}
 that 

\noi
\begin{align}
+\infty >  w_{\infty, M} 
=
&\sum_{k=R}^{M}  a_k^2 k! \int_{\R^d} \cC^k(z)dz  
\geq   a_q^2 q!  \int_{\R^d} \cC^q(z)dz > 0.
 \label{routine2}
\end{align}

\noi
In particular, one can find sufficiently large 
$t_0 > 0$ such that $ w_{t, M} \geq  \frac{1}{2} w_{\infty, M}  > 0 $
for $t \geq t_0$.
Therefore, the statement {\bf (i)} follows from 
\eqref{route1a} and \eqref{routine2}.

\medskip

Next, we prove the statement {\bf (ii)}. 
Using \eqref{neg1} and $\cC^R\geq 0$, we can easily 
get 

\noi
\begin{align}\label{routine2b}
w_{\infty, M} 
&\geq  w_{\infty, R} := a_R^2 R! \int_{\R^d} \cC^R(x)dx \in( 0, \infty].
\end{align}

\noi
If $w_{\infty, R}  < \infty$, we are in the case (c1)
with \eqref{routine2} replaced by \eqref{routine2b}
 and the proof is done. 
Then, let us assume $w_{\infty, R}  = \infty$, i.e., $\cC\notin L^R(\R^d)$.
 
  Putting $r_{t} = t \max\{ | x |: x\in D\}$, we obtain
  as in \eqref{routine1} that 
  
  \noi 
  \begin{align}
\tfrac{\s^2_{t, R} }{ \Vol(tD) } 
&  = a^2_R R!  \int_{tD- tD} \cC^R(z) 
   \frac{ \Vol( D \cap (D -\tfrac{z}{t})   ) }{\Vol(D)}  dz \label{expand1} \\
&  \leq  a^2_R R!  \int_{\{|z|\le 2\,r_t\}} \cC^R(z) 
   \frac{ \Vol( D \cap (D -\tfrac{z}{t})   ) }{\Vol(D)}  dz 
   \leq   w_{2r_t, R} \les   w_{t, R},  \label{expand1b} 
  \end{align}
 
 \noi
 where, in the last line,  we used $\cC^R\geq 0$ and 
 the fact that $w_{r, R}$ is regularly varying
 in $r$ (see Lemma \ref{lemP}-(iv)). This gives us 
 $\sigma_{t,R}^2\les   t^d w_{t,R}$.
 Note that the set $D - D$ contains an open ball $\{|z|\le b\}$ and 
 one can find  $t_0 > 0$ sufficiently large such that 
 for any $z\in\{|z|\le b\}$,
 \[
   \frac{ \Vol( D \cap (D -\tfrac{z}{t})   ) }{\Vol(D)} \geq 1/2,
    \,\, \forall t \geq t_0.
 \]

\noi
 Then, it follows from \eqref{expand1}
 that for $t\geq t_0,$

\noi
\begin{align*}
\frac{\s^2_{t, R} }{ \Vol(tD) } 
&\geq \tfrac{1}{2} a^2_R R!  \int_{\{|z|\le b t\}}  \cC^R(z) 
  dz =  \tfrac{1}{2}   w_{bt, R} \ges   w_{t, R},
   \end{align*} 
where, in the last inequality,  we used again   
the fact that $w_{r, R}$ is regularly  varying.
 Thus, we just proved 
 the existence of large $t_0 > 0$ and finite constants $c, c_0 > 0$
to ensure that 
 
 \noi
 \begin{align}   \notag 
      c\,t^d\,w_{t,R}\le  \sigma_{t,R}^2\le c_0\, t^d w_{t,R}
 \end{align}
 for any $t \geq t_0$. 
 Now 
 we write for any $N\geq R$ and for $t \geq t_0$:

 \noi
 \begin{align}
 c\,t^d\,w_{t,R}\le  \sigma_{t,R}^2
 \leq \s^2_{t, N}  
\leq \s^2_{t}  
& = t^d \Vol(D) \sum_{k=R}^\infty a_k^2 k! 
\int_{tD-tD} \cC^k(z)  \frac{ \Vol( tD \cap (tD -z)   ) }{\Vol(tD)}   dz \notag  \\
&\leq t^d \Vol(D) \sum_{k=R}^\infty a_k^2 k!  
 \int_{tD-tD} \cC^R(z) \frac{ \Vol( tD \cap (tD -z)   ) }{\Vol(tD)}  dz  
 \label{expand2}  \\
   &\les t^d    w_{t, R}, 
   \notag
 \end{align}
 where we used $|\cC|\leq 1$ and $\cC^R\geq 0$ in \eqref{expand2},
 and we applied \eqref{expand1b}
 and the fact that $\sum_{k=R}^\infty a_k^2 k! <\infty$
  in the last inequality. 
 Hence,  the result {\bf (ii)} is proved.

 \medskip

Finally, we prove the result {\bf (iii)}.
Since $|\cC(x)| \les |x|^{-\dl}$ for some $\dl > 0$
and  $M +1 > \tfrac{d}{\dl}$, 
we have $\cC\in L^{M+1}(\R^d)$.
As a consequence, 

\noi
\begin{align*} 
\Var\bigg( \sum_{k=M+1}^{\infty}a_q \int_{tD} H_q(B_x)dx \bigg) \les t^d,
\end{align*}

\noi
which can be proved as in \eqref{routine1}.
It remains to show 
\begin{align}\label{abovebdd}
 t^d w_{t, M} \les  \s^2_{t, M} \les  t^d w_{t, M}.
\end{align}
Indeed, applying \cite[Theorem 1.1]{Mai23}
with 
\begin{center}
$A_x = \sum_{k=R}^M a_k H_k(B_x)$
and $w_t=w_{t, M} = \int_{|x| < t} \E[ A_x A_0]dx $,
\end{center}
 yields the above bound \eqref{abovebdd}. 
Hence, the proof of Lemma \ref{tech2} is completed. \qedhere

\end{proof}

\begin{lemma}\label{tech3} 
    
Suppose $K:\R^d\to\R$ is locally integrable and define 
\[
 w_t:=\int_{\{|z|\le t\}}\,K(z)\,dz.
 \]
 Assume that  $w_t: (0,\infty)\to \R$ is  regularly varying
with index $\rho$ at $+\infty$ 
and 
 \begin{align} \label{tech3aa}
\text{the limit  $w_{\infty}:=\lim_{t\rightarrow\infty}w_t$ 
 exists in $(0,\infty].$}
 \end{align}

   \noi
Then,  for any arbitrarily small $\dl > 0$,
 there is a constant $X = X_{K, \dl}$
such that   
  for any $t_2\ge t_1\ge X$,

\noi
\begin{align} \label{tech3a}
\int_{\{ |x| \leq a t_1\} }\int_{\{ | y| \leq a t_2\}} 
K(x-y)\,dx\,dy
\leq  C_{\dl, K} \, t_1^d\,\,w_{t_2} ,
\end{align}

\noi
where the constant $C_{\dl, K, a}$  
depends only on $\delta, K$,  and $a$.

\end{lemma}

\begin{proof} 

Let us first make a few preparation for the proof. 

\medskip

\noi $\bul$ {\bf Preparation.}
Put $D_r = \{|z|\le r\}$ and 
define  
\[
g_{a,  b}(z):={\rm Vol}(D_a\cap(D_b+z)) = \int_{D_a} \ind_{D_b}(x-z) dx.
\]

\noi
It is not difficult to see that $g_{a, b} = g_{b, a}$ is radial, that is, 
$g_{a,  b}(r\theta) = g_{b, a}(r\theta)$ is constant in 
$\theta\in\mathbb{S}^{d-1}$.
To abuse the notation in this proof, we will just write 
\[
g_{a,  b}(r) = g_{a,  b}(z)\,\,\, \text{with $r =|z|$}.
\]

\noi
It is easy to show that  $g_{a, b}$ is Lipschitz continuous.
Suppose $x = (x_1, 0, ..., 0),  y = (y_1, 0, ..., 0), $ 
with $|x|  = x_1$
and $|y| =y_1$,
then, 
\begin{align*}
g_{a, a}(x) - g_{a, a}(y)
&= \int_{D_a}  \big[  \ind_{D_a}( z-x) - \ind_{D_a}( z-y) \big]  dz \\
&= \int_{|\eta_2| \leq a} dz_2...dz_d \int_{|z_1|\leq r}
\Big[ \ind_{\{ |z_1-x_1| \leq  r\} } 
- \ind_{ \{  |z_1-y_1| \leq r \} } \Big] dz_1
\end{align*}

\noi
with $\eta^2_2 = z_2^2 + ... + z_d^2$ 
and $r = \sqrt{a^2 -\eta_2^2}$.
It is easy to see that the inner integral (over $|z_1| \leq r$)
is bounded by $|y_1 - x_1| \leq | x - y|$.
It follows that 

\noi
\begin{align}\label{Lipa}
| g_{a, a}(x) - g_{a, a}(y)| \leq  c_{d-1}\,a^{d-1} | x- y|.
\end{align}

\noi
where $c_{d-1}$ depends only on $d$. For general $a,  b>0$,
we write as in \cite[page 7]{Mai23} that
\begin{align*}
| g_{a, b}(x) - g_{a, b}(y)|
&\leq \int_{\R^d} | \ind_{D_b}(x-z) - \ind_{D_b}(y-z) |^2 dz \\
& = 2 ( g_{b,b}(0) - g_{b, b}(x-y)   ) \le\,c_{d-1}\, b^{d-1} 
| x - y| ,
\end{align*}

\noi
where the last step follows from \eqref{Lipa}.
See, e.g., \cite[Proposition 2.2]{Mai23} and references therein
for results on more general compact subsets 
(instead of just Euclidean balls).
Let us end this preparation with a few more observations. 
Note that $g_{a, b}(r) =g_{a, b}(z) = 0$ for 
$r = |z| \geq 2\max\{a, b\}$,
and by Rademacher's theorem, $g_{a, b}$ 
is almost everywhere differentiable.
Moreover, for almost every $\ell\in\R_+$,
 $g'_{a,b}(\ell) \leq 0$, since $g_{a, b}(r)$ 
 is decreasing in $r$.\footnote{More 
 precisely, $D_a$ and $D_b +z$ are moving further away 
 as $r=|z|$ gets bigger.}
 As a result, we get
 \begin{align} \label{Rad2}
   \begin{aligned}
 g_{a, b}(r) &= \int_r^{2\max\{a, b\}} \psi_{a, b}(\ell) d\ell \\
 \text{with} \,\,\, 0\le \psi_{a, b}(\ell)  & := - g'_{a, b}(\ell) \le\, c_{d-1} \, b^{d-1} 
 \,\,\,\text{for almost every $\ell\in\R_+$} \\
 {\rm and}\quad \psi_{a, b}(\ell) &=0 \quad
 \text{for $\ell < |b-a|$ or $\ell > a+b$,}
 \end{aligned}
  \end{align}

Next, we present the bulk of the proof.

\medskip

\noi  
$\bul$ {\bf Bulk of the proof.} Let us first write for $t_1 < t_2$:
\noi
\begin{align}
     & w_{t_2}^{-1}\,t_1^{-d} 
     \int_{|x| \leq a t_1}\int_{| y| \leq a t_2}K(x-y)\,dx\,dy  \notag  \\
  = &  w_{t_2}^{-1}\,t_1^{-d}
  \int_{\R^d}\,K(z)\,g_{at_1,at_2}(z)\,dz \label{bulk1} \\
  =& w_{t_2}^{-1} \big(\tfrac{t_2}{t_1}\big)^d
  \int_{| z| < 2at_2} K(z) g_{\frac{t_1}{t_2}a, a}(z/t_2) dz     \label{bulk2} \\
    =& w_{t_2}^{-1} \big(\tfrac{t_2}{t_1}\big)^d
  \int_{\R^d} K(z) \int_{| z| /t_2}^{2a} \psi_{\frac{t_1}{t_2}a, a}(\ell) d\ell dz     \label{bulk2b} \\
    = & w_{t_2}^{-1} \big(\tfrac{t_2}{t_1}\big)^d
  \int_0^{2a}\psi_{\frac{t_1}{t_2}a, a}(\ell) 
  \bigg( \int_{|z| \leq \ell t_2} K(z) dz \bigg) d\ell \notag \\
  = &  \big(\tfrac{t_2}{t_1}\big)^d
  \int_0^{2a} \psi_{\frac{t_1}{t_2}a, a}(\ell)  \frac{w_{\ell t_2}   }{w_{t_2}}d\ell  ,  \label{bulk3}
  \end{align}
  
  \noi
 where  we applied the change of variable $z = x-y$ in \eqref{bulk1},
  we used the elementary relation  
 $g_{at_1,at_2}(z) = t_2^d \, g_{a\tfrac{t_1}{t_2},a}(z/t_2)$
 with $g_{\frac{t_1}{t_2}a, a}(z/t_2) =0$ for $\|z\| > 2at_2$
  in \eqref{bulk2},
  and we utilized \eqref{Rad2} in \eqref{bulk2b},
  followed by an application of Fubini's theorem.

It follows from Lemma \ref{lemP}-(i) with    \eqref{tech3aa}
that the index $\rho$ of regular variation of $w$ is nonnegative. Moreover, as a consequence of $w_\infty > 0$, we can find 
some sufficiently large $X_0 > 0$ such that $w_t \geq \frac{1}{2} w_\infty > 0$
for any $t \geq X_0$. 
Thus, by Potter's bound in Lemma \ref{lemP} we have 
that for every $\delta>0$, $\exists$ $X=X_{K,\dl} \geq X_0$ such that
\begin{align} \label{wa1}
	\frac{w_{\ell\,t_2}}{w_{t_2}}
	\leq
	2 \max\{\ell^{\rho-\dl}, \ell^{\rho+\dl}\}
	\quad
	\text{ for $t_2, t_2 \ell \geq X$}.
\end{align}
\noi
On the other hand,   for $t_2 \ell \leq X$ and $t_2 \geq X \geq X_0$
with $t_2 > t_1$
we deduce from local integrability of $K$ that 

\noi
\begin{align} \label{wa2}
	\big| \frac{w_{\ell\,t_2}}{w_{t_2}} \big|
	\leq   \bigg(\int_{\|z\| \leq X} | K(z) | dz \bigg) 
    \frac{1}{ \frac{1}{2}w_{\infty}}
	\les      1\,.
\end{align}
Combining \eqref{wa1}, \eqref{wa2}, \eqref{Rad2}, and \eqref{bulk3}, we obtain
\noi
\begin{align} \label{bulk4}
\begin{aligned}
\bigg| w_{t_2}^{-1}\,t_1^{-d} 
     \int_{|x| \leq a t_1}\int_{| y| \leq a t_2}K(x-y)\,dx\,dy  \bigg|
&\les
 \big(\tfrac{t_2}{t_1}\big)^d \big(\tfrac{t_1}{t_2}\big) ^{d-1}  \int_0^{2a} 
  \big| \frac{w_{\ell t_2}   }{w_{t_2}} \big|    d\ell \\
  & \les 
  \frac{t_2}{t_1} .
  \end{aligned}
\end{align}

On the other hand, in view of \eqref{Rad2}, $ \psi_{\frac{t_1}{t_2}a, a}(\ell)  =0$  
for $\ell \notin [a( 1- \tfrac{t_1}{t_2} ),
a( 1+ \tfrac{t_1}{t_2} )   ]$, so that
for $\ell \in  [a( 1- \tfrac{t_1}{t_2} ),
a( 1+ \tfrac{t_1}{t_2} )   ]$, 
we continue with \eqref{wa1} to write
with $t_2 > t_1$ 

\noi
 \begin{align}  \notag 
  0<   \frac{w_{\ell\,t_2}}{w_{t_2}}
     \les (1- \tfrac{t_1}{t_2} )^{-\dl}
\quad
\text{ for $t_2, t_2 \ell \geq X \geq X_0$}.
   \end{align}

 \noi
 Thus, combining this bound with \eqref{wa2}, \eqref{bulk3} and \eqref{Rad2} we obtain
\begin{align}  \label{bulk5}
	\begin{aligned} 
		\bigg| w_{t_2}^{-1}\,t_1^{-d} 
		\int_{|x| \leq a t_1}\int_{| y| \leq a t_2}
        K(x-y)\,dx\,dy  \bigg|
		&\les 
		\big(\tfrac{t_2}{t_1}\big)^d    (1- \tfrac{t_1}{t_2} )^{-\dl}
		\int_0^{2a} \psi_{\frac{t_1}{t_2}a, a}(\ell)    d\ell \\
		&=  \big(\tfrac{t_2}{t_1}\big)^d    (1- \tfrac{t_1}{t_2} )^{-\dl}
		g_{\frac{t_1}{t_2}a, a}(0)  \\
		&=   (1- \tfrac{t_1}{t_2} )^{-\dl} \Vol(D_a).
	\end{aligned} 
\end{align}
 \noi
  Hence, the desired bound \eqref{tech3a} follows by a combination of \eqref{bulk4} and \eqref{bulk5} (in particular, when $|t_2/t_1|$ is small we have \eqref{bulk4}, when $|t_2/t_1|$ is big we have \eqref{bulk5}).  
\qedhere

\end{proof}

 \begin{lemma}\label{tech4} 
  Recall  from \eqref{def_HT}-\eqref{XIM0} the definitions of $h_t$
  and  $\xi_m$.
Suppose that
Condition \ref{cond5} holds
for some $\dl >0$ and 
Condition \ref{cond7} holds as well. 
Then, for $m > d/\dl$, we have  $   \xi^2_m(t)  \les  t^{-a}$
for some $a > 0$.
For instance, we can pick  
 $a = \min\{ \dl m - d, d, \tfrac{\dl}{2} \}$.

\end{lemma}
\begin{proof} 

    By definition of $\xi_m(t)$, we have 
  
  \noi
    \begin{align}\label{tech3bdd}
 \xi^2_m(t) = \frac{1}{\s_t^4} \sup_{\substack{k_1, k_2\ge 1 \\
 \,k_1+k_2\geq m}}\,
 \int_{(tD)^4} \cC^{k_1}(x-y) \cC^{k_1}(z-w)
 \cC^{k_2}(x-z)\cC^{k_2}(y-w)\,dxdydzdw . 
    \end{align}
    
\noi

\noi
Note that due to  $|\cC | \leq 1$
and 
$| \cC^{r}(z-w)
 \cC^{m-r}(x-z)| \leq  
| \cC^{m}(z-w)| + 
 |\cC^{m}(x-z)|$,
 we can  bound
\eqref{tech3bdd}  by

  \noi
    \begin{align}        \label{tech4b} 
    \begin{aligned}  
 \xi^2_m(t) 
 &\leq \frac{1}{\s_t^4} \sup_{ 1\leq r \leq m-1}\,
 \int_{(tD)^4} \big| \cC^{r}(x-y) \cC^{r}(z-w)
 \cC^{m-r}(x-z)\cC^{m-r}(y-w ) \big| \,dxdydzdw      \\
 &\leq  \frac{1}{\s_t^4} \sup_{ 1\leq r \leq m-1}\,
 \int_{(tD)^4} \big| \cC^{r}(x-y) \cC^{m}(z-w) \cC^{m-r}(y-w ) \big| 
           \,dxdydzdw       \\
 &\quad
 +  \frac{1}{\s_t^4} \sup_{ 1\leq r \leq m-1}\,
 \int_{(tD)^4} \big| \cC^{r}(x-y)   \cC^{m}(x-z)\cC^{m-r}(y-w ) \big| 
         \,dxdydzdw     \\
 &\leq 2 \frac{\Vol(tD)}{\s_t^4}   
 \bigg( \int_{|z| \leq 2t a_0} | \cC^m(z) | dz \bigg)
\qquad     \qquad (\text{with $a_0 = \max\{ |x-y|: x\in D\}$}) \\
 &\qquad \times
 \sup_{ 1\leq r \leq m-1}\,
 \bigg( \int_{|z|\leq 2ta_0 } \big| \cC^{r}(z) \big| dz \bigg)
 \bigg( \int_{|z|\leq 2ta_0 } \big| \cC^{m-r}(z) \big| dz \bigg),
 \end{aligned}
    \end{align}

\noi
where the last step is obtained by performing integration 
in the order of $dz, dw, dy$, and $dx$.

\medskip

Since $|\cC(z)| \les |z|^{-\dl}$ with $\dl > d/m$,
we have $\cC\in L^m(\R^d)$; while 
we read from Lemma \ref{tech2} that

\noi
\begin{align}      \label{tech4a}
\sigma_t^2 
\ges t^d.
\end{align}

\noi
Next, we estimate the integral  
$\int_{|z|\leq 2t a_0 } \big| \cC^{r}(z) \big| dz$
for $r=1, ... , R-1$ and
for $t > 1$:  we deduce from $|\cC(x)| \les \min\{1,  |x|^{-\dl}\}$
 that

\noi
\begin{align*}    
\begin{aligned}
\int_{|z|\leq 2t a_0 } \big| \cC^{r}(z) \big| dz
&\les 1 + \int_{ a_0 \leq |z|\leq 2t a_0 } \big| \cC^{r}(z) \big| dz 
\les 1 +  \int_{ a_0 \leq |z|\leq 2t a_0 }  |z|^{-\dl r} dz \\
&\les t^{d-\dl r } \ind_{\{ d > \dl r\}} + \ind_{\{d < \dl r\}} 
+   \ind_{\{ d = \dl r\}}   \log t ,
\end{aligned}
\end{align*}

\noi
from which, together with \eqref{tech4a} and \eqref{tech4b},
we can obtain 
\[
\xi^2_m(t) \les t^{-a}
\]
with $a = \min\{ \dl m - d, d, \tfrac{\dl}{2} \} > 0$,
where the number $\tfrac{\dl}{2} $ comes 
from the rough estimate $t^{-\dl r} \log t \les t^{-\dl /2}$
for $1 \leq r \leq m-1$ and $t > 1$.  
Hence, the proof is completed now. 
\qedhere

\end{proof}

\begin{lemma}
    \label{tech5}
  Suppose that $\bfB=( B_x: x\in\R^d)$ is a centered, stationary Gaussian
    random field  with unit variance and radial covariance function
    \begin{align*}
    \cC(x)=c(|x|):=|x|^{-\beta}\,L(|x|) ,
    \end{align*}

    \noi
    where $\be\in (0,d/R)$ and $L:\R_+\to\R$ is slowly varying at infinity. 
    Suppose $\varphi:\R\to\R$ is measurable 
    with $\varphi(B_0)\in L^2(\Omega)$
    such that $\varphi - \E[ \varphi(B_0) ]$ has Hermite rank $R\geq 1$.
    Then, the $R$-th chaotic  component of 
    \[
    Y_t := \int_{t D}  \big( \varphi(B_x) - \E[ \varphi(B_0) ] \big)dx
    \] 
    is dominant as  $t\to+\infty$.
   Moreover,
   $Y_t/\sqrt{\Var(Y_t)}$ does not converge in law to a standard normal
   whenever $R\geq 2$.

\end{lemma}
\begin{proof}
	First of all, since $\cC\in L^M\setminus L^R$ 
    for some $M$ and $\cC^R(x)\ge0$ for $x\in[X,\infty)$, 
    $X>0$ large enough (indeed $\cC$ is regularly varying), 
    by Theorem \ref{recap}-(ii) the $R$-th chaotic component is dominating, i.e., $\s_t^2\sim \s_{t,R}^2$ 
    and then  $Y_t$ is asymptotically $L^2$ equivalent
    to its $R$-th chaotic component. 
    Moreover, by Lemma \ref{lemP}-(iii), 
    since $\cC^R(x)=L^R(|x|)|x|^{-R\beta}$ is regularly varying with index $-R\beta\in(-d,0)$, we have (see, e.g., \cite[(9),(10) and Remark 2.7]{Mai23})
	\[
	\s_{t,R}^2\,\asymp\,t^{2d}c^R(t)\,.
	\]

\noi
Finally,  let us prove that  $Y_t /\sqrt{\Var(Y_t)}$ can not have a Gaussian limit when  $R \geq 2$.
Using the asymptotic $L^2$-equivalence, it suffices to show 
that the $R$-th chaotic component does not admit Gaussian fluctuations
as $t\to+\infty$. It follows from Nualart-Peccati's fourth moment theorem
(Theorem \ref{FMT_NP})
that it is enough to show that for some $r\in\{1, ..., R-1\}$,

\noi
\begin{center}
	$\frac{h_t\left(r,R-r\right)}{\s_t^4}$ does not converge to zero as $t\to\infty$;
\end{center}
see also \eqref{GTQ_r2}.
Now we can deduce from  $\s_t^2 \sim \s_{t, R}^2 \asymp  t^{2d} c^R(t)$
that

\noi
    \begin{align*}
	\frac{h_t(r,R-r )}{\s_t^4}
	&\asymp 
	\int_{D^4}
	\frac{c^r(t|x-y|)}{c^r(t)} 
	\frac{c^r(t|z-w|)}{c^r(t)}
	\frac{c^{R-r}(t|x-z|)}{c^{R-r}(t)}
	\frac{c^{R-r}(t|y-w|)}{c^{R-r}(t)}\,dxdydzdw \\
	&=\int_{D^4\setminus \mathcal{D}_\epsilon}
	\frac{c^r(t|x-y|)}{c^r(t)} 
	\frac{c^r(t|z-w|)}{c^r(t)}
	\frac{c^{R-r}(t|x-z|)}{c^{R-r}(t)}
	\frac{c^{R-r}(t|y-w|)}{c^{R-r}(t)}\,dxdydzdw\\
	&\quad +\int_{\mathcal{D}_\epsilon}
	\frac{c^r(t|x-y|)}{c^r(t)} 
	\frac{c^r(t|z-w|)}{c^r(t)}
	\frac{c^{R-r}(t|x-z|)}{c^{R-r}(t)}
	\frac{c^{R-r}(t|y-w|)}{c^{R-r}(t)}\,dxdydzdw,
\end{align*}

\noi
where $\mathcal{D}_\epsilon = \{ (x, y, z, w)\in D^4: |x-y|, |z-w|, |x-z|, |y-w| \ge \epsilon\}$, 
where $\epsilon>0$ is chosen small enough 
so that ${\rm Vol}(\mathcal{D}_{\epsilon})>0$.\footnote{This is possible.
 Indeed, ${\rm Vol}(\mathcal{D}_{0})={\rm Vol}(D)^4>0$ (since ${\rm Vol}(D)>0$ by assumption) 
 and it is easy to see that $\epsilon\mapsto {\rm Vol}(\mathcal{D}_{\epsilon})$ is continuous in $0$.}
Then, using Potter's bound from Lemma \ref{lemP}, we can write, 
with $A=2$ and $\dl \in( 0, \be)$ small enough (so that $\beta+\dl<d$),
that

 \noi
\begin{align}
	&\label{lowerpotter}
	\frac{c(t|x-y|)}{c(t)} \ge \frac{1}{2} \min\{ |x-y|^{- \be + \dl},  |x-y|^{- \be - \dl} \}
	\ges    |x-y|^{-\beta+\dl} \,\,\,,\\
	&\label{upperpotter}
	\frac{c(t|x-y|)}{c(t)} \le 2 \max\{ |x-y|^{- \be + \dl},  |x-y|^{- \be - \dl} \}
	  \les  |x-y|^{-\beta-\dl}\,\,\,,
\end{align}
for $X/t  \le |x-y| \leq {\rm diam}(D)=\max\{ |x-y|\,:\,x,y\in D\}<\infty$, 
where $X=X_{c,\delta,2}>0$ is given by Potter's bound.
Thus, we can deduce from \eqref{lowerpotter} with $\epsilon=X/t$
that  as $t\to\infty$, 

\noi
\begin{align*}
&\int_{\mathcal{D}_{X/t}}
	\frac{c^r(t|x-y|)}{c^r(t)} 
	\frac{c^r(t|z-w|)}{c^r(t)}
	\frac{c^{R-r}(t|x-z|)}{c^{R-r}(t)}
	\frac{c^{R-r}(t|y-w|)}{c^{R-r}(t)}\,dxdydzdw
	\\
	&\ges \int_{\mathcal{D}_{X/t}}
	|x-y|^{-r(\beta-\dl)}|z-w|^{-r(\beta-\dl)}|x-z|^{-(q-r)(\beta-\dl)}|y-w|^{-(q-r)(\beta-\dl)}dxdydzdw\\
	&\rightarrow \int_{D^4}
	|x-y|^{-r(\beta-\dl)}|z-w|^{-r(\beta-\dl)}|x-z|^{-(q-r)(\beta-\dl)}|y-w|^{-(q-r)(\beta-\dl)}dxdydzdw
	>0,
\end{align*}
where the last convergence result follows from monotone convergence theorem. 
To conclude the proof, we need to show that as $t\to\infty$,

\noi
\begin{align}
	\label{finalstep}
	\int_{D^4\setminus \mathcal{D}_{X/t}}
	\frac{c^r(t|x-y|)}{c^r(t)} 
	\frac{c^r(t|z-w|)}{c^r(t)}
	\frac{c^{R-r}(t|x-z|)}{c^{R-r}(t)}
	\frac{c^{R-r}(t|y-w|)}{c^{R-r}(t)}\,dxdydzdw\rightarrow 0\,.
\end{align}

\noi
To do this, we further decompose the above integration domain
$
D^4\setminus \mathcal{D}_{X/t}
$
and it suffices to estimate the following integrals:

\noi
\begin{align}
	&\int_{\substack{(x,y,z,w)\in D^4\\|x-y|,|z-w| < X/t \\ |x-z|,|y-w|< X/t}}
	\frac{c^r(t|x-y|)}{c^r(t)} 
	\frac{c^r(t|z-w|)}{c^r(t)}
	\frac{c^{R-r}(t|x-z|)}{c^{R-r}(t)}
	\frac{c^{R-r}(t|y-w|)}{c^{R-r}(t)}\,dxdydzdw \label{term1}\\
&\int_{\substack{(x,y,z,w)\in D^4\\|x-y|,|z-w|,|x-z|< X/t\\
			|y-w|\ge X/t}}
	           \qquad \& \qquad
	           \int_{\substack{(x,y,z,w)\in D^4\\|x-y|,|z-w|< X/t\\
			|x-z|,|y-w|\ge X/t}}
	                \qquad \& \qquad
	                \int_{\substack{(x,y,z,w)\in D^4\\|x-y|,|x-z|< X/t\\
			|z-w|,|y-w|\ge X/t}}
	      \label{term4}\\
 &\quad\qquad \qquad\qquad\qquad\qquad  \& \qquad  \int_{\substack{(x,y,z,w)\in D^4\\|x-y|< X/t\\
			|z-w|,|x-z|,|y-w|\ge X/t}}
	 \label{term5}\,.
\end{align}

Let us first consider the last integral  \eqref{term5}.
Using   $|c(t|x-y|)|\le1$ and   \eqref{upperpotter} with
$0< \dl < \frac{d}{R} - \be$,
we have  

\noi
\begin{align} \label{forYoung}
\text{$z\in\R^d \mapsto f(z):=|z|^{-(\beta+\dl)}\ind_{D-D}(z)$ in $L^r(\R^d)$ for $1\leq r\leq R$},
\end{align}

\noi
and thus,  we obtain 

\noi
\begin{align}
	\eqref{term5}
	&\le \frac{1}{c^r(t)}\int_{\substack{(x,y,z,w)\in D^4\\|x-y|< X/t}}
	|z-w|^{-r(\beta+\dl)}|x-z|^{-(R-r)(\beta+\dl)}|y-w|^{-(R-r)(\beta+\dl)}\,dxdydzdw \nonumber\\
	& = \frac{1}{c^r(t)}\int_{\substack{(x,y,z,w)\in D^4\\|x-y|< X/t}}
	f^r(z-w)f^{R-r}(x-z)f^{R-r}(y-w)\,dxdydzdw  \nonumber\\
	&\leq   \frac{\Vol(D)}{c^r(t)}\int_{\R^{3d}}f^r(v)f^{R-r}(a+u - v) f^{R-r}(u)  \ind_{\{ |a| < X/t \}}dadudv  \nonumber  \\
	&=  \frac{\Vol(D)}{c^r(t)}\int_{D-D}  \big( f^r \ast f^{R-r} \ast  f^{R-r} \big)(a) \ind_{\{ |a| < X/t \}}da,
    \label{changeofvariable}
\end{align}

\noi
where   $\ast$ denotes  the convolution in $ \R^d$.
Therefore,   using H\"older's and  Young's convolution inequalities
with \eqref{forYoung},
we have

\noi
\begin{align*}
\eqref{term5} 
&\les  \frac{1}{c^r(t)\,t^d}  \big\| f^r\ast f^{R-r}\ast f^{R-r} \big\|_{L^\infty(\R^d)} 
\leq \frac{1}{c^r(t)\,t^d} \| f^{R-r}  \|_{L^{\frac{R}{R-r}}(\R^d)}^2  
             \cdot  \| f^r  \|_{L^{\frac{R}{2r}}(\R^d)}    \\
&=\frac{1}{c^r(t)\,t^d}    \| f  \|_{L^{R}(\R^d)}^{2(R-r)}  
             \cdot  \| f  \|_{L^{\frac{R}{2}}(\R^d)}^r 
             \les \frac{1}{c^r(t)\,t^d}\,\xrightarrow{t\to\infty} 0\,,
\end{align*}

\noi
where the last 
step follows by Lemma \ref{lemP}-(i) with $c^r(t) = |t|^{-\be r} L^r(t) \ges  |t|^{-\be r - \eps_0}$
for $\eps_0 > 0$ small enough.

Regarding the other  integrals in \eqref{term1}-\eqref{term4}, 
we note that at least two among the four numbers $|x-y|,|z-w|,|x-z|, |z-w|$ 
are less than $X/t$ in the description of the integration domains. 
Thus, we can argue as in \eqref{changeofvariable} to bound 
these integrals by 
\begin{align*}
\les \frac{1}{c^{2R}(t)}\,\int_{\substack{(a,u,v)\in (D-D)^3\\|a|<X/t\,,\,|u|<X/t}}da du dv\les \frac{1}{t^{2d}\,c^{2R}(t)}\rightarrow 0\,,
\end{align*}
where the last limit follows again by Lemma \ref{lemP}-(i). 
Hence, \eqref{finalstep} is verified and we can conclude    our proof now.
\qedhere

\end{proof}

%
%
%
%
%

\end{document}